\documentclass{amsart}

\usepackage[colorlinks=true, urlcolor=blue, citecolor=blue, linkcolor=blue, hypertexnames=false]{hyperref}
\usepackage{amssymb}
\usepackage{amsfonts}
\usepackage{amsmath}
\usepackage{amsthm}
\usepackage{mathrsfs}
\usepackage{dsfont}
\usepackage{mathtools}
\usepackage{tikz}
\DeclareFontFamily{U}{mathx}{\hyphenchar\font45}
\DeclareFontShape{U}{mathx}{m}{n}{
      <5> <6> <7> <8> <9> <10>
      <10.95> <12> <14.4> <17.28> <20.74> <24.88>
      mathx10
      }{}
\DeclareSymbolFont{mathx}{U}{mathx}{m}{n}
\DeclareFontSubstitution{U}{mathx}{m}{n}
\DeclareMathAccent{\widecheck}{0}{mathx}{"71}

\usepackage{aliascnt}

\newtheorem{thm}{Theorem}[section]

\newaliascnt{lem}{thm}
\newtheorem{lem}[lem]{Lemma}
\aliascntresetthe{lem}

\newaliascnt{cnj}{thm}  

\aliascntresetthe{cnj}

\newaliascnt{qst}{thm}  

\aliascntresetthe{qst}

\newaliascnt{prp}{thm}  
\newtheorem{prp}[prp]{Proposition}
\aliascntresetthe{prp}

\newaliascnt{cor}{thm}  
\newtheorem{cor}[cor]{Corollary}
\aliascntresetthe{cor}

\theoremstyle{definition}

\newaliascnt{dfn}{thm}  
\newtheorem{dfn}[dfn]{Definition}
\aliascntresetthe{dfn}

\newaliascnt{xpl}{thm}
\newtheorem{xpl}[xpl]{Example}
\aliascntresetthe{xpl}

\numberwithin{equation}{section}

\author{Tristan Bice}
\address{
Federal University of Santa Catarina\\
Florianopolis\\
Brazil
}
\email{Tristan.Bice@gmail.com}
\thanks{This research has been supported by a CAPES (Brazil) postdoctoral fellowship through the program ''Science without borders'', PVE project 085/2012.}

\begin{document}

\title{Annihilators and Type Decomposition in C*-Algebras}

\begin{abstract}
We initiate the study of annihilators in C*-algebras, showing that they are, in many ways, the best C*-algebra analogs of projections in von Neumann algebras.  Using them, we obtain a type decomposition for arbitrary C*-algebras that is symmetric and completely consistent with the classical von Neumann algebra type decomposition.  We also show that annihilators admit a very simple notion of equivalence that is again completely consistent with the notion of Murray-von Neumann equivalence in von Neumann algebras, sharing many of its general order theoretic properties.
\end{abstract}

\maketitle

\section{Introduction}

\subsection{Motivation}\label{Motivation}

It is no exaggeration to say that type decomposition and Murray-von Neumann equivalence of projections are absolutely fundamental to the theory of von Neumann algebras.  These concepts have been extended to other operators in a couple of different ways (see below), and this has certainly led to some interesting theory.  However, only very specific aspects of the von Neumann algebra theory have been generalized to C*-algebras in this way.  No doubt it was accepted that this is an unavoidable fact of life, that much of the von Neumann algebra theory simply can not be applied in any general way to the much larger class of C*-algebras.  In the present paper we show this to be false.  By choosing generalizations appropriately, using annihilators rather than projections, a surprising amount of the basic von Neumann algebra theory does indeed extend fully to C*-algebras.

Let us first go back to the beginning and consider a von Neumann algebra $A$.  Here, type decomposition is obtained by utilizing the projections $\mathcal{P}(A)$, and crucial to this is their order structure, specifically the fact $\mathcal{P}(A)$ is a complete orthomodular lattice.  Also crucial is the fact that projections exist in abundance in an arbitrary von Neumann algebra.  In C*-algebras, on the other hand, there may be no non-zero projections whatsoever, even in the simple case.  And even when they are plentiful, they may fail to form a lattice.  Consequently, to prove results for C*-algebras that generalize or are analogous to classical von Neumann algebra results, like those relating to type decomposition, involves finding an appropriate replacement for projections on which an appropriate analog of Murray-von Neumann equivalence can be defined.

One example of this can be found in \cite{CuntzPedersen1979}, where projections in a C*-algebra $A$ are replaced with positive elements and $a,b\in A_+$ are said to be equivalent if there exists $(x_n)\subseteq A$ such that $a=\sum x_nx_n^*$ and $b=\sum x_n^*x_n$, where the sums are norm convergent.  The close relationship between traces and Murray-von Neumann equivalence classes in von Neumann algebras generalizes to positive operators with this notion of equivalence, as demonstrated in \cite{CuntzPedersen1979}.  An analogous classification and even a decomposition of an arbitrary C*-algebra into types I, II and III is also obtained in \cite{CuntzPedersen1979} Proposition 4.13.  However, this decomposition is neither symmetric (the type III part can only be found in the quotient w.r.t. the type I part, not the other way around, and likewise the type II part is only obtained at the end as a quotient of quotients) nor consistent with the classical von Neumann algebra type classification (for example, $B(l^2)$ is a type I von Neumann algebra but not a type I C*-algebra\footnote{Here, and here only, we are using the terminology of \cite{CuntzPedersen1979}, where a C*-algebra is said to be of type I if all its representations are type I.  Other standard terms for such C*-algebras are `GCR' and `postliminal'.  Throughout the rest of this article we will use the term `postliminal'.} (see \cite{Pedersen1979} 6.1.2)).  Furthermore, $A_+$ will not be a lattice unless $A$ is commutative and other natural order theoretic properties fail for $A_+$ in general, for example the sum of two finite elements of $A_+$ may be infinite (see \cite{CuntzPedersen1979} Corollary 7.10).

Slight variants of the above can be obtained by changing the sums in the definition of equivalence, e.g. by specifying that they must be finite, or allowing them to represent supremums that might not necessarily converge in norm.  Another quite different example is given in \cite{Cuntz1977}, where projections are replaced with arbitrary operators and $a,b\in A$ are said to be equivalent if $a\lessapprox b$ and $b\lessapprox a$, where $a\lessapprox b$ means $a=cbd$ for some $c,d\in A$.  Again, this leads to a natural notion of a finite and factorial (simple) C*-algebra.  The natural norm closed variant of this (i.e. $c_nbd_n\rightarrow a$, for some $(c_n),(d_n)\subseteq A$) has also received considerable attention and has a strong relation to the dimension functions on $A$ (see \cite{BlackadarHandelman1982}).  Although again the order structures so obtained are generally less tractable and there is a limit to how far the analogy to projections in von Neumann algebras can be pushed (there does not appear to even be a canonical decomposition into types in this case, for example).

However, there is another quite different, but very natural, candidate to replace projections with in an arbitrary C*-algebra $A$, one that seems to have been largely overlooked.  Namely, we can use the left (or right) annihilator ideals, i.e. those of the form $\{a\in A:\forall b\in B(ab=0)\}$ for some $B\subseteq A$.  Equivalently, we can use the hereditary C*-subalgebras corresponding to these left ideals (see \cite{Effros1963} Theorem 2.4 or \cite{Pedersen1979} Theorem 1.5.2), which we refer to simply as \emph{annihilators}.  Indeed, the map $p\mapsto pAp$ is an order isomorphism from projections (with their canonical order $p\leq q\Leftrightarrow pq=p$) to a subset of annihilators (ordered by inclusion).  This map is even surjective whenever $A$ is a von Neumann algebra (see \autoref{annisep} (\ref{annisep1})) or, more generally, an AW*-algebra (see \cite{Berberian1972}), thus yielding a precise correspondence between projections and annihilators in this case.\footnote{This correspondence is well known, as is the fundmental importance of studying projections in AW*-algebras.  Despite this, however, there does not seem to be any indication, either in \cite{Berberian1972} or elsewhere in the operator algebra literature (except in \cite{Arzikulov2013}), that it was ever thought the annihilators themselves might be of interest in more general contexts.} 

Unlike projections, though, the annihilators still exist in abundance in an arbitrary C*-algebra, thanks to the continuous functional calculus (see the discussion preceeding \autoref{xab}).  And we can see from the outset that they have greater potential to fulfil the role of projections in a von Neumann algebra, as they also always form a complete lattice.  Furthermore, there is a natural orthocomplementation on annihilators, something we do not have for arbitrary hereditary C*-subalgebras (the collection of all hereditary C*-subalgebras may not even be complemented).  It turns out that this orthocomplementation is not always orthomodular (see \autoref{nonorthoxpl}), but we can prove a close approximation to orthomodularity (see \S\ref{SvsO}) which allows much of the von Neumann algebra theory to be generalized fully to arbitrary C*-algebras.  In particular, we can obtain a type decomposition that is symmetric and completely consistent with the classical type decomposition of von Neumann algebras, together with a natural analog of Murray-von Neumann equivalence that is again completely consistent with the classical notion.  The type decomposition itself can actually be obtained in a very general order theoretic context, and in the specific case of annihilators in a C*-algebra, the definitions in \eqref{pI}, \eqref{pII}, \eqref{pIII}, and \eqref{pIV} yield the following result.
\begin{thm}
For any C*-algebra $A$ we have orthogonal annihilator ideals $A_\mathrm{I}$, $A_\mathrm{II}$, $A_\mathrm{III}$ and $A_\mathrm{IV}$ such that \[A_\mathrm{I}\oplus A_\mathrm{II}\oplus A_\mathrm{III}\oplus A_\mathrm{IV}\] is an essential ideal in $A$.
\end{thm}

When $A$ is a von Neumann algebra, $A_\mathrm{I}$, $A_\mathrm{II}$ and $A_\mathrm{III}$ are indeed the usual type I, II and III parts in the classical von Neumann algebra decomposition (see the comments after \eqref{pIV}).  As every von Neumann algebra is an AW*-algebra, i.e. every annihilator ideal is of the form $pA$ for some central $p\in\mathcal{P}(A)$, a finite sum of annihilator ideals is again an annihilator ideal, coming from the sum of the corresponding projections.  As the only essential annihilator ideal is the entire algebra itself, we have $A=A_\mathrm{I}\oplus A_\mathrm{II}\oplus A_\mathrm{III}$.  As $\mathcal{P}(A)$ is orthomodular, the extra type IV part is $\{0\}$ here, although we do not know if the same is true for annihilators in C*-algebras.  Indeed, this paper puts us in much the same position as von Neumann himself was in at the early stages of his investigation into von Neumann algebras (see \cite{vonNemann1930} and \cite{MurrayvonNemann1936}).  Namely, we can decompose an arbitrary C*-algebra into various types but do not know if all these potential types are actually realized by some C*-algebra, most notably we do not know if there are any type IV C*-algebras (just as von Neumann did not verify the existence of type III von Neumann algebras until later in \cite{vonNemann1940}).

It should now be clear that the annihilators in a C*-algebra are of fundamental importance.  The C*-algebra theory required for the initial investigation presented here is not even particularly great, and this paper should be accessible to anyone familiar with the material in the relevant parts of the first few chapters of \cite{Pedersen1979}.  Given this, it is very surprising that more papers analyzing the annihilator structure of C*-algebras have not been written before and, in our opinion, such an analysis is well overdue.  We know of only one such article, namely \cite{Arzikulov2013}, where a few ideas similar to those presented here are also discussed.  However, there are simple counterexamples to \cite{Arzikulov2013} Lemma 16.1 (see the discussion at the end of \S\ref{NCT}) which, unfortunately, is used repeatedly in \cite{Arzikulov2013} and thus puts the results there into question.  The key point is that care must be taken to distinguish arbitrary open projections from those that are also topologically regular, which amounts to distinguishing aribitrary hereditary C*-subalgebras from annihilators.  And there is no mention of an analog of Murray-von Neumann equivalence for annihilators in \cite{Arzikulov2013}, although analogs of Murray-von Neumann equivalence for arbitrary hereditary C*-subalgebras have been considered before (see \cite{OrtegaRordamThiel2012} and \cite{PeligradZsido2000}).  The difference between annihilators and arbitrary hereditary C*-subalgebras may at first seem slight, but it turns out that it is the annihilators that are more amenable to attack by a fortuitous combination of non-commutative topology, order theory and algebra, as we proceed to demonstrate in this paper.

It is really the order theory that is central here.  Kaplansky initiated a program to isolate the algebraic structure of von Neumann algebras, and this is what allowed the von Neumann algebra theory to be generalized to AW*-algebras (and, more generally, Baer *-rings).  All we are really doing is taking this a step further, isolating the order structure of projections in AW*-algebras in such a way that the theory can be generalized to annihilators in C*-algebras.  Von Neumann himself isolated the order structure of projections in finite von Neumann algebras, those in which the projection lattice is modular, resulting in the elegant theory of continuous geometries (see \cite{vonNeumann1960}).  Even when the projection lattice of a von Neumann algebra is not modular, it is still orthomodular, and this inspired the development of a large body of work on orthomodular lattices (see \cite{Kalmbach1983}).  Type decompositions have also been obtained for certain orthomodular lattices, namely the dimension lattices of \cite{Loomis1955}, and these have been successively generalized to various other contexts, like the espaliers in \cite{GoodearlWehrung2005} and the effect algebras in \cite{FoulisPulmannova2013}.  However, these other contexts still assume something equivalent to orthomodularity in the ortholattice case, like unique orthogonal complements.\footnote{This is no longer true for the pre-effect algebras in \cite{ChajdaKuhr2012} and presumably type decomposition could also be done for the subclass of pre-effect algebras corresponding to separative complete ortholattices, and probably some more general subclass (centrality is discussed in \cite{ChajdaKuhr2012}, although they do not quite go as far as doing type decomposition).  But we decided to stick to ortholattices rather than pre-effect algebras, as these are certainly sufficient for analyzing the annihliators in a C*-algebra, and probably also more familiar to operator algebraists.}

We diverge from this previous work with the simple observation that separativity, a significantly weaker assumption than orthomodularity, is sufficient for much of the development of the theory.  And this is most fortunate, for we can only verify that the annihilators in an arbitrary C*-algebra are separative (although in strong sense quite close to orthomodularity \textendash\, see \autoref{epsep} and \autoref{nonorthoxpl}).  Also, we work with what we call type relations, which are again more general than the dimension relations in dimension lattices (e.g. they need not satisfy finite (orthogonal) divisibility).  Again, we see that these are sufficient for much of the theory to be developed which, yet again, is fortunate because we can only verify these weaker properties for what we believe to be the natural equivalence relation on annihilators generalizing Murray-von Neumann equivalence.  We also make a number of other order theoretic observations and generalizations of our own that do not seem to have appeared elsewhere in the literature, even in the orthomodular case.  Really, you could see this paper as bringing the theory of lattices and operator algebras back together after over half a century of divergence from von Neumann's seminal work in both fields, namely in continuous geometry and von Neumann algebras.

\subsection{Outline}

We start off in \S\ref{OrderTheory} by developing the theory of type decomposition in an abstract order theoretic context general enough to be applied later to annihilators in a C*-algebra.  We take \cite{Kalmbach1983} as our primary reference, although we have to do things in greater generality as we are concerned with ortholattices that may not be orthomodular.  In particular, we have to be careful to distinguish $[p]$ from $[p]_p$ (see the discussion following \autoref{[p]_q}).  As such, this section should be of independent interest to lattice theorists, although many of the new results are relatively straightforward generalizations of known results for orthomodular lattices.  As mentioned above, the key observation here is that separativity, rather than orthomodularity, is sufficient to prove \autoref{sepcomportho}.  

In \S\ref{AvsP}, we start by gathering together some relevant facts about projections and annihilators, and the relationship between them.  In \S\ref{annsec} we set  the stage for our investigation of the annihilators, defining them as the orthocompletion of a C*-algebra $A$ with respect to a certain preorthogonality relation.  Next we introduce some standard non-commutative topological terminology in \S\ref{NCT} and show in \autoref{annproj} that the annihilators correspond precisely to the topologically regular open projections.  Then we make some important observations about spectral projections and projections in general in \S\S\ref{specsec} and \ref{psec}.

It is in \S\ref{SvsO} that we develop the theory needed to prove that the annihilators satisfy the all important property known as separativity.  In fact, \autoref{epsep} says that the annihilators satisfy a strong form of separativity which is very close to orthomodularity.  On the way, we also strengthen a result from \cite{AkemannEilers2002}, showing that non-regular open dense projections must, in fact, be as non-regular as possible (see \autoref{0sep} and the discussion at the start of \S\ref{SvsO}).

With separativity out of the way, we are free to apply the theory in \S\ref{OrderTheory} and investigate the interplay between the algebraic structure of a C*-algebra $A$ and the order structure of its annihilators $[A]^\perp$.  In \S\ref{annideals}, we see that the central annihilators are precisely the annihilator ideals and that the annihilators always have the relative centre property.  Next, in \S\ref{Equivalence}, we define and investigate what we believe to be the natural analog of the fundamental notion of Murray-von Neumann equivalence.  We then move on to discuss the abelian annihilators in \S\ref{AA}, starting with \autoref{commuteBoolean} which says that $A$ is commutative precisely when $[A]^\perp$ is a Boolean algebra.  We then extend the results of \S\ref{AA} to homogeneous annihilators in \S\ref{HomogeneousAnnihilators}, and show that the annihilator notion of homogeneity is closely related to the more classical representation theoretic notion.

In \S\ref{MVF} we investigate C*-algebras of continuous functions from topological spaces $X$ to finite rank matrices $M_n$.  More specifically we show how to represent hereditary C*-subalgebras/open projections of such C*-algebras by lower semicontinuous projection valued functions on $X$.  The moral of the story here is that the theory of annihilators turns out to be the theory of these projection functions modulo nowhere dense subsets of $X$.  Lastly, in \S\ref{Examples}, we give a number of examples illustrating the subtle distinction between various order theoretic and algebraic notions.

\subsection{Acknowledgements}
The author would like to thank Charles Akemann and Dave Penneys for many helpful comments on earlier versions of this paper, as well as Vladimir Pestov, for giving the author the opportunity to pursue the research that lead to this paper.

\section{Order Theory}\label{OrderTheory}

\subsection{Basic Definitions}\label{BasicDefinitions}

\begin{dfn}[Relation Terminology]
A relation $R$ on a set $S$ is
\begin{itemize}
\item \emph{reflexive} if $sRs$, for all $s\in S$.
\item \emph{transitive} if $sRt$ and $tRu\Rightarrow sRu$, for all $s,t,u\in S$.
\item \emph{symmetric} if $sRt\Leftrightarrow tRs$, for all $s,t\in S$.
\item \emph{antisymmetric} if $sRt$ and $tRs\Rightarrow s=t$, for all $s,t\in S$.
\item \emph{annihilating} if $sRs\Rightarrow\forall t\in S(sRt)$, for all $s\in S$.
\item a \emph{preorder} if $R$ is reflexive and transitive.
\item a \emph{partial order} if $R$ is an antisymmetric preorder.
\item an \emph{equivalence relation} if $R$ is a symmetric preorder.
\item a \emph{preorthogonality relation} if $R$ is symmetric and annihilating.
\item an \emph{orthogonality relation} if $R$ is a preorthogonality relation and, for $s,t\in S$,
\begin{equation}\label{orthorel}
\forall u\in S(s\perp u\Leftrightarrow t\perp u)\quad\Leftrightarrow\quad s=t.
\end{equation}
\end{itemize}
\end{dfn}

\begin{dfn}[Partial Order Terminology]
Let $\mathbb{P}$ be a partial order.  We call $\mathbb{P}$ a \emph{lattice} if every pair $p,q\in\mathbb{P}$ has a supremum (least upper bound) and infimum (greatest lower bound), denoted by $p\vee q$ and $p\wedge q$ respectively.  A lattice $\mathbb{P}$ is \emph{complete} if every $S\subseteq\mathbb{P}$ has a supremum and infimum, denoted by $\bigvee S$ and $\bigwedge S$ respectively.  If $\mathbb{P}$ has a greatest element $\mathbf{1}$ and least element $\mathbf{0}$ then $p$ and $q$ are \emph{complementary} if $p\vee q=\mathbf{1}$ and $p\wedge q=\mathbf{0}$.  If every element of $\mathbb{P}$ has a complement then $\mathbb{P}$ is \emph{complemented}.  For a preorder $\mathbb{P}$, we write $p<q$ to mean $p\leq q$ but $q\nleq p$, and we call $S\subseteq\mathbb{P}$ \emph{order-dense} in $\mathbb{P}$ if
\begin{equation}\label{densedef}
\forall p\in\mathbb{P}(p>\mathbf{0}\Rightarrow \exists s\in S(\mathbf{0}<s\leq p)),
\end{equation}
and we call $S$ \emph{join-dense} in $\mathbb{P}$ if, for all $p\in\mathbb{P}$ (with $p>\mathbf{0}$),
\begin{equation}\label{jd}
p=\bigvee\{q\in S:q\leq p\}.
\end{equation}
\end{dfn}

\begin{dfn}[Function Terminology]
A function $f$ on $\mathbb{P}$ is
\begin{itemize}
\item \emph{involutive} if $f(f(p))=p$, for all $p\in\mathbb{P}$.
\item a \emph{complementation} if $p$ and $f(p)$ are complementary, for all $p\in\mathbb{P}$.
\item \emph{order preserving} if $p\leq q\Rightarrow f(p)\leq f(q)$, for all $p,q\in\mathbb{P}$.
\item \emph{antitone} if $p\leq q\Rightarrow f(q)\leq f(p)$, for all $p,q\in\mathbb{P}$.
\item an \emph{orthocomplementation} if $f$ is an antitone involutive complementation.
\item an \emph{order isomorphism} if $f$ is 1-1, onto, and $f$ and $f^{-1}$ are order preserving.
\item an \emph{orthoisomorphism} if, further, $f(p^\perp)=f(p)^\perp$, for all $p\in\mathbb{P}$.
\end{itemize}
A partial order $\mathbb{P}$ with a distinguished orthocomplementation is an \emph{orthoposet} and, if $\mathbb{P}$ is also a lattice, an \emph{ortholattice}.
\end{dfn}

\subsection{Orthocompletions}\label{TheCompletion}

We will be interested in a particular case of the following situation (see \S\ref{annsec}).  We are given a relation $\perp$ on a set $S$ and, for $T\subseteq S$, define \[T^\perp=\{s\in S:\forall t\in T(t\perp s)\}\qquad\textrm{and}\qquad[S]^\perp=\{T^\perp:T\subseteq S\}.\]
Also, for future reference, we make the following definition.
\begin{dfn}
We call $T\subseteq S$ \emph{essential} (w.r.t. $\perp$) if $T^{\perp\perp}=S$.
\end{dfn}

For a collection of subsets $\mathcal{T}$ of $S$, we have $(\bigcup\mathcal{T})^\perp=\bigcap\{T^\perp:T\in\mathcal{T}\}$.  In particular, this means infimums, w.r.t. the inclusion order, always exist in $[S]^\perp$ are are simply given by intersections.  Furthermore, $S=\emptyset^\perp$ is the largest element of $[S]^\perp$, while $S^\perp=\{s\in S:\forall t\in S(t\perp s)\}$ is the smallest, and $T\mapsto T^\perp$ is
\begin{enumerate}\label{perp}
\item\label{perp1} antitone,
\item\label{perp2} involutive on $[S]^\perp$, if $\perp$ is symmetric, and
\item\label{perp3} an orthocomplementation on $[S]^\perp$, if $\perp$ is a preorthogonality relation.
\end{enumerate}
\eqref{perp1} is immediate, and for \eqref{perp2} note that symmetry implies $T\subseteq T^{\perp\perp}$, for all $T\subseteq S$, and thus, by \eqref{perp1}, $(T^{\perp\perp})^\perp\subseteq T^\perp\subseteq (T^\perp)^{\perp\perp}$, i.e. $T^\perp=T^{\perp\perp\perp}$.  Lastly, if $\perp$ is also annihilating then $T\cap T^\perp=S^\perp$, for all $T\in[S]^\perp$, and hence $T^\perp\vee T=T^\perp\vee T^{\perp\perp}=S$ so $T$ and $T^\perp$ are complementary, i.e. $[S]^\perp$ is a complete ortholattice.  We call $[S]^\perp$ the \emph{orthocompletion} of $S$ w.r.t. $\perp$.  In fact, we have really just proved a slightly more general version of \cite{MacLaren1964} Lemma 2.1, and what we have denoted by $[S]^\perp$ is denoted in \cite{MacLaren1964} by $L(S)$, where it is called the completion of $S$.  As shown in \cite{MacLaren1964} Theorem 2.4, it really is the canonical completion by cuts when $S$ itself is an orthoposet.


We define the preorder $\dashv$ \emph{induced} by $\perp$ on $S$ by \[s\dashv t\quad\Leftrightarrow\quad\{t\}^\perp\subseteq\{s\}^\perp.\]  Note that $\dashv$ is a partial order if and only if $\perp$ satisfies \eqref{orthorel} and that
\begin{equation}\label{sperpperp}
s\mapsto\{s\}^{\perp\perp}
\end{equation}
is an order preserving map from $S$ (ordered by $\dashv$) to $[S]^\perp$ (ordered by $\subseteq$).  If $\perp$ is symmetric then ${}^\perp$ is involutive and hence we actually have \[s\dashv t\quad\Leftrightarrow\quad\{s\}^{\perp\perp}\subseteq\{t\}^{\perp\perp}.\]
If $T\subseteq S$ is join-dense in $S$ (see \eqref{jd}) w.r.t. $\dashv$ then, by (a slight generalization of) \cite{MacLaren1964} Theorem 2.5, the map
\begin{equation}\label{jdorthoiso}
U\mapsto U\cap T\quad\textrm{ is an orthoisomorphism witnessing}\quad[S]^\perp\cong[T]^\perp.
\end{equation}

Going in the other direction, given an orthoposet $\mathbb{P}$, we can always define an orthogonality relation $\perp$ by \[p\perp q\quad\Leftrightarrow\quad p\leq q^\perp.\]  If $\mathbb{P}=[S]^\perp$, where $\perp$ is a preorthogonality relation on $S$ then, for $T,U\in[S]^\perp$, the relation $\perp$ on $[S]^\perp$ defined in this way is related to the original relation $\perp$ on $S$ by \[T\perp U\quad\Leftrightarrow\quad\forall t\in T\forall u\in U(t\perp u).\]  Furthermore, for $T\in[S]^\perp$, we can consider the restriction of $\perp$ to $T$ and then \[[T]^\perp=[T]_T,\] according to \autoref{[p]_q}.


\subsection{Relative Complements}

\begin{dfn}\label{[p]_q}
Given an ortholattice $\mathbb{P}$ and $p,q,r\in\mathbb{P}$ we define $r^{\perp_p}=r^\perp\wedge p$ and
\[[q]_p=\{r^{\perp_p}:r\leq p\textrm{ and }r^{\perp_p}\leq q\}.\]
\end{dfn}

In particular, $[p]_p=\{r^{\perp_p}:r\leq p\}$.  We also drop the subscript when $r=\mathbf{1}$, i.e. \[[p]=[p]_\mathbf{1}=\{q:q\leq p\}\quad\textrm{and}\quad[q]_p=[p]_p\cap[q].\]  It is important to note that $[p]$ may not be an ortholattice (see the Hasse diagram below), in contrast to $[p]_p$.

\begin{prp}\label{[p]_p}
If $\mathbb{P}$ is an ortholattice and $p\in\mathbb{P}$ then $[p]_p$ is an ortholattice.  If $q\in[p]_p$ then $[q]_q\subseteq[p]_p$ while, for any $q\in[p]$, we have
\begin{equation}\label{[p]_peq}
q^{\perp_p\perp_p}=\bigwedge\{r\in[p]_p:q\leq r\}.
\end{equation} 
\end{prp}

\begin{proof}
First note that for $q,r\leq p$ we have $(q^\perp\wedge p)\wedge(r^\perp\wedge p)=(q\vee r)^\perp\wedge p\in[p]_p$ so infimums exist and agree with those in $\mathbb{P}$.  Next note that, when $q\leq p$, we have $q\leq(p\wedge q^\perp)^\perp\wedge p$ so $q^\perp\geq((p\wedge q^\perp)^\perp\wedge p)^\perp$ and \[p\wedge q^\perp\geq((p\wedge q^\perp)^\perp\wedge p)^\perp\wedge p=((p\wedge q^\perp)\vee p^\perp)\wedge p\geq p\wedge q^\perp,\] so $^{\perp_p}$ is involutive and, therefore, actually characterizes the elements of $[p]_p$, i.e. \[\qquad[p]_p=\{q=q^{\perp_p\perp_p}:q\in\mathbb{P}\}.\]  And if $q\in[p]_p$ and $r\in[q]_q$ then $r^{\perp_q}\leq p$ and hence $r^{\perp_q\perp_p}\in[p]_p$.  Therefore $r=r^{\perp_q\perp_q}=r^{\perp_q\perp_p}\wedge q\in[p]_p$, i.e. $[q]_q\subseteq [p]_p$.  As ${}^\perp$ is order reversing, so is ${}^{\perp_p}$ so, in particular, supremums also exist.  Also, if $r\in[p]_p$ and $q\in[r]$ then $r^{\perp_p}\leq q^{\perp_p}$ and hence $q^{\perp_p\perp_p}\leq r^{\perp_p\perp_p}=r$, thus verifying \eqref{[p]_peq}.  Finally, for any $q\in[p]_p$, we have $q^{\perp_p}\wedge q\leq q^\perp\wedge q\leq\mathbf{0}$ and hence also $q^{\perp_p}\vee_p q=\mathbf{0}^{\perp_p}=p$.  Thus $q^{\perp_p}$ is a complement of $q$ in $[p]_p$, i.e. $[p]_p$ is an ortholattice with orthocomplement function ${}^{\perp_p}$.
\end{proof}

On the other hand, $[p]$ is always a sublattice of $\mathbb{P}$, while $[p]_p$ may not be.  Indeed, while infimums in $[p]_p$ agree with those in $\mathbb{P}$, the same can not be said for supremums.  For example, in the ortholattice represented by the following Hasse diagram ($x\leq y$ in such a diagram if and only if $y$ appears above $x$ and joined to it by lines), which appears as \cite{Kalmbach1983} Figure 6.5, $[p]_p=\{\mathbf{0},a^\perp,c^\perp,p\}$ and hence $a^\perp\vee_p c^\perp=p$, even though $a^\perp\vee c^\perp=b$.  Also, $[p]=\{\mathbf{0},a^\perp,c^\perp,p,b\}$, which does not possess any orthocomplement functions. 

\begin{figure}[h!]
\caption{}\label{H1}
\begin{center}
\begin{tikzpicture}
  \node (max) at (0,3) {$\mathbf{1}$};
	\node (p) at (0,2) {$p$};
  \node (a) at (-2,1) {$a$};
  \node (b) at (0,1) {$b$};
  \node (c) at (2,1) {$c$};
  \node (cp) at (-2,0) {$c^\perp$};
  \node (bp) at (0,0) {$b^\perp$};
  \node (ap) at (2,0) {$a^\perp$};
	\node (pp) at (0,-1) {$p^\perp$};
  \node (min) at (0,-2) {$\mathbf{0}$};
  \draw (min) -- (cp) -- (a) -- (max) -- (c) -- (ap) -- (min)
  (min) -- (pp) -- (bp) -- (a)
  (bp) -- (c)
	(max) -- (p) -- (b);
  \draw[preaction={draw=white, -,line width=6pt}] (cp) -- (b) -- (ap);
\end{tikzpicture}
\end{center}
\end{figure}

\subsection{Order Types}

\begin{dfn}\label{orthomod}
A preorder $\mathbb{P}$ is \emph{separative} if, for all $p,q\in\mathbb{P}$, \[p\nleq q\quad\Rightarrow\quad\exists r\in\mathbb{P}(\mathbf{0}<r\leq p\textrm{ and }r\wedge q=\mathbf{0}).\]
We call an orthoposet $\mathbb{P}$ \emph{orthomodular} if, for all $p,q\in\mathbb{P}$, $p\perp q\Rightarrow p\vee q$ exists, and
\begin{equation}\label{orthomodeq}
q\leq p\quad\Rightarrow\quad p=q\vee(p\wedge q^\perp)
\end{equation}
A lattice $\mathbb{P}$ is \emph{modular} if, for $p,q,r\in\mathbb{P}$, 
\[q\leq p\quad\Rightarrow\quad p\wedge(q\vee r)=q\vee(p\wedge r).\]
A lattice $\mathbb{P}$ is \emph{distributive} if, for all $p,q,r\in\mathbb{P}$,
\begin{equation}\label{distributive}
p\wedge(q\vee r)=(p\wedge q)\vee(p\wedge r)\qquad\textrm{and}\qquad p\vee(q\wedge r)=(p\vee q)\wedge(p\vee r).
\end{equation}
A \emph{Boolean algebra} is a distributive complemented lattice.
\end{dfn}

Every element of a Boolean algebra in fact has a \emph{unique} complement and the map taking each element to this unique complement is an orthocomplement function.  In fact, an ortholattice is uniquely complemented if and only if it is a Boolean algebra, by \cite{Kalmbach1983} \S3 Proposition 7.  For an ortholattice, we immediately have \[\textrm{distributivity}\quad\Rightarrow\quad\textrm{modularity}\quad\Rightarrow\quad\textrm{orthomodularity}\quad\Rightarrow\quad\textrm{separativity}.\]

To see that the first two of these implications can not be reversed, it suffices to note that the subspaces of a Hilbert space $H$ are modular (more generally, submodules of a module are modular, hence the name) but not distributive if $\dim(H)>1$, while the \emph{closed} subspaces are orthomodular but not modular if $\dim(H)=\infty$, by \cite{Kalmbach1983} \S5 Proposition 5.  This last fact is actually key to showing that the projections in an infinite AW*-algebra are not modular (see \cite{Kaplansky1955} Theorem on page 1).  There are also finite ortholattices that illustrate these differences, for example the Chinese latern MO2 represented by \cite{Kalmbach1983} Figure 2.1 11 is modular but not distributive, while the ortholattice in \cite{Kalmbach1983} Figure 3.2 is orthomodular but not modular.  For an example of an ortholattice that is separative but not orthomodular, consider the orthodouble of the 8 element Boolean algebra given in \cite{Flachsmeyer1982} Figure 2b, as represented by the following Hasse diagram.

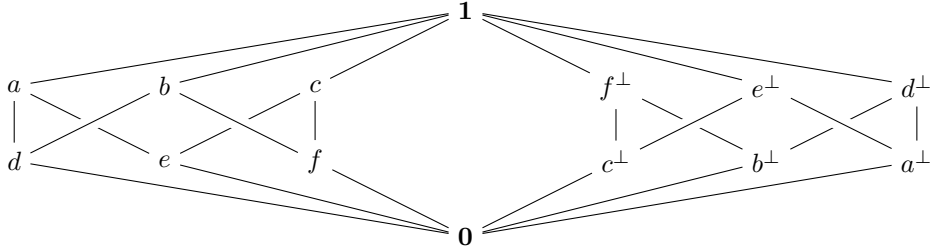
\begin{center}
\begin{tikzpicture}
  \node (max) at (0,2) {$\mathbf{1}$};
  \node (a) at (-6,1) {$a$};
  \node (b) at (-4,1) {$b$};
  \node (c) at (-2,1) {$c$};
  \node (d) at (-6,0) {$d$};
  \node (e) at (-4,0) {$e$};
  \node (f) at (-2,0) {$f$};
	\node (ap) at (6,0) {$a^\perp$};
  \node (bp) at (4,0) {$b^\perp$};
  \node (cp) at (2,0) {$c^\perp$};
  \node (dp) at (6,1) {$d^\perp$};
  \node (ep) at (4,1) {$e^\perp$};
  \node (fp) at (2,1) {$f^\perp$};
  \node (min) at (0,-1) {$\mathbf{0}$};
  \draw (min) -- (d) -- (a) -- (max) -- (c) -- (f) -- (min)
  (min) -- (e) -- (a)
  (e) -- (c)
	(max) -- (b);
  \draw[preaction={draw=white, -,line width=6pt}] (d) -- (b) -- (f);
  \draw (min) -- (cp) -- (fp) -- (max) -- (dp) -- (ap) -- (min)
  (min) -- (bp) -- (fp)
  (bp) -- (dp)
	(max) -- (ep);
  \draw[preaction={draw=white, -,line width=6pt}] (cp) -- (ep) -- (ap);
\end{tikzpicture}
\end{center}
For an example of an ortholattice that is not even separative, just consider $O_6$ in \cite{Kalmbach1983} Figure 3.1.

\subsection{Separativity}

Separativity is fundamental to our later work because it is precisely what is required to turn order-density into join-density.

\begin{prp}\label{jdod}
A preorder $\mathbb{P}$ is separative if and only if, for all $S\subseteq\mathbb{P}$, \[S\textrm{ is join-dense}\quad\Leftrightarrow\quad S\textrm{ is order-dense}.\]
\end{prp}

\begin{proof} Join-density certainly implies order-density, while if $S$ is not join-dense then we can find $p,q\in\mathbb{P}$ with $p\nleq q$ even though $q\geq r$, for all $r\in[p]\cap S$.  If $\mathbb{P}$ is separative then we can find $t\in\mathbb{P}$ with $\mathbf{0}<t\leq p$ and $t\wedge q=\mathbf{0}$, and hence there is no $s\leq t$ with $\mathbf{0}<s\in S$, i.e. $S$ is not order-dense.

On the other hand, if $\mathbb{P}$ is not separative then we have $p,q\in\mathbb{P}$ with $p\nleq q$ even though there is no $r\leq p$ with $r>\mathbf{0}$ and $r\wedge q=\mathbf{0}$.  Now consider \[S=[q]\cup\{s\in\mathbb{P}:s\nleq p\}.\]  If $t\nleq p$ then $t\in S$, while if $\mathbf{0}<t\leq p$ then $t\wedge q\neq\mathbf{0}$, i.e. there exists $s>\mathbf{0}$ with $t\geq s\in[q]\subseteq S$.  Thus $S$ is order-dense, however, $S\cap[p]\subseteq[q]$ which, as $p\nleq q$, means that $p\neq\bigvee S\cap[p]$, so $S$ is not join-dense.
\end{proof}

In a lattice $\mathbb{P}$ another equivalent of separativity is obtained if we replace $p\nleq q$ in the definition of separativity with the apparently stronger condition $q<p$.  For $p\nleq q$ implies $p\wedge q<p$ and hence we could find $r\leq p$ with $\mathbf{0}=(q\wedge p)\wedge r=q\wedge(p\wedge r)=q\wedge r$.

\subsection{Perspectivity}\label{persec}

\begin{dfn}
If $\mathbb{P}$ is an ortholattice, we say $p,q\in\mathbb{P}$ are
\begin{enumerate}
\item \emph{perspective} if $p$ and $q$ have a common complement.
\item \emph{orthoperspective} if $p$ and $q$ have a common orthogonal complement.
\item \emph{semiorthoperspective} if $p$ and $q^\perp$ are complementary.
\end{enumerate}
These relations will be denoted by $\sim_\mathrm{p}$, $\sim_\mathrm{op}$ and $\sim_\mathrm{sop}$ respectively.
\end{dfn}
The definition of perspectivity is perfectly valid in an arbitrary lattice with $\mathbf{1}$ and $\mathbf{0}$ and is fundamental to the theory of continuous geometries (see \cite{vonNeumann1960}).  Note that $p$ and $q^\perp$ are complementary if and only if $p^\perp$ and $q$ are complementary, so semiorthoperspectivity is a symmetric relation.  Semiorthoperspective $p$ and $q$ are sometimes said to be `in position $p'$', while if $p$ is also semiorthoperspective to $q^\perp$ then they are `in generic position' or `in position $p$' (see \cite{Berberian1972} \S13 Definition 2 and Definition 3).  Perspectivity is weaker than semiorthoperspectivity, however their transitive closures are the same.  For if $p$ and $q$ have common complement $r$ then $p$ is semiorthoperspective to $r^\perp$ which is in turn semiorthoperspective to $q$.  Orthoperspectivity on the other hand is much stronger, and is often just equality (see \autoref{orthoequiv}).  Note that $p$ and $q$ are orthoperspective if and only if $(p\vee q)^\perp$ is complementary to both $p$ and $q$.

\begin{prp}\label{orthoperpequiv}
For a symmetric transitive relation $\sim$ on an ortholattice $\mathbb{P}$, the following are equivalent.
\begin{enumerate}
\item\label{orthoperpequiv1} $\sim$ is weaker than orthoperspectivity.
\item\label{orthoperpequiv2} $q\leq p$ and $q^\perp\wedge p=\mathbf{0}\Rightarrow p\sim q$.
\item\label{orthoperpequiv3} $q\sim q^{\perp_p\perp_p}$, for all $p\in\mathbb{P}$ and $q\in[p]$.
\end{enumerate}
\end{prp}

\begin{proof}\
\begin{itemize}
\item[\eqref{orthoperpequiv1}$\Rightarrow$\eqref{orthoperpequiv2}] $p\vee q=p$ so $p\vee(p\vee q)^\perp=\mathbf{1}$ and $q\vee(p\vee q)^\perp=(q^\perp\wedge p)^\perp=\mathbf{1}$ so $p$ and $q$ are orthoperspective and hence $p\sim q$.
\item[\eqref{orthoperpequiv2}$\Rightarrow$\eqref{orthoperpequiv1}] If $p\vee(p\vee q)^\perp=\mathbf{1}$ then $p^\perp\wedge(p\vee q)=\mathbf{0}$ so orthoperspectivity and \eqref{orthoperpequiv2} imply $p\sim p\vee q\sim q$ which, by transitivity, gives $p\sim q$.
\item[\eqref{orthoperpequiv2}$\Rightarrow$\eqref{orthoperpequiv3}] $q\leq q^{\perp_p\perp_p}$ and $q^\perp\wedge q^{\perp_p\perp_p}=q^{\perp_p}\wedge q^{\perp_p\perp_p}=\mathbf{0}$ so $q\sim q^{\perp_p\perp_p}$.
\item[\eqref{orthoperpequiv3}$\Rightarrow$\eqref{orthoperpequiv2}] $q^{\perp_p\perp_p}=(q^\perp\wedge p)^{\perp_p}=\mathbf{0}^{\perp_p}=p$ so $p\sim q$.
\end{itemize}
\end{proof}

(Transitivity only appears in \eqref{orthoperpequiv2}$\Rightarrow$\eqref{orthoperpequiv1} so orthoperspectivity satisfies \eqref{orthoperpequiv2} and \eqref{orthoperpequiv3}.)

\begin{prp}\label{pqposp'}
If $\mathbb{P}$ is an ortholattice, $p,q\in\mathbb{P}$, $[p]_p=[p]$ and $[q]=[q]_q$ then
\[p'=(p\wedge q^\perp)^{\perp_p}\quad\textrm{and}\quad q'=(q\wedge p^\perp)^{\perp_q}\quad\textrm{are semiorthoperspective}.\]
\end{prp}

\begin{proof}
Note that $s=p'\wedge q'^\perp\leq p$ and hence $s\perp(q\wedge p^\perp)$.  But then $s\perp(q\wedge p^\perp)\vee q'=(q\wedge p^\perp)\vee(q\wedge p^\perp)^{\perp_q}=q$, as $[q]=[q]_q$, and hence $s\leq p\wedge q^\perp$ which, as $s\leq(p\wedge q^\perp)^\perp$, means $s=\mathbf{0}$.  Now $q'\wedge p'^\perp=\mathbf{0}$ follows by a symmetric argument.
\end{proof}

\begin{cor}\label{sopcor}
If $\mathbb{P}$ is an ortholattice, $p,q\in\mathbb{P}$, $[p]_p=[p]$ and $[q]=[q]_q$ then
\[\exists r\in\mathbb{P}(p\leq r\sim_\mathrm{sop}q)\quad\Rightarrow\quad\exists s\in\mathbb{P}(p\sim_\mathrm{sop}s\leq q).\]
\end{cor}

\begin{proof}
With $p'$ and $q'$ as above we have $p'\sim_\mathrm{sop}q'$.  But $p\wedge q^\perp\leq r\wedge q^\perp=\mathbf{0}$ so $p'=\mathbf{0}^{\perp_p}=p$.  So, setting $s=q'$, we are done.
\end{proof}

\subsection{Finiteness}

\begin{dfn}
Given a symmetric relation $\sim$ on $\mathbb{P}$, we call $p\in\mathbb{P}$ \emph{$\sim$-finite} if \[p\sim q\leq p\quad\Rightarrow\quad p=q.\]
If $\mathbb{P}$ is an ortholattice, we call $p\in\mathbb{P}$ \emph{$\sim$-orthofinite} if \[p\sim q\leq p\quad\Rightarrow\quad p\wedge q^\perp=\mathbf{0}.\] We call $\sim$ itself \emph{(ortho)finite} if all elements of $\mathbb{P}$ are $\sim$-(ortho)finite.
\end{dfn}

Note that if $\sim$ is transitive and weaker than orthoperspectivity then, when $p$ is not orthofinite, there exists $q\leq p$ with $p\sim q\sim q^{\perp_p\perp_p}<p$, as $q^{\perp_p}\neq\mathbf{0}$, i.e.
\begin{eqnarray*}
p\textrm{ is orthofinite} &\Leftrightarrow& \{p\}^\sim\cap[p]_p=\{p\}.\textrm{ Also, by definition,}\\
p\textrm{ is finite} &\Leftrightarrow& \{p\}^\sim\cap[p]=\{p\}.
\end{eqnarray*}

\begin{dfn}
We say $\sim$ is \emph{finitely additive} if, whenever $p\perp q$, $r\perp s$, $p\sim r$ and $q\sim r$, we have $p\vee q\sim r\vee s$.
\end{dfn}

\begin{prp}
A finitely additive reflexive relation $\sim$ on an ortholattice $\mathbb{P}$ is orthofinite if and only if $\mathbf{1}$ is (ortho)finite.
\end{prp}

\begin{proof}
If $\sim$ is orthofinite then, in particular, $\mathbf{1}$ is (ortho)finite.  If $\sim$ is not orthofinite, we have $p\sim q\leq p$ with $p\wedge q^\perp\neq\mathbf{0}$.  As $\sim$ is finitely additive and reflexive, we have $\mathbf{1}=p\vee p^\perp\sim q\vee p^\perp$ even though $(q\vee p^\perp)^\perp=p\wedge q^\perp\neq\mathbf{0}$ so $q\vee p^\perp\neq\mathbf{1}$.
\end{proof}

\subsection{Orthomodularity}

Orthomodularity has a number of important equivalents (for more see \cite{Kalmbach1983} \S3 Theorem 2).

\begin{prp}\label{orthoequiv}
For an ortholattice $\mathbb{P}$, the following are equivalent.
\begin{enumerate}
\item\label{orthoequiv1} $\mathbb{P}$ is orthomodular.
\item\label{orthoequiv2} Orthogonal complements are unique.
\item\label{orthoequiv3} $[p]_p=[p]$, for all $p\in\mathbb{P}$.
\item\label{orthoequiv4} Orthoperspectivity is finite.
\item\label{orthoequiv6} Orthoperspectivity is equality.
\item\label{orthoequiv7} $\sim$-orthofiniteness and $\sim$-finiteness always coincide.
\item\label{orthoequiv5} For all $p,q,r\in\mathbb{P}$, $q\leq p$ and $q\perp r\Rightarrow p\wedge(q\vee r)=q\vee(p\wedge r)$.
\end{enumerate}
\end{prp}

\begin{proof}\
\begin{itemize}
\item[(\ref{orthoequiv1})$\Rightarrow$(\ref{orthoequiv2})] If $q\leq p^\perp$ and $p\vee q=\mathbf{1}$ then $p^\perp\wedge q^\perp=\mathbf{1}^\perp=\mathbf{0}$ so orthomodularity gives $p^\perp=q\vee(q^\perp\wedge p^\perp)=q$, i.e. orthogonal complements are unique.
\item[\eqref{orthoequiv2}$\Rightarrow$\eqref{orthoequiv1}] Given $q\leq p$ let $r=q\vee(p\wedge q^\perp)\leq p$.  Then $p^\perp\vee r=p^\perp\vee q\vee(p\wedge q^\perp)=p^\perp\vee q\vee(p^\perp\vee q)^\perp=\mathbf{1}$ so \eqref{orthoequiv2} $r=p$, showing that $\mathbb{P}$ is orthomodular.

\item[\eqref{orthoequiv2}$\Rightarrow$\eqref{orthoequiv4}] If $q\leq p=p^{\perp\perp}$ and $\mathbf{0}=q^\perp\wedge p=(q\vee p^\perp)^\perp$ then \eqref{orthoequiv2} gives $q=p^{\perp\perp}=p$.
\item[\eqref{orthoequiv4}$\Rightarrow$\eqref{orthoequiv2}] If $p\in\mathbb{P}$ has an orthogonal complement $q<p^\perp$ then $p^\perp\wedge q^\perp=\mathbf{0}$.

\item[\eqref{orthoequiv6}$\Rightarrow$\eqref{orthoequiv4}]  Equality is finite.
\item[\eqref{orthoequiv4}$\Rightarrow$\eqref{orthoequiv6}]  If $p$ and $q$ are orthoperspective then so are $p$ and $p\vee q$ which, if orthoperspectivity is finite, means $p=p\vee q$.  Likewise $q=p\vee q=p$.

\item[\eqref{orthoequiv4}$\Rightarrow$\eqref{orthoequiv7}]  $q\leq p$ and $q^\perp\wedge p=\mathbf{0}$ means $p$ and $q$ are orthoperspective which, from \eqref{orthoequiv4}, gives $p=q$.  Thus $\sim$-orthofiniteness implies $\sim$-finiteness.
\item[\eqref{orthoequiv7}$\Rightarrow$\eqref{orthoequiv4}]  Orthoperspectivity is orthofinite.  If it is not finite these notions differ.

\item[\eqref{orthoequiv6}$\Rightarrow$\eqref{orthoequiv3}]  $q$ and $q^{\perp_p\perp_p}$ are orthoperspective so \eqref{orthoequiv6} gives $q=q^{\perp_p\perp_p}\in[p]_p$ for $q\in[p]$.
\item[\eqref{orthoequiv3}$\Rightarrow$\eqref{orthoequiv4}] If $q<p$ but $q^\perp\wedge p=\mathbf{0}$ then $q\notin[p]_p$.

\item[\eqref{orthoequiv1}$\Rightarrow$\eqref{orthoequiv5}] See \cite{Kalmbach1983} \S3 Theorem 5.
\item[\eqref{orthoequiv5}$\Rightarrow$\eqref{orthoequiv1}] Immediate by setting $r=q^\perp$.
\end{itemize}
\end{proof}

So if $\mathbb{P}$ is not orthomodular then we have $[p]\neq[p]_p$, for some $p\in\mathbb{P}$, and there is no reason to think that properties $[p]$ inherits from $\mathbb{P}$, like separativity, are necessarily inherited by $[p]_p$.  However, if we happen to know that $[p]_p$ is order-dense in $[p]$ (which is true for the annihilators we will be interested in \textendash\, see the comment after \autoref{prp1}), then $[p]_p$ will indeed be separative if $\mathbb{P}$ is.

\begin{prp}
If $\mathbb{P}$ is an orthomodular lattice and $p,q\in\mathbb{P}$ then
\[p'=(p\wedge q^\perp)^{\perp_p}\quad\textrm{and}\quad q'=(q\wedge p^\perp)^{\perp_q}\quad\textrm{are maximal semiorthoperspective}.\]
\end{prp}

\begin{proof}
Semiorthoperspectivity is immediate from \autoref{pqposp'} and \autoref{orthoequiv}.  For maximality, note that if $s>p'$ then $\mathbf{0}\neq s\wedge p'^\perp=s\wedge p\wedge q^\perp$, by orthomodularity.  Thus $s\wedge q^\perp\neq0$ so $s$ could not be semiorthoperspective to anything in $[q]$.
\end{proof}


\subsection{Modularity}

As shown in \cite{Jacobson1985} Theorem 8.4, $\mathbb{P}$ is modular if and only if, for $p,q,r\in\mathbb{P}$,
\begin{equation}\label{persp}
p\vee r=q\vee r,\ p\wedge r=q\wedge r\textrm{ and }p\leq q\quad\Rightarrow\quad p=q.
\end{equation}\label{perfin}
In particular, this means perspectivity is a finite relation on $\mathbb{P}$.  In fact, if $\mathbb{P}$ is an ortholattice then
\begin{equation}\label{modperfin}
\textrm{modularity}\qquad\Leftrightarrow\qquad\textrm{perspectivity is finite}.
\end{equation}
To see the converse, first note that if perspectivity is finite then so is orthoperspectivity and hence $\mathbb{P}$ is orthomodular, by \autoref{orthoequiv}.  Now say $p,q\in\mathbb{P}$ satisfy the conditions on the left of \eqref{persp} then set $p'=p\wedge(p\wedge r)^\perp$ and $p''=p'\vee(p'\vee r)^\perp$ and likewise for $q'$ and $q''$.  Then we see that $p'\wedge r=\mathbf{0}=q'\wedge r$ and hence $p''\wedge r=\mathbf{0}=q''\wedge r$, as well as $p''\vee r=\mathbf{1}=q''\vee r$.  So if perspectivity is finite then $p''=q''$.  As $\mathbb{P}$ is orthomodular, $p'\vee r=p\vee r=q\vee r=q'\vee r$, and so orthomodularity together with $p''=q''$ implies $p'=q'$ and hence orthomodularity together with $p\wedge r=q\wedge r$ finally yields $p=q$.

\subsection{The Centre}

\begin{dfn}\label{centredef}
In an ortholattice $\mathbb{P}$, we say $s,t\in\mathbb{P}$ \emph{commute} if \eqref{distributive} holds whenever $p,q,r\in\{s,s^\perp,t,t^\perp\}$.  We call $p\in\mathbb{P}$ \emph{central} if it commutes with all $q\in\mathbb{P}$.  If $\mathbb{P}$ is complete and $p\in\mathbb{P}$, we define the \emph{central cover} $\mathrm{c}(p)$ of $p$ by \[\mathrm{c}(p)=\bigwedge\{q\geq p:q\textrm{ is central}\}.\]  Given $T\subseteq\mathbb{P}$ we define \[\mathrm{c}T=\{\mathrm{c}(t):t\in T\}\] so $\mathrm{c}\mathbb{P}=\{p\in\mathbb{P}:p\textrm{ is central}\}$.  We call $p,q\in\mathbb{P}$ \emph{very orthogonal} if $\mathrm{c}(p)\perp \mathrm{c}(q)$.
\end{dfn}

The only non-trivial instances of \eqref{distributive}, for $p,q,r\in\{s,s^\perp,t,t^\perp\}$, are of the form
\begin{eqnarray}
p\wedge q &=& p\wedge(p^\perp\vee q)\quad\textrm{and}\label{com1}\\
p &=& (p\wedge q)\vee(p\wedge q^\perp),\label{com2}
\end{eqnarray}
for $p\in\{s,s^\perp\}$ and $q\in\{t,t^\perp\}$, or $p\in\{t,t^\perp\}$ and $q\in\{s,s^\perp\}$.  Thus, the order theoretic definition of commutativity agrees with the algebraic definition for projections in a C*-algebra, by \autoref{pq=qp}.  In fact, for orthomodular lattices, each non-trivial instance of \eqref{distributive} is equivalent to every other and to \[\mathbf{1}=(s\wedge t)\vee(s^\perp\wedge t)\vee(s\wedge t^\perp)\vee(s^\perp\wedge t^\perp),\] by \cite{Kalmbach1983} \S3 Lemma 3 and Proposition 8.  Even if $\mathbb{P}$ is not orthomodular, we still have the following important characterizations of centrality.

\begin{dfn}
Given partial orders $\mathbb{P}$ and $\mathbb{Q}$ we order $\mathbb{P}\times\mathbb{Q}$ coordinatewise, i.e. $(p,q)\leq(r,s)\Leftrightarrow p\leq r\textrm{ and }q\leq s$.  If they are orthocomplemented, we make $\mathbb{P}\times\mathbb{Q}$ orthocomplemented by defining $(p,q)^\perp=(p^\perp,q^\perp)$.  Given an ortholattice $\mathbb{P}$ and $p\in\mathbb{P}$, we say $\mathbb{P}$ is \emph{canonically isomorphic} to $[p]\times[p^\perp]$ to mean that $[p]=[p]_p$, $[p^\perp]=[p^\perp]_{p^\perp}$ and the (order preserving) maps $(q,r)\mapsto q\vee r$ and $q\mapsto(q\wedge p,q\wedge p^\perp)$, from $[p]\times[p^\perp]$ to $\mathbb{P}$ and vice versa, are inverse to each other.
\end{dfn}

\begin{thm}\label{centralequiv}
If $\mathbb{P}$ is an ortholattice then the following are equivalent for $p\in\mathbb{P}$.
\begin{enumerate}
\item\label{centralequiv1} $q=(q\wedge p)\vee(q\wedge p^\perp)$, for all $q\in\mathbb{P}$.
\item\label{centralequiv2} $p$ is central.
\item\label{centralequiv3} $\mathbb{P}$ is canonically isomorphic to $[p]\times[p^\perp]$.
\end{enumerate}
\end{thm}

\begin{proof}
\item[\eqref{centralequiv1}$\Rightarrow$\eqref{centralequiv3}] See \cite{Kalmbach1983} \S3 Theorem 1 or \cite{MacLaren1964} Theorem 3.2.
\item[\eqref{centralequiv3}$\Rightarrow$\eqref{centralequiv2}] Immediately verified by coordinatewise calculations.
\item[\eqref{centralequiv2}$\Rightarrow$\eqref{centralequiv1}] Immediate from the definition of central.
\end{proof}

Note that \eqref{centralequiv3} is important because it means we can now do calculations coordinatewise in $[p]\times[p^\perp]$ rather than $\mathbb{P}$.  For example, say we had $p\in \mathrm{c}\mathbb{P}$ and $q\in\mathbb{P}$ and we want to show that
\begin{equation}\label{c(pwedgeq)}
\mathrm{c}(p\wedge q)=p\wedge\mathrm{c}(q).
\end{equation}
If equality fails here then we would have $r\in\mathrm{c}\mathbb{P}$ with $r\geq p\wedge q$ with $r\ngeq p\wedge\mathrm{c}(q)$.  By replacing $r$ with $r\wedge p\wedge\mathrm{c}(q)$ if necessary we may assume that $r<p\wedge\mathrm{c}(q)$.  Then $r\vee(p^\perp\wedge\mathrm{c}(q))\geq(p\wedge q)\vee(p^\perp\wedge q)=q$ even though \[r\vee(p^\perp\wedge \mathrm{c}(q))<(p\wedge\mathrm{c}(q))\vee(p^\perp\wedge\mathrm{c}(q))=\mathrm{c}(q),\] where we know the first inequality is strict because the inequality in the first coordinates in $[p]\times[p^\perp]$ is strict.  This contradicts the fact $\mathrm{c}(q)$ is the central cover of $q$ and we are done.

Still assuming $p\in\mathrm{c}\mathbb{P}$, we have $q\wedge p=(q\wedge p^\perp)^{\perp_q}\in[q]_q$.  For any $r\in[q]_q$, we have $r\wedge(q\wedge p)=(r\wedge q)\wedge p=r\wedge p$ and \[r\wedge(q\wedge p)^{\perp_q}=r\wedge q\wedge(q\wedge p)^\perp=r\wedge(q\wedge p^\perp)=r\wedge p^\perp.\]  So $r=(r\wedge p)\vee(r\wedge p^\perp)=(r\wedge(q\wedge p))\vee_q(r\wedge(q\wedge p)^{\perp_q})$ and $q\wedge p\in\mathrm{c}_q[q]_q$, i.e.
\begin{equation}\label{c_q}
\{q\wedge p:p\in\mathrm{c}\mathbb{P}\}\subseteq\mathrm{c}_q[q]_q.
\end{equation}
In general this inclusion can be strict (see \autoref{H1}, where $\mathrm{c}_b[b]_b=[b]=\{b,c^\perp,a^\perp,\mathbf{0}\}$ even though $\mathrm{c}\mathbb{P}=\{\mathbf{1},\mathbf{0}\}$), although not for the annihilators in a C*-algebra (see \autoref{c_B}).  This inclusion is also an equality for central elements, i.e. given $p\in\mathrm{c}\mathbb{P}$, we have $[p]=[p]_p$ and $\mathrm{c}\mathbb{P}\cong\mathrm{c}[p]\times\mathrm{c}[p^\perp]$ so
\begin{equation}\label{c([p])}
\mathrm{c}_p[p]_p=\mathrm{c}[p].
\end{equation}
Another consequence of \eqref{centralequiv3} is that, for any $p\in\mathrm{c}\mathbb{P}$, $p^\perp$ is the unique complement of $p$ so, in fact, $\mathrm{c}\mathbb{P}$ is a Boolean algebra, by \cite{Kalmbach1983} \S3 Proposition 7.  For orthomodular $\mathbb{P}$, the converse also holds.

\begin{prp}\label{cenuniqcomp}
For orthomodular $\mathbb{P}$, $p\in\mathrm{c}\mathbb{P}$ if $p^\perp$ is its only complement.
\end{prp}

\begin{proof}
If $p$ is not central then $q>(q\wedge p)\vee(q\wedge p^\perp)$, for some $q\in\mathbb{P}$.  Setting $q'=q\wedge(q\wedge p)^\perp$, we have $p\wedge q'=\mathbf{0}$ and, by orthomodularity, $q'>q\wedge p^\perp\geq q'\wedge p^\perp$.  This implies $q''=q'\vee(p\vee q')^\perp\neq p^\perp$ and, again by orthomodularity, $p\wedge q''=p\wedge q'=\mathbf{0}$.  Also $p\vee q''=p\vee q'\vee(p\vee q')^\perp=\mathbf{1}$, so $q''$ is a complement of $p$ different from $p^\perp$.
\end{proof}

For order theoretic type decompositions, what we actually need is an infinite version of \eqref{centralequiv3}, and the key extra condition required for this is separativity.

\begin{thm}\label{sepcomportho}
If $\mathbb{P}$ is a separative complete ortholattice, $q\in\mathbb{P}$ and $(p_\alpha)\subseteq\mathrm{c}\mathbb{P}$,
\begin{enumerate}
\item\label{sepcomportho1} $q\wedge\bigvee p_\alpha=\bigvee q\wedge p_\alpha$, and
\item\label{sepcomportho2} $q=(q\wedge\bigvee p_\alpha)\vee(q\wedge(\bigvee p_\alpha)^\perp)$.
\end{enumerate}
\end{thm}

\begin{proof}
Both statements follow from the fact \[S=[\bigwedge p^\perp_\alpha]\cup\bigcup[p_\alpha]\] is join-dense in $\mathbb{P}$.  For this, it suffices to prove $S$ is order-dense, by \autoref{jdod}.  Now note that, for any $q\in\mathbb{P}$, $q\wedge p_\alpha=\mathbf{0}$ implies $q\leq p^\perp_\alpha$, as $q=(q\wedge p_\alpha)\vee(q\wedge p^\perp_\alpha)$.  So if this were true for all $\alpha$, we would have $q\leq\bigwedge p^\perp_\alpha$.  In any case, if $q>\mathbf{0}$, we have $s\in S$ with $\mathbf{0}<s\leq q$.
\end{proof}

So, by \eqref{sepcomportho2} above and \autoref{centralequiv} \eqref{centralequiv1}, if $\mathbb{P}$ is a separative complete ortholattice,
\begin{equation}\label{centresublattice}
\mathrm{c}\mathbb{P}\textrm{ is a complete sublattice.}
\end{equation}
Another important consequence is the following.

\begin{cor}\label{ordertd}
For orthogonal $(p_\alpha)\subseteq\mathrm{c}\mathbb{P}$ in a separative complete ortholattice $\mathbb{P}$,
\begin{equation}\label{ordertdeq}
[\bigvee p_\alpha]\cong\prod[p_\alpha].
\end{equation}
\end{cor}

\begin{proof}
Define $f:[\bigvee_\alpha p_\alpha]\rightarrow\prod_\alpha [p_\alpha]$ and $g:\prod_\alpha[p_\alpha]\rightarrow[\bigvee_\alpha p_\alpha]$ by \[f(q)=\prod(p_\alpha\wedge q)\qquad\textrm{and}\qquad g((q_\alpha))=\bigvee q_\alpha.\]  Take $(q_\alpha)\subseteq\mathbb{P}$ with $q_\alpha\leq p_\alpha$, for all $\alpha$.  Given that $p_\alpha$ commutes with $q_\alpha$, we have $p_\alpha\wedge(q_\alpha\vee(\bigvee_{\beta\neq\alpha}q_\beta))\leq p_\alpha\wedge(q_\alpha\vee p_\alpha^\perp)=q_\alpha$, so $f\circ g$ is the identity map.  But $g\circ f$ is also the identity map, by \autoref{sepcomportho} \eqref{sepcomportho1} and thus $g$ and $f$ are (canonical) isomorphisms inverse to each other.
\end{proof}

Given a collection of ortholattices $(\mathbb{P}_\alpha)$ and $(p_\alpha),(q_\alpha)\in\prod\mathbb{P}_\alpha$ with $(q_\alpha)\leq(p_\alpha)$, we have $(q_\alpha)^\perp\wedge(p_\alpha)=(q_\alpha^{\perp_\alpha})\wedge(p_\alpha)=(q_\alpha^{\perp_\alpha}\wedge p_\alpha)$, i.e. \[[(p_\alpha)]_{(p_\alpha)}=\prod[p_\alpha]_{p_\alpha}.\]  This means that, by \autoref{ordertd}, given very orthogonal $(p_\alpha)\subseteq \mathbb{P}$ in a separative complete ortholattice $\mathbb{P}$, we also have \begin{equation}\label{bigveepalpha}
[\bigvee p_\alpha]_{\bigvee p_\alpha}\cong\prod[p_\alpha]_{p_\alpha}.
\end{equation}

\subsection{Type Decomposition}\label{tdsec}

\begin{dfn}\label{td}
Given a complete ortholattice $\mathbb{P}$, we call $T\subseteq\mathbb{P}$ a \emph{type ideal} if, whenever we have pairwise very orthogonal $S\subseteq\mathbb{P}$,
\begin{equation}\label{typedef}
S\subseteq T\qquad\Leftrightarrow\qquad\bigvee S\in T.
\end{equation}
If \eqref{typedef} holds for arbitrary $S\subseteq T$ then $T$ is a \emph{complete ideal}.
\end{dfn}

These type ideals correspond to the type-determining subsets defined in \cite{FoulisPulmannova2010} \S4 (6) in the context of effect algebras.  Note that if $T$ is a complete ideal then $T=[\bigvee T]$, i.e. complete ideals are precisely those subsets of $\mathbb{P}$ of the form $[p]$.  And if $T$ is a type ideal in a complete Boolean algebra $\mathbb{P}$ then $T$ is actually a complete ideal.  For then $\mathrm{c}\mathbb{P}=\mathbb{P}$ so if $p\leq q\in T$ then $q=p\vee(p^\perp\wedge q)$ and hence $p\in T$.  While given any $(t_\alpha)\subseteq T$ we can define (very) orthogonal $(s_\alpha)$ by $s_\alpha=t_\alpha\wedge(\bigvee_{\beta<\alpha}t_\beta)^\perp$ and transfinite induction gives $\bigvee t_\alpha=\bigvee s_\alpha\in T$.

The most common examples of type ideals come from the the type relations and type classes to be defined in the following sections.  However, there is one important example we can give straight away, namely the \emph{equality type ideal} $\mathbb{P}_=$ of any ortholattice $\mathbb{P}$ defined by
\begin{equation}\label{eqtypeideal}
\mathbb{P}_==\{p\in\mathbb{P}:[p]=[p]_p\}.
\end{equation}
This is indeed a type ideal, by \eqref{ordertdeq} and \eqref{bigveepalpha}, and $\mathbb{P}=\mathbb{P}_=$ if and only if $\mathbb{P}$ is orthomodular, by \autoref{orthoequiv}.

We call $T$ and $S$ \emph{complementary} when $(\bigvee T)\wedge(\bigvee S)=\mathbf{0}$ and $(\bigvee T)\vee(\bigvee S)=\mathbf{1}$.  

\begin{prp}\label{cideal}
Given a separative complete ortholattice $\mathbb{P}$ and a type ideal $T\subseteq\mathbb{P}$, the following pairs are complementary complete ideals of $\mathrm{c}\mathbb{P}$.
\begin{eqnarray}
T\cap\mathrm{c}\mathbb{P}\quad &\textrm{and}& \quad\{p\in\mathrm{c}\mathbb{P}:\mathrm{c}[p]\cap T=\{\mathbf{0}\}\}.\label{tdprp1}\\
\mathrm{c}T\quad &\textrm{and}& \quad\{p\in\mathrm{c}\mathbb{P}:[p]\cap T=\{\mathbf{0}\}\}.\label{tdprp2}
\end{eqnarray}
\end{prp}

\begin{proof}\
\begin{itemize}
\item[\eqref{tdprp1}] As $T$ is a type ideal, so is $T\cap\mathrm{c}\mathbb{P}$ which, as $\mathrm{c}\mathbb{P}$ is a complete Boolean algebra, means it is a complete ideal of $\mathrm{c}\mathbb{P}$.  Letting $q=\bigvee(T\cap\mathrm{c}\mathbb{P})$, we immediately have
\[\mathrm{c}[q^\perp]\subseteq\{p\in\mathrm{c}\mathbb{P}:\mathrm{c}[p]\cap T=\{\mathbf{0}\}\},\]
while if $p\in\mathrm{c}\mathbb{P}$ satisfies $\mathrm{c}[p]\cap T=\{\mathbf{0}\}$ then, as $p\wedge q\in T\cap\mathrm{c}\mathbb{P}$ (because $p\wedge q\leq q$ and $q\in T\cap\mathrm{c}\mathbb{P}$), we must have $p\wedge q=\mathbf{0}$ and hence $p\in[q^\perp]$, i.e. the inclusion above is an equality.

\item[\eqref{tdprp2}] By \eqref{c(pwedgeq)} and \eqref{centresublattice}, $\mathrm{c}(\bigvee p_\alpha)=\bigvee\mathrm{c}(p_\alpha)$ whenever $(p_\alpha)\subseteq\mathbb{P}$ so $\mathrm{c}T$ is a type ideal and thus, as above, a complete ideal of $\mathrm{c}\mathbb{P}$.  The fact that its complementary complete ideal is $\{p\in\mathrm{c}\mathbb{P}:[p]\cap T=\{\mathbf{0}\}\}$ follows in a similar manner, with another application of \eqref{c(pwedgeq)}.
\end{itemize}
\end{proof}

It immediately follows that any type ideal naturally gives rise to the following two type decompositions.

\begin{cor}\label{tdcor}
Given a type ideal $T$ in a separative complete ortholattice $\mathbb{P}$, there are unique $p_T,q_T\in\mathrm{c}\mathbb{P}$ such that
\begin{eqnarray}
p_T\in T\quad &\textrm{and}& \quad\mathrm{c}[p_T^\perp]\cap T=\{\mathbf{0}\}.\label{tdcor1}\\
q_T\in\mathrm{c}T\quad &\textrm{and}& \quad[q_T^\perp]\cap T=\{\mathbf{0}\}.\label{tdcor2}
\end{eqnarray}
\end{cor}

It follows immediately from \autoref{td} that, if $\mathbb{P}$ is a complete ortholattice, $\kappa$ is some (finite or infinite) cardinal and $T\subseteq\mathbb{P}$ is a type ideal, then so is
\begin{equation}\label{T^n}
T^\kappa=\{\bigvee S:S\subseteq T\textrm{ and }|T|=\kappa\}.
\end{equation}
If $\mathbb{P}$ is also separative then $\mathrm{c}(\bigvee S)=\bigvee\mathrm{c}(S)$, for any $S\subseteq\mathbb{P}$, which means $\mathrm{c}T=\mathrm{c}T^\kappa$, as $\mathrm{c}T$ is a complete ideal by \autoref{cideal}, i.e. $q_T=q_{T^\kappa}$.  It also means that $p_{T^\kappa}\leq q_T$, even though in this case we can have $p_T<p_{T^\kappa}$.


\subsection{Homogeneity}

We now show how order-density yields homogeneous decompositions.

\begin{dfn}\label{homdef}
Assume $\mathbb{P}$ is a complete ortholattice.  We say $p\in\mathbb{P}$ is $T$-\emph{homogeneous} if there are orthogonal $S\subseteq T\cap[p]$ with $p=\bigvee S$ and $\mathrm{c}S=\{\mathrm{c}(p)\}$.  If $|S|=\kappa$, we say that $p$ has \emph{order} $\kappa$ or that $p$ is \emph{$\kappa$-$T$-homogeneous}.  We denote the $\kappa$-$T$-homogeneous elements by $T_\kappa$.  We call $p\in\mathbb{P}$ $T$-\emph{subhomogeneous} if there are very orthogonal $T$-homogeneous $S$ with $p=\bigvee S$.  If each $s\in S$ has order $<\kappa$ then we say that $p$ is \emph{$<\!\!\kappa$-$T$-subhomogeneous}.  By $\kappa$-$T$-subhomogeneous we mean $<\!\!\kappa^+$-$T$-subhomogeneous.  We denote the $T$-subhomogeneous and $<\!\!\kappa$-$T$-subhomogeneous elements by $T_\mathrm{sub}$ and $T_{<\kappa}$ respectively
\end{dfn}

If $\mathbb{P}$ is a separative complete ortholattice and $T$ is a type ideal in $\mathbb{P}$ then so is $T_\mathrm{sub}$, $T_{<\kappa}$ and $T_\kappa$, for each cardinal $\kappa$.  Hence we get the type decompositions in \autoref{tdcor}, and the central $<\!\!\kappa$-$T$-subhomogeneous part is just the join of all the central $\lambda$-$T$-homogeneous parts, for $\lambda<\kappa$, i.e.
\[p_{T_{<\kappa}}=\bigvee_{\lambda<\kappa}p_{T_\lambda}\qquad\textrm{and}\qquad p_{T_\mathrm{sub}}=\bigvee_\lambda p_{T_\lambda}\]
Also $T\subseteq T_\kappa\subseteq T_{<\kappa^+}\subseteq T^\kappa$ (see \eqref{T^n}), so $\mathrm{c}T\subseteq\mathrm{c}T_\kappa\subseteq\mathrm{c}T_{<\kappa^+}\subseteq\mathrm{c}T^\kappa=\mathrm{c}T$, for all $\kappa$, and $\mathrm{c}T_\mathrm{sub}=\mathrm{c}T$, so $p_{T_\mathrm{sub}}\leq q_{T_\mathrm{sub}}=q_T$.  We aim to show that $p_{T_\mathrm{sub}}=q_T$ for suitable $T$.

Also note that the order of a homogeneous element is not unique in general.  For one thing, as long as $\mathbf{0}\in T$, then $\mathbf{0}$ is $\kappa$-$T$-homogeneous, for all $\kappa$.  However, this can also happen for non-zero elements, i.e. we can have $T_\lambda\cap T_\kappa\neq\{\mathbf{0}\}$ for $\lambda\neq\kappa$.

\begin{thm}\label{homthm}
If $T$ is an order-dense type ideal in a separative complete ortholattice $\mathbb{P}$ then $\mathbf{1}$ is $T$-subhomogeneous.
\end{thm}

\begin{proof}
By \eqref{tdprp2}, we may recursively define $t_\alpha\in T_\alpha$ so that $\mathrm{c}(t_\alpha)=\bigvee\mathrm{c}T_\alpha$, where $T_\alpha$ is the type ideal given by $T_\alpha=T\cap[(\bigvee_{\beta<\alpha}t_\beta)^\perp]$.  Let $s_\alpha=\mathrm{c}(t_\alpha)^\perp\wedge\bigwedge_{\beta<\alpha}\mathrm{c}(t_\beta)\in\mathrm{c}\mathbb{P}$, by \eqref{centresublattice}.  If $s=s_\alpha\wedge(\bigvee_{\beta<\alpha}t_\beta)^\perp\neq\mathbf{0}$ then, by the order density of $T$, we would have non-zero $t\in T\cap[s]\subseteq T_\alpha$.  This $t$ would then be very orthogonal to $t_\alpha$ so $t\vee t_\alpha\in T_\alpha$ and $\mathrm{c}(t\vee t_\alpha)>\mathrm{c}(t_\alpha)$, contradicting the definition of $t_\alpha$.  Thus $s_\alpha=s_\alpha\wedge\bigvee_{\beta<\alpha}t_\beta=\bigvee_{\beta<\alpha}(s_\alpha\wedge t_\beta)$.  Also, for each $\beta<\alpha$, \[\mathrm{c}(s_\alpha\wedge t_\beta)=s_\alpha\wedge\mathrm{c}(t_\beta)=s_\alpha,\] so each $s_\alpha$ is homogeneous.  As the $(t_\alpha)$ are orthogonal, they must be eventually $\mathbf{0}$.  Thus $(\bigvee s_\alpha)^\perp=\bigwedge\mathrm{c}(t_\alpha)=\mathbf{0}$ so the $(s_\alpha)$ witness the subhomogeneity of $\mathbf{1}$.
\end{proof}

\begin{cor}
If $T$ is an order-dense type ideal in a separative complete ortholattice $\mathbb{P}$ and $T_\lambda\cap T_\kappa\cap\mathrm{c}\mathbb{P}=\{\mathbf{0}\}$, whenever $\lambda\neq\kappa$, then there are unique orthogonal $(t_\kappa)\subseteq\mathrm{c}\mathbb{P}$ with $t_\kappa\in T_\kappa$, for all $\kappa$, and \[\mathbf{1}=\bigvee t_\kappa.\]
\end{cor}

\begin{proof}
Existence follows immediately from \autoref{homthm} by joining all resulting homogeneous elements of the same order.  Uniqueness now follows from \autoref{tdcor}, as $T_\lambda\cap T_\kappa\cap\mathrm{c}\mathbb{P}=\{\mathbf{0}\}$, for $\lambda\neq\kappa$, means that $t_\kappa=p_{T_\kappa}$, for all $\kappa$.
\end{proof}

\subsection{Type Relations}

\begin{dfn}\label{tr}
A relation $\sim$ on a complete ortholattice $\mathbb{P}$ is a \emph{type relation} if, whenever $(q_\alpha),(r_\alpha)\subseteq\mathbb{P}$ are such that $(q_\alpha\vee r_\alpha)$ are very orthogonal,
\begin{equation}\label{treq}
\forall\alpha(q_\alpha\sim r_\alpha)\qquad\Leftrightarrow\qquad\bigvee q_\alpha\sim\bigvee r_\alpha.
\end{equation}
We call $\sim$ \emph{proper} if $p\sim\mathbf{0}\Rightarrow p=\mathbf{0}$, for all $p\in\mathbb{P}$.
\end{dfn}

For any complete ortholattice $\mathbb{P}$, equality is a trivial example of a proper reflexive type relation, and it is of course the strongest reflexive relation on $\mathbb{P}$.  While 
\begin{equation}\label{cpleqcq}
\mathrm{c}(p)\leq\mathrm{c}(q)
\end{equation}
also defines a type relation, the weakest proper type relation.  For if $\precsim$ is such a relation then $p\precsim q$ implies $\mathrm{c}(q)^\perp\wedge p\precsim\mathrm{c}(q)^\perp\wedge q=\mathbf{0}$ so $p\leq\mathrm{c}(q)$ and hence $\mathrm{c}(p)\leq\mathrm{c}(q)$.  Thus \[\mathrm{c}(p)=\mathrm{c}(q)\] is the weakest symmetric proper type relation and this will coincide with equality if and only if $\mathbb{P}=\mathrm{c}\mathbb{P}$, i.e. if and only if $\mathbb{P}$ is a Boolean algebra, in which case they both represent the unique proper reflexive symmetric type relation on $\mathbb{P}$.  Slightly more interesting examples are given by perspectivity, orthoperspectivity and semiorthoperspectivity, which are type relations by \autoref{ordertd}.

Any type relation $\sim$ naturally defines other type relations $\precsim$ and $\precsim_\mathrm{rel}$ by
\begin{eqnarray}
p\precsim q &\Leftrightarrow& \exists r\in[q](p\sim r),\textrm{ and}\label{precsimdef}\\
p\precsim_\mathrm{rel} q &\Leftrightarrow& \exists r\in[q]_q(p\sim r).\label{precsimreldef}
\end{eqnarray}
In general $\precsim_\mathrm{rel}$ is stronger than $\precsim$, and they coincide when $\mathbb{P}$ is orthomodular, by \autoref{orthoequiv}.  Even in the non-orthomodular case, if $\sim$ is weaker than orthoperspectivity (on $\mathbb{P}^\sim$ at least), then $p\sim r\leq q$ implies $r\sim r^{\perp_q\perp_q}\leq q$ and hence, if $\sim$ is also transitive, $\precsim_\mathrm{rel}$ and $\precsim$ again coincide.

\begin{prp}\label{trti}
For any reflexive type relation on a separative complete ortholattice $\mathbb{P}$, the following subsets are type ideals.
\begin{enumerate}
\item\label{trtd1} The $\sim$-finite elements of $\mathbb{P}$.
\item\label{trtd2} The $\sim$-orthofinite elements of $\mathbb{P}$.
\end{enumerate}
\end{prp}

\begin{proof}\
\begin{itemize}
\item[\eqref{trtd1}] Say we have very orthogonal $(p_\alpha)\subseteq\mathbb{P}$.  If $\bigvee p_\alpha$ were not $\sim$-finite, we would have $\bigvee p_\alpha\sim\bigvee q_\alpha$ for some $(q_\alpha)$ with $q_\beta<p_\beta$, for some $\beta$, by \autoref{ordertd}.  As $\sim$ is a type relation, we would have $p_\beta\sim q_\beta$, and hence $p_\beta$ would not be $\sim$-finite.  On the other hand, if $p_\beta$ is not $\sim$-finite, for some $\beta$, then $p_\beta\sim q_\beta$ for some $q_\beta<p_\beta$ and then $\bigvee p_\alpha\sim\bigvee q_\alpha<\bigvee p_\alpha$, where $q_\alpha=p_\alpha$ for $\alpha\neq\beta$, and hence $\bigvee p_\alpha$ is not $\sim$-finite.
\item[\eqref{trtd2}] Essentially the same proof as in \eqref{trtd1}.
\end{itemize}
\end{proof}

\begin{thm}\label{permod}
If $\sim$ is a finite symmetric transitive type relation on a complete ortholattice $\mathbb{P}$ weaker than perspectivity then $\mathbb{P}$ is modular and $\sim$ coincides with $\sim_\mathrm{p}$.
\end{thm}

\begin{proof}
If $\mathbb{P}$ were not orthomodular, we would have $[p]\neq[p]_p$, for some $p\in\mathbb{P}$, and thus $q<q^{\perp_p\perp_p}$, for any $q\in[p]\setminus[p]_p$, even though $q^{\perp_p\perp_p}\sim q$, by \autoref{orthoperpequiv}, contradicting finiteness.  Perspectivity must also be finite, being weaker than $\sim$, and this comibined with orthomodularity means $\mathbb{P}$ is modular, by \eqref{perfin}.

As $\mathbb{P}$ is complete, it is a continuous geometry, by \cite{Kaplansky1955}.  Following the proof of \cite{Kaplansky1951} Theorem 6.6(c), we note that \cite{vonNeumann1960} Part III Theorem 2.7 means that perspectivity satisfies generalized comparison, i.e. for any $q,r\in\mathbb{P}$, we have $p\in\mathrm{c}\mathbb{P}$ (the centre we have defined is the same as the centre defined at the start of \cite{vonNeumann1960} Part III Chapter I, either by \autoref{centralequiv} \eqref{centralequiv3} or by \cite{vonNeumann1960} Part I Theorem 5.4 and \autoref{cenuniqcomp}) with $p\wedge q$ perspective to some $s\leq p\wedge r$ and $p^\perp\wedge r$ perspective to some $t\leq p^\perp\wedge q$.  So if $q\sim r$ then $s\sim p\wedge q\sim p\wedge r$ which, by finiteness, means $s=p\wedge r$ and, likewise, $t=p^\perp\wedge q$.  So $p\wedge q$ and $p\wedge r$ are perspective, as are $p^\perp\wedge q$ and $p^\perp\wedge r$ so, by \autoref{centralequiv} \eqref{centralequiv3}, $q$ and $r$ are also perspective.
\end{proof}

\subsection{Type Classes}

\begin{dfn}\label{tc}
Let $\mathbf{C}$ be the class of complete lattices with a relation $\perp$.  We call $\mathbf{T}\subseteq\mathbf{C}$ a \emph{type class} if it is closed under isomorphisms and, for $(\mathbb{P}_\alpha)\subseteq\mathbf{T}$,
\begin{equation}\label{tcprod}
(\mathbb{P}_\alpha)\subseteq\mathbf{T}\qquad\Leftrightarrow\qquad\prod\mathbb{P}_\alpha\in\mathbf{T}
\end{equation}
\end{dfn}

\begin{prp}\label{tctd}
If $\mathbf{T}$ is a type class and $\mathbb{P}$ is a separative complete ortholattice then the following subsets of $\mathbb{P}$ are type ideals.
\begin{eqnarray}\label{tctdeq}
\mathbb{P}_\mathbf{T} &=& \{p\in\mathbb{P}:[p]\in\mathbf{T}\}.\label{tctd1}\\
\mathbb{P}_{\mathbf{T}\mathrm{rel}} &=& \{p\in\mathbb{P}:[p]_p\in\mathbf{T}\}.\label{tctd2}
\end{eqnarray}
\end{prp}

\begin{proof}
$\mathbb{P}_\mathbf{T}$ and $\mathbb{P}_{\mathbf{T}\mathrm{rel}}$ are type ideals by \eqref{ordertdeq} and \eqref{bigveepalpha} respectively.
\end{proof}

Any subclass of $\mathbf{C}$ consisting of all those complete lattices satisfying some universally quantified sentence in the language $\{\mathbf{0},\leq,\perp,\wedge,\vee\}$ will be a type class, like the following important examples \textendash
\begin{eqnarray}
\mathbf{D} &=& \{\mathbb{P}\in\mathbf{C}:\mathbb{P}\textrm{ is a distributive}\},\label{Ddef}\\
\mathbf{M} &=& \{\mathbb{P}\in\mathbf{C}:\mathbb{P}\textrm{ is a modular}\},\textrm{ and}\label{Mdef}\\
\mathbf{O} &=& \{\mathbb{P}\in\mathbf{C}:\mathbb{P}\textrm{ is a orthomodular}\}.\label{Odef}
\end{eqnarray}
(note that \autoref{orthoequiv}\eqref{orthoequiv5} characterizes orthomodularity just with an orthogonality relation $\perp$, rather than an orthocomplementation ${}^\perp$, which we may not have for $[p]$).  We also have the type class $\mathbf{S}=\{\mathbb{P}\in\mathbf{C}:\mathbb{P}\textrm{ is a separative}\}$, although this does not lead to any interesting type decompositions in a separative complete ortholattice $\mathbb{P}$ because then $\mathbb{P}=\mathbb{P}_\mathbf{S}$ and, even though we may have $\mathbb{P}\neq\mathbb{P}_{\mathbf{S}\mathrm{rel}}$ (see the comments after \autoref{orthoequiv}), we will still have $\mathrm{c}\mathbb{P}\subseteq\mathbb{P}_{\mathbf{S}\mathrm{rel}}$.  Now, using the type ideals and type decompositions coming from several instances of \eqref{tctd1} and \eqref{tdcor2} respectively, we can define
\begin{eqnarray}
p_\mathrm{I} &=& q_{\mathbb{P}_\mathbf{D}},\label{pI}\\
p_\mathrm{II} &=& q^\perp_{\mathbb{P}_\mathbf{D}}\wedge q_{\mathbb{P}_\mathbf{M}},\label{pII}\\
p_\mathrm{III} &=& q^\perp_{\mathbb{P}_\mathbf{M}}\wedge q_{\mathbb{P}_\mathbf{O}},\textrm{ and}\label{pIII}\\
p_\mathrm{IV} &=& q^\perp_{\mathbb{P}_\mathbf{O}}.\label{pIV}
\end{eqnarray}

If $\mathbb{P}$ is the projection lattice $\mathcal{P}(A)$ of a von Neumann algebra $A$ then these projections do indeed correspond to those you get from the classical von Neumann algebra type decomposition.  Specifically, to see that $p_\mathrm{I}A$ corresponds to the type I part, see the comments after \autoref{pq=qp}, and to see that this even extends to annihilators in C*-algebras, see \autoref{commuteBoolean}.  Also $p_\mathrm{II}A$ corresponds the classical type II part because a von Neumann algebra is finite if and only if its projection lattice is modular (see \cite{Berberian1972} \S34 Proposition 1 and \cite{Kaplansky1955} Theorem on page 1).  This fact can also be extended, at least partially, to annihilators in C*-algebras, by \autoref{permod} and \autoref{simtr}.  And $\mathcal{P}(A)$ is necessarily orthomodular, which means $p_\mathrm{IV}=0$, so $p_\mathrm{III}A$ also corresponds to the type III part in this case.  We do not know if the same is true for the annihilators in a C*-algebra, i.e. whether there exist any type IV C*-algebras at all.  The annihilators in such a C*-algebra would be wildly different from any ortholattices seen before in operator algebras, as they would fail to be orthomodular in a very strong way.

We can also use \autoref{tdcor} and \eqref{T^n}, to define
\begin{eqnarray*}
p_\mathrm{I_n} &=& p_{\mathbb{P}_{\mathbf{D}n}},\textrm{ and}\\
p_\mathrm{II_1} &=& q^\perp_{\mathbb{P}_\mathbf{D}}\wedge p_{\mathbb{P}_\mathbf{M}}.
\end{eqnarray*}
If $\mathbb{P}$ is again the projection lattice $\mathcal{P}(A)$ of a von Neumann algebra $A$ then $p_\mathrm{I_n}A$ and $p_\mathrm{II_1}A$ are again the type $\mathrm{I_n}$ and $\mathrm{II_1}$ part respectively in the classical von Neumann type decomposition.  In this case, a supremum of finite projections is again finite, which means that $p_\mathrm{I_n}\leq p_{\mathbb{P}_\mathbf{M}}$ and $\mathbb{P}^n_\mathbf{M}=\mathbb{P}_\mathbf{M}$.  In an arbitrary separative complete ortholattice $\mathbb{P}$ we do not have any notion of Murray-von Neumann equivalence and so there is no reason to believe these facts remain true in general.  For annihilators in C*-algebras we will, however, define an analgous notion (see \S\ref{Equivalence}) which will enable us to prove the first of these facts (see \autoref{nsubcor}\eqref{nsubcor3}).

We could have also used the relative type-ideals in \eqref{tctd2}, rather than those in \eqref{tctd1}, to define the above type classes, as given below.
\begin{eqnarray}
p_{\mathrm{I}\mathrm{rel}} &=& q_{\mathbb{P}_{\mathbf{D}\mathrm{rel}}},\label{pIp}\\
p_{\mathrm{II}\mathrm{rel}} &=& q^\perp_{\mathbb{P}_{\mathbf{D}\mathrm{rel}}}\wedge q_{\mathbb{P}_{\mathbf{M}\mathrm{rel}}},\label{pIIp}\\
p_{\mathrm{III}\mathrm{rel}} &=& q^\perp_{\mathbb{P}_{\mathbf{M}\mathrm{rel}}}\wedge q_{\mathbb{P}_{\mathbf{O}\mathrm{rel}}},\textrm{ and}\label{pIIIp}\\
p_{\mathrm{IV}\mathrm{rel}} &=& q^\perp_{\mathbb{P}_{\mathbf{O}\mathrm{rel}}}.\label{pIVp}
\end{eqnarray}
If $\mathbb{P}$ is the projection lattice in a von Neumann algebra then, as this is orthomodular, we have $[p]=[p]_p$ (see \autoref{orthoequiv} \eqref{orthoequiv3}) and this decomposition is exactly the same as the previous one.  But for annihilators in C*-algebras, it may well be different, although we do at least know that $p_{\mathrm{I}\mathrm{rel}}=p_\mathrm{I}$ in this case, by \autoref{commuteBoolean}.  Indeed, these relative versions using $[p]_p$ are really more natural for annihilators, as using $[p]$ instead might result in the rather awkward situation that an annihilator $B$ could be modular in the larger C*-algebra $A$ even though it is not modular in itself (i.e. $[B]$ could be modular lattice even when $[B]_B$ is not), or vice versa.  Although we could avoid this problem by using the even smaller type-ideal of $\mathbb{P}$ given by \[\{p\in\mathbb{P}:[p]_p=[p]\in\mathbf{M}\}\subseteq\mathbb{P}_{\mathbf{M}\mathrm{rel}}\cap\mathbb{P}_\mathbf{M}.\]  In fact, this is just the intersection of $\mathbb{P}_\mathbf{M}$ (or $\mathbb{P}_{\mathbf{M}\mathrm{rel}}$) with the equality type-ideal $\mathbb{P}_=$ defined in \eqref{eqtypeideal}.  Taking intersections of the various resulting type-ideals we get from these type-classes would yield more type-ideals leading to even finer type decompositions.  Although, whether all these potential types are actually realized by certain C*-algebras, we do not know.

\subsection{Boolean Elements}

The following terminology comes from \cite{Chevalier1991}

\begin{dfn}
An ortholattice $\mathbb{P}$ has the \emph{relative centre property} if, for all $p\in\mathbb{P}$, \[\mathrm{c}_p[p]_p=\{p\wedge q:q\in\mathrm{c}\mathbb{P}\}.\]
\end{dfn}

This is equivalent to saying $\mathrm{c}(q)\wedge p=\mathrm{c}_p(q)$ whenever $q\in[p]_p$.

\begin{dfn}\label{gcdef}
A relation $\precsim$ on a complete ortholattice $\mathbb{P}$ satisfies \emph{generalized comparison} if, for any $q,r\in\mathbb{P}$, there exists $p\in\mathrm{c}\mathbb{P}$ with \[p\wedge q\precsim p\wedge r\quad\textrm{and}\quad p^\perp\wedge r\precsim p^\perp\wedge q.\]
\end{dfn}

\begin{prp}\label{simgcprp}
If $\sim$ is a symmetric type relation and $\precsim_\mathrm{rel}$ is defined as in \eqref{precsimreldef} then $\precsim_\mathrm{rel}$ satisfies generalized comparison if and only if, for all $q,r\in\mathbb{P}$,
\begin{equation}\label{simgc}
\exists u\in[q]_q\exists v\in[r]_r(u\sim v\textrm{ while }q\wedge u^\perp\textrm{ is very orthogonal to }r\wedge v^\perp).
\end{equation}
\end{prp}

\begin{proof}
If $\precsim_\mathrm{rel}$ satisfies generalized comparison then, for any $q,r\in\mathbb{P}$ we have $p\in\mathrm{c}\mathbb{P}$, $s\in[p\wedge r]_r$, $t\in[p^\perp\wedge q]_q$ such that $p\wedge q\sim s$ and $p^\perp\wedge r\sim t$.  If $\sim$ is a symmetric type relation then $u=(p\wedge q)\vee_qt=(p\wedge q)\vee t\sim s\vee(p^\perp\wedge r)=s\vee_r(p^\perp\wedge r)=v$ while $q\wedge u^\perp\leq p^\perp$ and $r\wedge v^\perp\leq p$.

Conversely, given such a $u$ and $v$, setting $p=\mathrm{c}(r\wedge v^\perp)$ gives $p\wedge q\leq u^{\perp_q\perp_q}=u$ and hence \[p\wedge q=p\wedge u\sim p\wedge v\leq p\wedge r\] and, likewise, $p^\perp\wedge r=p^\perp\wedge v\sim p^\perp\wedge u\leq p^\perp\wedge q$.
\end{proof}


As noted after \autoref{tr}, $p\sim q\Leftrightarrow c(p)=c(q)$ defines a (proper reflexive symmetric) type relation on any complete ortholattice $\mathbb{P}$.  To see that $\precsim_\mathrm{rel}$ then satisfies generalized comparison, simply take any $q,r\in\mathbb{P}$, set $u=q\wedge c(r)$ and $v=r\wedge c(q)$ and note that $c(u)=c(q)\wedge c(r)=c(v)$ while $q\wedge u^\perp\leq c(r)^\perp$ and $r\wedge v^\perp\leq c(q)^\perp$.  The next result shows that, under suitable extra hypotheses, it is the only such relation on $\mathbb{P}_{\mathbf{D}\mathrm{rel}}$.  In fact it shows that, if $\sim$ is a proper type relation on a complete Boolean algebra $\mathbb{P}$ then generalized comparison is equivalent to reflexivity, i.e. in this case generalized comparison just says $\sim$ is equality.

\begin{thm}\label{Boolthm}
If $\mathbb{P}$ is a complete ortholattice with the relative centre property, $\sim$ is a proper symmetric type relation and $\precsim_\mathrm{rel}$ satisfies generalized comparison then
\begin{enumerate}
\item\label{Boolthm1} For $p\in\mathbb{P}_{\mathbf{D}\mathrm{rel}}$ and $q\in\mathbb{P}$, $\mathrm{c}(p)\leq\mathrm{c}(q)\Leftrightarrow p\precsim_\mathrm{rel}q$.
\item\label{Boolthm2} For $p,q\in\mathbb{P}_{\mathbf{D}\mathrm{rel}}$, $\mathrm{c}(p)=\mathrm{c}(q)\Leftrightarrow p\sim q$.
\item\label{Boolthm3} When $p,q\in\mathbb{P}_{\mathbf{D}\mathrm{rel}}$ and $p\sim q$, $\sim$ is an orthoisomorphism on $[p]_p\times[q]_q$.
\end{enumerate}
\end{thm}

\begin{proof}\
\begin{enumerate}
\item The $\Leftarrow$ part is immediate from the comments after \eqref{cpleqcq}.  For the $\Rightarrow$ part, note that, by generalized comparison, we have $s\in[p]_p$ and $t\in[q]_q$ with $s\sim t$.  As $\mathbb{P}$ has the relative centre property and every element of a Boolean algebra is central, we have $\mathrm{c}(p\wedge s^\perp)\wedge p=\mathrm{c}_p(p\wedge s^\perp)=p\wedge s^\perp$.  As $\sim$ is a type relation, $\mathbf{0}=\mathrm{c}(p\wedge s^\perp)\wedge s\sim \mathrm{c}(p\wedge s^\perp)\wedge t$.  But $\sim$ is also proper and hence $\mathrm{c}(p\wedge s^\perp)\wedge t=\mathbf{0}$ so $\mathrm{c}(p\wedge s^\perp)\wedge q=\mathbf{0}$ and $\mathrm{c}(p\wedge s^\perp)\wedge\mathrm{c}(q)=\mathbf{0}$.  As $\mathrm{c}(p)\leq\mathrm{c}(q)$, we must have $p\wedge s^\perp=\mathbf{0}$ and therefore $p=s$, i.e. $p\sim t\leq q$.
\item A symmetric argument yields $q=t$ too, i.e. $p\sim q$.
\item Say we had $r,s\in[p]_p$ and $t\in[q]_q$ with $r\sim t$ and $s\sim t$.  We would then have $\mathrm{c}(r)=\mathrm{c}(t)=\mathrm{c}(s)$ and hence $r=\mathrm{c}_p(r)=\mathrm{c}(r)\wedge p=\mathrm{c}(s)\wedge p=\mathrm{c}_p(s)=s$.  Also, for any $s\in[p]_p$, $\mathrm{c}(\mathrm{c}(s)\wedge q)=\mathrm{c}(s)\wedge\mathrm{c}(q)=\mathrm{c}(s)\wedge\mathrm{c}(p)=\mathrm{c}(s)$ and hence $s\sim\mathrm{c}(s)\wedge q\in[q]_q$.  Thus $\sim$ restricted to $[p]_p\times[q]_q$ is the function $s\mapsto\mathrm{c}(s)\wedge q$.  It is clearly order preserving and a symmetric argument shows that the same is true of the inverse function $t\mapsto\mathrm{c}(t)\wedge p$.  Finally note that, as $[p]_p$ and $[q]_q$ are Boolean algebras, any order isomorphism is actually an orthoisomorphism.
\end{enumerate}
\end{proof}

\section{Annihilators and Projections}\label{AvsP}

\subsection{Annihilators}\label{annsec}

Throughout, $A$ denotes a C*-algebra with positive elements $A_+=\{aa^*:a\in A\}$, self-adjoint elements $A_\mathrm{sa}=\{a=a^*:a\in A\}$ (or $A_\mathrm{sa}=\{a-b:a,b\in A_+\}$), unit ball $A^1=\{a:||a||\leq1\}$ and projections $\mathcal{P}(A)=\{a\in A:a=a^*=a^2\}$, where $p\leq q$ means $p=pq$, for $p,q\in\mathcal{P}(A)$.

Consider the following relations on $A$.
\begin{eqnarray*}
aLb &\Leftrightarrow& ab^*=0.\\
a\bot b &\Leftrightarrow& aLb\textrm{ and }aLb^*.\\
a\top b &\Leftrightarrow& a\bot b\textrm{ and }a^*\bot b.
\end{eqnarray*}
Taking the orthocompletion of $A$ w.r.t. any of these relations amounts to the same thing.  To see why, first note that $\{a\}^\perp=\{a^*a\}^\top$ and $\{a\}^\top=\{a,a^*\}^\perp$ so \[[A]^\perp=[A]^\top\subseteq\{S\subseteq A:S=S^*\}\] and, for any $S\subseteq A$ with $S=S^*(=\{s^*:s\in S\})$, we have $S^\perp=S^\top$.  As $\top$ is a preorthogonality relation, it follows from \S\ref{TheCompletion} that \[[A]^\perp\textrm{ is a complete ortholattice.}\]  In fact, all we have used here is the fact that $A$ is a *-ring with proper involution.  Also $L$ is a preorthogonality relation on $A$ and the map \[B\mapsto B\cap B^*\] is an orthoisomorphism from $[A]^L$ to $[A]^\perp$.  Every element of $[A]^L$ is clearly a closed left ideal and so every element of $[A]^\perp$ is a hereditary C*-subalgebra of $A$ (see \cite{Pedersen1979} Theorem 1.5.2).  As C*-algebras, rather than their left ideals, are our primary object of study, and the equivalence in \S\ref{Equivalence} is slightly easier to define with $\bot$ rather than $\top$, we shall focus on the relation $\bot$.  The elements of $[A]^\perp$ will be called \emph{annihilators} of $A$.

Another thing to note is that the above relations all agree on $A_\mathrm{sa}$.  The fact that $\{a\}^\perp=\{a^*a\}^\perp$, shows that $A_+$ is join-dense in $A$, w.r.t. to the preorder induced by $\perp$, and thus \[[A]^\perp\cong[A_+]^\perp.\]  Indeed, it will often be convenient in proofs to use the elements of $A_+$ rather than $A$.  If $A$ has real rank zero, then every hereditary C*-subalgebra contains an approximate unit of projections, so every annihilator will be an annihilator of subset of projections, i.e. $\mathcal{P}(A)$ will be join-dense in $A$ and \[[A]^\perp\cong[\mathcal{P}(A)]^\perp.\]  We generalize this observation in \autoref{SP}.

One other thing worth pointing out is that we have something extra on $[A]^\perp$ that we do not have for an arbitrary ortholattice.  Specifically, we actually have a function $||\cdot\cdot||$ from $[A]^\perp\times[A]^\perp$ to $[0,1]$ which quantifies the degree of orthogonality of $B,C\in[A]^\perp$ given by
\begin{equation}\label{||BC||}
||BC||=\sup_{b\in B^1_+,c\in C^1_+}||bc||.
\end{equation}
The important properties of this function are that, for $B,C,D\in[A]^\perp$,
\begin{eqnarray*}
||BC|| &=& ||CB||,\\
B\neq\{0\} &\Leftrightarrow& ||BB||=1,\\
B\perp C &\Leftrightarrow& ||BC||=0,\textrm{ and}\\
C\subseteq D &\Rightarrow& ||BC||\leq||BD||.
\end{eqnarray*}
Indeed, these properties could be derived from the relevant properties of $||\cdot\cdot||$ on $A^1_+$ and the the fact that $a\perp b\Leftrightarrow||ab||=0$, for $a,b\in A^1_+$ (and $[A]^\perp\cong[A^1_+]^\perp$).  If, further, $||BC^\perp||$ satisfies the triangle inequality, i.e. for all $B,C,D\in[A]^\perp$, \[||BD^\perp||\leq||BC^\perp||+||CD^\perp||,\] then we naturally call $||\cdot\cdot||$ an \emph{orthonorm}.  We do not know if \eqref{||BC||} always yields an orthonorm on $[A]^\perp$, although we show it does for certain C*-algebras in \autoref{nsubcor}\eqref{nsubcor2}.  We can even use $||\cdot\cdot||$ to define the \emph{orthospectrum} of $B,C\in[A]^\perp$ by
\begin{equation*}
\sigma(BC)=\overline{\{||BD||^2:D\in[C]\textrm{ and }((D\vee B^\perp)\wedge B)\perp((D^{\perp_B}\vee B^\perp)\wedge B)\}},
\end{equation*}
except when $B=C=A$, in which case we define $\sigma(AA)=\{1\}$ (see \autoref{orthospecpq}).

\subsection{Non-Commutative Topology}\label{NCT}

From this point on, it will be convenient to assume that $A$ is concretely represented faithfully and non-degenerately on some Hilbert space $H$.  This is, of course, valid, thanks to the standard GNS construction.  One canonical choice would be the universal representation, which has the advantage that its projections distinguish all hereditary C*-subalgebras of $A$.  The same is also true for the atomic representation, by by \cite{Pedersen1979} Proposition 4.3.13 and Theorem 4.3.15.  However, we are primarily concerned with annihilators, and the projections in any faithful representation of $A$ still distinguish the annihilators (see \autoref{annproj}).  So, as long as we fix it throughout, any faithful non-degenerate representation will do.

However, we will still use standard non-commutative topology terminology (see \cite{Akemann1969}, \cite{Akemann1970} and \cite{Pedersen1979} 3.11.10) which, traditionally, has only be used with reference to projections in $A^{**}$.  This restriction is certainly convenient for many of the proofs, and necessary for some of the results, but many of the results themselves are valid also for arbitrary representations, thanks to the following property of $A^{**}$.

\begin{thm}\label{A**}
Any representation $\pi$ of $A$ has a normal extension to $A^{**}$.
\end{thm}

\begin{proof}
See \cite{Pedersen1979} Theorem 3.7.7.
\end{proof}

\begin{dfn}
We call $p\in\mathcal{P}(\mathcal{B}(H))$ \emph{open} if $p=\sup(p_\alpha)$, for some increasing net $(p_\alpha)\subseteq A^1_+$, and \emph{closed} if $p^\perp$ is open.
\end{dfn}

Equivalently, a projection $p$ is closed if and only if $p=\inf(p_\alpha)$ for some decreasing net $(p_\alpha)\subseteq 1-A^1_+$.  And the supremum/infimum of an increasing/decreasing net in $A$ coincides with the limit of that net in the strong (or weak) topology, so open and closed projections always lie in the double commutant $A''$ of $A$.  We will denote the sets of open and closed projections by $\mathcal{P}(A'')^\circ$ and $\overline{\mathcal{P}(A'')}$ respectively.  Also note that the identity operator $1$ is open, as we are dealing with a non-degenerate representation.

If $B$ is a C*-subalgebra of $A$, then $\{b\in B_+:||b||<1\}$ is directed (see \cite{Pedersen1979} Theorem 1.4.2) and hence has a supremum $p_B=\sup(B^1_+)\in\mathcal{P}(A'')^\circ$.  Conversely, if $p\in\mathcal{P}(A'')^\circ$, then $B=pAp\cap A$ is a hereditary C*-subalgebra of $A$ with $p=\sup(B^1_+)=p_B$. So a projection is open precisely when it is the supremum of $B^1_+$ for some hereditary C*-subalgebra $B$ of $A$.\footnote{Incidentally, the closed projections of $A$ can similarly be characterized as the infimums of norm filters (see \cite{Bice2011} Corollary 3.4), although we will not have further occasion to refer to these.}  The open and closed projections can also be characterized as the limits of increasing and decreasing sequences respectively in $\widetilde{A}^1_+=(A+\mathbb{C}1)^1_+$ (or even $\widetilde{A}^1_\mathrm{sa}$), a surprisingly non-trivial fact when $A$ is not unital (and hence $A\neq\widetilde{A}$).  Another important fact is that the collection of open projections $\mathcal{P}(A'')^\circ$ is norm closed (see \cite{Pedersen1979} Proposition 3.11.9).

\begin{dfn}
We define the \emph{interior} $p^\circ$ of $p\in\mathcal{P}(\mathcal{B}(H))$ by $p^\circ=\sup(pAp\cap A)^1_+$ and the \emph{closure} by $\overline{p}=p^{\perp\circ\perp}$.  We call $p$ \emph{topologically regular}\footnote{As far as we know, such projections have not been considered or named before.  They are the analog of regular open subsets of a topological space, although we are averse to simply calling them regular, as this is already a standard term meaning something different (see \cite{Akemann1969} Definition II.11 and the discussion at the start of \S\ref{SvsO}).} if $p=\overline{p}^\circ$ and denote the collection of all topologically regular open projections by $\overline{\mathcal{P}(A'')}^\circ$.
\end{dfn}

It follows that $p^\circ$ is the largest open projection satisfying $p^\circ\leq p$ and $\overline{p}$ is the smallest closed projection satisfying $p\leq\overline{p}$.  The existence of such projections also follows from the fact that the supremum of a collection of open projections is open and the infimum of a collection of closed projections is closed (see \cite{Akemann1969} Proposition II.5 and combine it with \autoref{A**} to obtain the result for arbitrary representations).  Also note that the interior of \emph{any} closed projection is in fact topologically regular.  For if $p=q^\circ$, for some closed $q\in\mathcal{P}(\mathcal{B}(H))$, then $\overline{p}\leq q$ so
\begin{equation}\label{intclosedtopreg}
\overline{p}^\circ\leq q^\circ=p\leq\overline{p}^\circ.
\end{equation}

One more important thing to note is the difference between complements of annihilators and their corresponding projections.  Specifically, given $B\in[A]^\perp$, we may have $p_{B^\perp}\neq p_B^\perp$, and one could view the complications of extending projection results to annihilators as all stemming from this fact.  Indeed, $p_{B^\perp}$ is open while $p_B^\perp$ is closed, so they could not be equal unless they were clopen.  This occurs when these projections, or their complements, lie in $A$ itself, i.e. if $B=pAp$ for some $p\in A$.  In fact, if $A''=A^{**}$ (i.e. if we are dealing with the universal representation of $A$) and $1\in A$ then the clopen projections are precisely those in $A$ (see \cite{Akemann1969} Proposition II.18).  However, $B^\perp=p_B^\perp Ap_B^\perp\cap A$ so, by definition, we do always have
\begin{equation}\label{perpseq}
p_{B^\perp}=p_B^{\perp\circ}.
\end{equation}

\begin{thm}\label{annproj}
There are order isomorphisms between $[A]^\perp$ and $\overline{\mathcal{P}(A'')}^\circ$ given by \[B\mapsto p_B\qquad\textrm{and}\qquad p\mapsto pAp\cap A\]
\end{thm}

\begin{proof}
For $B\in[A]^\perp$, $p_B=p_{B^{\perp\perp}}=p_{B^\perp}^{\perp\circ}$, by \eqref{perpseq}.  But $p_{B^\perp}=\sup(B^\perp)^1_+$ is open so $p^\perp_{B^\perp}$ is closed and hence its interior, $p_B$, is a topologically regular projection.  Also, if $p_B=p=p_C$, for $B,C\in[A]^\perp$, then $B^\perp=p^\perp Ap^\perp\cap A=C^\perp$ and hence $B=B^{\perp\perp}=C^{\perp\perp}=C$, so the first map is injective.

For topologically regular $p$, \[(p^{\perp\circ}Ap^{\perp\circ}\cap A)^\perp=p^{\perp\circ\perp}Ap^{\perp\circ\perp}\cap A=p^{\perp\circ\perp\circ}Ap^{\perp\circ\perp\circ}\cap A=\overline{p}^\circ A\overline{p}^\circ\cap A=pAp\cap A,\] so $pAp\cap A$ is indeed an annihilator.  Also $p_{pAp\cap A}=\sup(pAp\cap A)^1_+=p$, as $p$ is open.  This, combined with the injectivity of the first map, shows that these maps are indeed inverse to each other which, as they are immediately seen to be order preserving, means they are also order isomorphisms.
\end{proof}

When making order theoretic statements about projections we must now always take care to note whether they are with respect to $\mathcal{P}(A'')$ or $\overline{\mathcal{P}(A'')}^\circ$, and we shall adopt the convention that, by default, it is the order structure of $\mathcal{P}(A'')$ being referred to unless otherwise specified, with subscripts for example.  For while it follows from \autoref{annproj} that $\overline{\mathcal{P}(A'')}^\circ$ is a complete lattice, it is not a sublattice of $\mathcal{P}(A'')$.  When $p,q\in\overline{\mathcal{P}(A'')}^\circ$ satisfy $pq=qp$, we do have $p\wedge q=p\wedge_{\overline{\mathcal{P}(A'')}^\circ}q=pq$, thanks to \cite{Akemann1969} Theorem II.7, so infimums do at least agree in this case, but even commutativity does not imply that supremums agree.  Moreover, as mentioned above, the orthocomplement functions in the two structures are different.

However, there are relations between some order theoretic concepts in $\mathcal{P}(A'')$ and the corresponding concepts in $\overline{\mathcal{P}(A'')}^\circ$.  For example, the concept of centrality coincides in the two structures (see \autoref{centralannihilators} and the comments that follow) and the following result shows that commutativity is usually stronger in $\mathcal{P}(A'')$ than in $\overline{\mathcal{P}(A'')}^\circ$ (and it can be strictly stronger \textendash\, see the comments after \autoref{commuteclosure2}).

\begin{prp}\label{commutativityimplication}
If $\overline{\mathcal{P}(A'')}^\circ$ is orthomodular then, for $p,q\in\overline{\mathcal{P}(A'')}^\circ$, 
\[\exists r\in\{p,\overline{p},p^\perp,\overline{p}^\perp\}\exists s\in\{q,\overline{q},q^\perp,\overline{q}^\perp\}(rs=sr)\quad\Rightarrow\quad p\textrm{ and }q\textrm{ commute in }\overline{\mathcal{P}(A'')}^\circ.\]
\end{prp}

\begin{proof}
For any projections $p$ and $q$ in a C*-algebra, $pq=qp\Leftrightarrow pq^\perp=q^\perp p$, and any $p$ and $q$ in an orthomodular lattice commute if and only if $p$ and $q^\perp$ commute.  Thus, without loss of generality, we may assume that $pq=qp$ and hence $pq=p\wedge_{\overline{\mathcal{P}(A'')}^\circ}q$, by the comments above.  As $\overline{\mathcal{P}(A'')}^\circ$ is orthomodular, we have $r\in\overline{\mathcal{P}(A'')}^\circ$ with $pqr=0$, $r\leq p$ and \[p=pq\vee_{\overline{\mathcal{P}(A'')}^\circ}r=\overline{pq\vee r}^\circ.\]  But then $p^\perp qr=qp^\perp r=0$ and hence $qr=0$.  So $r\leq p\wedge_{\overline{\mathcal{P}(A'')}^\circ}q^{\perp_{\overline{\mathcal{P}(A'')}^\circ}}$ and hence \[p=pq\vee_{\overline{\mathcal{P}(A'')}^\circ}r\leq(p\wedge_{\overline{\mathcal{P}(A'')}^\circ}q)\vee_{\overline{\mathcal{P}(A'')}^\circ}(p\wedge_{\overline{\mathcal{P}(A'')}^\circ}q^{\perp_{\overline{\mathcal{P}(A'')}^\circ}})\leq p.\]
As $\overline{\mathcal{P}(A'')}^\circ$ is orthomodular, we are done, by the comments after \autoref{centredef}.
\end{proof}

In fact, the following result shows that orthomodularity itself is only an issue when $q<p$ and $p\overline{q}\neq\overline{q}p$ (which is indeed possible, by \autoref{commuteclosure1}).  Combining this argument, the proof of \autoref{commutativityimplication} and \cite{Kalmbach1983} \S3 Lemma 3, we get that, even without orthomodularity, if $p,q\in\overline{\mathcal{P}(A'')}^\circ$ then
\[\forall r,s\in\{p,\overline{p}^\perp,q,\overline{q}^\perp\}(rs=sr\textrm{ and }r\overline{rs}=\overline{rs}r)\quad\Rightarrow\quad p\textrm{ and }q\textrm{ commute in }\overline{\mathcal{P}(A'')}^\circ.\]

\begin{prp}
If $p,q\in\overline{\mathcal{P}(A'')}^\circ$, $q<p$ and $p\overline{q}=\overline{q}p$ then $p=q\vee_{\overline{\mathcal{P}(A'')}^\circ}\overline{q}^\perp p$.
\end{prp}

\begin{proof}
Note $\overline{q\vee\overline{q}^\perp p}\geq\overline{q}\vee\overline{q}^\perp p\geq\overline{q}p\vee\overline{q}^\perp p=p$ so $p\leq\overline{q\vee\overline{q}^\perp p}^\circ=q\vee_{\overline{\mathcal{P}(A'')}^\circ}\overline{q}^\perp p$.
\end{proof}

We also note some elementary facts about ideals and their associated projections.

\begin{prp}\label{centralideals}
$p\in\mathcal{P}(A'')^\circ$ is central if and only if $pAp\cap A$ is an ideal.
\end{prp}

\begin{proof}
If $pAp\cap A$ is an ideal in $A$ then its weak closure is an ideal in $A''$ containing $p$ as its unit.  Thus, for any $a\in A$ we have $ap=pap=pa$ so $p\in A'$.  On the other hand, if $p$ is central then, for any $a\in A$ and $b\in pAp\cap A$, we have $ab=apbp=pabp\in pAp\cap A$ and $ba=pbpa=pbap\in pAp\cap A$.
\end{proof}

\begin{prp}\label{centralprojections}
If $p\in\mathcal{P}(A'')$ is central then so are both $\overline{p}$ and $p^\circ$.
\end{prp}

\begin{proof}
If $a\in A$ and $ap=0$ then $abp=apb=0=bap$, for any $b\in A$, so $\{p\}^\perp$ is an ideal.  Thus $p_{\{p\}^\perp}=p^{\perp\circ}$ is central, by \autoref{centralideals}.  As $p$ is central if and only if $p^\perp$ is, we are done.
\end{proof}


In particular, if $I$ is an ideal in $A$ then $p_{I^\perp}=\overline{p}_I^\perp$ is central and hence $I^\perp$ is also an ideal (although this can also be proved by elementary algebraic means).

Of course, in the commutative case, all the concepts in this subsection correspond to their usual topological counterparts.  Specifically, if $A=C_0(X)$ for some locally compact topological space then the atomic representation identifies every element of $C_0(X)$ with the multiplication operator on $l^2(X)$ it defines.  Then $A''$ is the set of all bounded multiplicaiton operators on $l^2(X)$, which can naturally be identified with $l^\infty(X)$ in the same way.  Under this identification, projections are just characteristic functions $\chi_S$ of subsets $S$ of $X$, where $\chi_S$ is open, closed or topologically regular if and only if $S$ is, in the topology of $X$ (and note that the characteristic function of an open (closed) set is lower (upper) semicontinuous, a fact which will be generalized in \S\ref{MVF}).  In particular, any open subset $S$ of $X$ that is open but not (topologically) regular corresponds to an open projection $\chi_S$ that is not topologically regular, which itself corresponds to a hereditary C*-subalgebra that is not an annihilator, contradicting \cite{Arzikulov2013} Lemma 16.1 (with $e=f=\chi_S$).  For example, we could have $X=[-1,1]$ and $S=[-1,0)\cup(0,1]$.

Incidentally, the atomic representation will usually not be the same as the universal representation, even in the commutative case considered in the previous paragraph, illustrating that the universal representation may not always be the best to work with.  Above, we could also consider the subrepresentation on $l^2(Y)$, for some $Y\subseteq X$, which will still be faithful as long as $Y$ is dense in $X$.  This is sometimes nicer in some sense, for example if $A=C^b(\mathbb{N})\cong C(\beta\mathbb{N})$ where $\beta\mathbb{N}$ is the Stone-\v{C}ech compactification of $\mathbb{N}$, we can consider the faithful subrepresentation on $l^2(\mathbb{N})$, which has the advantage that all projections in $A''$ are clopen (as all subsets of $\mathbb{N}$ are).  These facts provide some justification for our choice to work with arbitrary representations.

\subsection{Spectral Projections}\label{specsec}

Some quite useful open and closed projections are the spectral projections of self-adjoint operators corresponding to open and closed intervals of $\mathbb{R}$.  First, we define continuous functions $f_{r,s}$ on $\mathbb{R}$, for $r<s$, like so
\[f_{r,s}(t) =
\begin{dcases}
0 & \text{for } t\in(-\infty,r)\\
\frac{t-r}{s-r} & \text{for }t\in[r,s]\\
1 & \text{for } t\in(s,\infty).
\end{dcases}\]
Also, for future reference, define $f_\delta=f_{\delta/2,\delta}$, for $\delta>0$.  Using the continuous functional calculus and the fact we can take infimums and supremums of monotone nets in $A''_+$, we further define, for any $a\in A_\mathrm{sa}$,
\begin{eqnarray*}
a_{[s,\infty)} &=& \inf_{r<s}f_{r,s}(a),\textrm{ and}\\
a_{(s,\infty)} &=& \sup_{r>s}f_{s,r}(a).\\
\end{eqnarray*}
We can similarly define spectral projections $a_S$ for any open or closed (even Borel) subset $S\subseteq\mathbb{R}$.  As weak/strong limits of elements that commute with $a$, spectral projections also commute with $a$ so $a_{(s,\infty)}a=aa_{(s,\infty)}=a_{(s,\infty)}aa_{(s,\infty)}\in A''_\mathrm{sa}$, and likewise for $a_{[s,\infty)}$.  We also have the following important operator inequalities. \[aa_{(-\infty,s]}\leq sa_{(-\infty,s]}\qquad\textrm{and}\qquad aa_{[s,\infty)}\geq sa_{[s,\infty)}.\]  We also define $[a]=(aa^*)_{(0,\infty)}=$ the projection onto $\overline{\mathcal{R}(a)}$, the norm closure of the range of $a$.  Note $[a^*]^\perp=(a^*a)_{(-\infty,0]}$ is the projection onto $\mathcal{N}(a)$, the kernel of $a$.

In fact, these inequalities (almost) uniquely define the spectral projections.  Specifically, for any $a\in\mathcal{B}(H)_+$, $a_{(-\infty,s]}$ and $a_{(s,\infty)}$ are the unique complementary orthogonal projections in $\mathcal{B}(H)$ such that $\langle av,v\rangle\leq s$, for all unit $v\in\mathcal{R}(a_{(-\infty,s]})$, and $\langle av,v\rangle>s$, for all unit $v\in\mathcal{R}(a_{(s,\infty)})$.  Using this fact we obtain the following result, which will be useful later on.

\begin{prp}\label{[aa*a]}
For any $s,t$ with $0\leq s<t$ and $a\in\mathcal{B}(H)$,
\[(aa^*)_{(s,t]}=[a(a^*a)_{(s,t]}].\]
\end{prp}

\begin{proof}  First note that the map $p\mapsto[ap]$ preserves the orthogonality of spectral projections of $a^*a$.  Specifically, if $S$ and $T$ are disjoint Borel subsets of $\mathbb{R}_+\setminus\{0\}$ then, as $[a^*a(a^*a)_S]\leq(a^*a)_S\perp(a^*a)_T$, we have \[[a(a^*a)_S]\perp[a(a^*a)_T].\]  In particular, this holds for $S=(0,s]$ and $T=(s,\infty]$ and the result will follow if we can show that \[(aa^*)_{(0,s]}=[a(a^*a)_{(0,s]}]\quad\textrm{and}\quad (aa^*)_{(s,\infty]}=[a(a^*a)_{(s,\infty]}].\]

Note that $(aa^*)_{\{0\}}=[a]^\perp$ so, by the comments above, we just need to show $\langle aa^*av,av\rangle\leq s\langle av,av\rangle$, for all $v\in\mathcal{R}((a^*a)_{(0,s]})$, and $\langle aa^*av,av\rangle>s\langle av,av\rangle$, for all non-zero $v\in\mathcal{R}((a^*a)_{(s,\infty]})$.  But we immediately have $\langle a^*a(a^*a-s)v,v\rangle\leq0$, for all $v\in\mathcal{R}((a^*a)_{(0,s]})$, while also $\langle a^*a(a^*a-s)v,v\rangle>0$, for non-zero $v\in\mathcal{R}((a^*a)_{(s,\infty]})$, so we are done.
\end{proof}

\autoref{specann} below indicates why spectral projections are a convenient tool when dealing with annihilators.  It also gives an idea of how plentiful they are.  For example, assume $A$ is infinite dimensional so we have $a\in A_+$ with $\sigma(a)$ infinite.  Then define a sequence $(f_n)$ of continuous functions from $\sigma(a)$ to $[0,1]$ with the sets $(f_n^{-1}(0,1])$ disjoint and non-empty.  These give rise to infinitely many orthogonal annihilators $\{f_n(a)\}^{\perp\perp}$ in $A$.  In fact, if we let $g_N=\sum_{n\in N}2^{-n}f_n$, for $N\subseteq\mathbb{N}$, then we get continuum many annihilators $B_N=\{g_N(a)\}^{\perp\perp}$ that, while no longer orthogonal, are still far apart in the sense that $||p_{B_N}-p_{B_M}||=1$, for all distinct $M,N\subseteq\mathbb{N}$.

First, though, we prove the following simple, but useful, algebraic lemma.

\begin{lem}\label{xab}
For any $x\in A$, $a,b\in A_+$ and $\alpha,\beta,\gamma>0$, \[xa^\alpha b^\beta a^\gamma=0\quad\Leftrightarrow\quad xa^\alpha b=0.\]
\end{lem}

\begin{proof}
Note that $ycc^*=0\Leftrightarrow yc=0$, for all $y,c\in a$, as $ycc^*=0$ gives $0=ycc^*y^*=(yc)(yc)^*0$ and hence $yc=0$.  Applying this to $y=xa^\alpha b^\beta$ and $c=a^{\gamma/2}$ we see that $xa^\alpha b^\beta a^\gamma=0$ implies $xa^\alpha b^\beta a^{\gamma/2}=0$.  We may continue to halve the last exponent in this way until it gets below $\alpha$, and then multiply on the right by a suitable exponent of $a$ to make it actually equal $\alpha$.  Then applying the note again, this time with $y=x$ and $c=a^\alpha b^{\beta/2}$, we get $xa^\alpha b^{\beta/2}=0$.  By reducing the last exponent again, and increasing it at the end if necessary, we finally get $xa^\alpha b=0$.  The converse is similar.
\end{proof}

This lemma holds even if $A$ is just an arbitrary *-ring with proper involution, as long as you can also take positive square roots.  Note that by iterating it we also get the same result for arbitrarily long sequences of (powers of) $a$'s and $b$'s.  If $x=1$ then we actually get the slightly better result \[a^\alpha b^\beta a^\gamma=0\quad\Leftrightarrow\quad ab=0.\]  However, for arbitrary $x$, we must keep the $\alpha$ exponent (e.g. if $A=M_2\cong\mathcal{B}(\mathbb{C}^2)$, $x$ is the projection onto $\mathbb{C}(2,-1)$, $a=\begin{bmatrix} 1 & 0 \\ 0 & 2 \end{bmatrix}$ and $b$ is the projection onto $\mathbb{C}(1,1)$ then $xa^\alpha b=0$ if and only if $\alpha=1$).

\begin{prp}\label{prp1}
If $a\in A^1_+$, $S\subseteq A_+$ and $as=s$, for all $s\in S$, then $ab=b$, for all $b\in S^{\perp\perp}$.
\end{prp}

\begin{proof}
If we had $1\in A$ then, as $as=s\Leftrightarrow(1-a)s=0$, it would immediately follow that $1-a\in S^\perp$ and hence $(1-a)b=0$, for all $b\in S^{\perp\perp}$.  Even if $1\notin A$, we still have $(1-a)b(1-a)\in S^\perp$, for any $b\in A$.  Thus, if $b\in S^{\perp\perp}_+$, then $(1-a)b(1-a)b=0$ and hence $b(1-a)=0=(1-a)b$, by \autoref{xab}, i.e. $ab=b$.  As any C*-algebra is generated by its positive elements, this completes the proof.
\end{proof}

One immediate consequence of \autoref{prp1} is that, even if $[B]_B$ is strictly contained in $[B]$, it is still order-dense in $[B]$ (see \eqref{densedef}).  In fact, for all $b\in B_+$, we can find $D\in[B]_B$ contained $\overline{bAb}$, the hereditary C*-subalgebra generated by $b$.  Just let $D=\{f_{2\delta}(b)\}^{\perp_B\perp_B}$, for $\delta<||b||$, and note that, as $f_\delta(b)f_{2\delta}(b)=f_{2\delta}(b)$, we have $d=f_\delta(b)d=f_\delta(b)df_\delta(b)\in\overline{bAb}$, for all $d\in D_+$.

\begin{prp}\label{specann}
If $a\in A^1_+$ and $f(a)\in A$, for some $f:[0,1]\rightarrow[0,1]$, then \[\{f(a)\}^{\perp\perp}\subseteq a_{\overline{f^{-1}(0,1]}}Aa_{\overline{f^{-1}(0,1]}}.\]
\end{prp}

\begin{proof}
If $0\in\overline{\mathbb{R}\backslash f^{-1}(0,1]}$, take continuous $g$ on $[0,1]$ with $g^{-1}\{0\}=\overline{\mathbb{R}\backslash f^{-1}(0,1]}$, and hence $g(0)=0$ so $g(a)\in A$.  Also, $g(a)\in\{f(a)\}^\perp$, so \[\{f(a)\}^{\perp\perp}\subseteq\{g(a)\}^\perp\subseteq g(a)_{\{0\}}Ag(a)_{\{0\}}=a_{\overline{f^{-1}(0,1]}}Aa_{\overline{f^{-1}(0,1]}}.\]  On the other hand, if $0\notin\overline{\mathbb{R}\backslash f^{-1}\{0\}}$ then take continuous $g$ on $[0,1]$ with $g^{-1}\{1\}=\overline{\mathbb{R}\backslash f^{-1}\{0\}}$ and $g(0)=0$, so $g(a)\in A$ and $g(a)b=b$.  By \autoref{prp1}, $g(a)c=c$, for all $c\in\{b\}^{\perp\perp}$, and hence $\{b\}^{\perp\perp}\subseteq g(a)_{\{1\}}Ag(a)_{\{1\}}=a_{\overline{f^{-1}(0,1]}}Aa_{\overline{f^{-1}(0,1]}}$.
\end{proof}

\subsection{Projection Properties}\label{psec}

We now point out some basic properties of projections, useful in their own right, but also good to keep in mind as results that might admit generalization to the annihilators in some way.  Indeed, \S\ref{SvsO} is devoted to proving some of these generalizations.

The projections in any C*-algebra are orthomodular (if $A$ is not unital then $q^\perp$ is not well-defined, but we can still interpret $p\wedge q^\perp$ as denoting the largest projection below $q$ that annihilates $p$, when such a projection exists).  In fact, we have the following order theoretic characterizations of commutatitivity, of which \eqref{pq=qp1}$\Rightarrow$\eqref{pq=qp2} implies orthomodularity (we should mention that the projections in a C*-algebra do not always form a lattice, so \eqref{pq=qp3} and \eqref{pq=qp2} should be interpretted as saying the given supremums and infimums actually exist \emph{and} satisfy the given identity).

\begin{prp}\label{pq=qp}
For $p,q\in\mathcal{P}(A)$, the following are equivalent.
\begin{eqnarray}
pq &=& qp.\label{pq=qp1}\\
p\wedge q &=& p\wedge(p^\perp\vee q).\label{pq=qp3}\\
p &=& (p\wedge q)\vee(p\wedge q^\perp).\label{pq=qp2}
\end{eqnarray}
\end{prp}

\begin{proof}\
\begin{itemize}
\item[\eqref{pq=qp1}$\Rightarrow$\eqref{pq=qp2}] Note that $p\wedge q=pq$ when \eqref{pq=qp1} holds so $p=pq+pq^\perp=(p\wedge q)\vee(p\wedge q^\perp)$.
\item[\eqref{pq=qp2}$\Rightarrow$\eqref{pq=qp1}] As $p\wedge q\leq q$, $p\wedge q$ commutes with $q$.  Likewise, $p\wedge q^\perp$ commutes with $q^\perp$ and hence $q$ so, if \eqref{pq=qp2} holds, $p=(p\wedge q)+(p\wedge q^\perp)$ commutes with $q$ too. 
\item[\eqref{pq=qp1}$\Rightarrow$\eqref{pq=qp3}] By \eqref{pq=qp1}, $p\wedge(p^\perp\vee q)=p\wedge(p\wedge q^\perp)^\perp=p(pq^\perp)^\perp=p-pq^\perp=pq=p\wedge q$.
\item[\eqref{pq=qp3}$\Rightarrow$\eqref{pq=qp2}] As $p\wedge q^\perp\leq p$, we have $p\wedge(p^\perp\vee q)=p\wedge(p\wedge q^\perp)^\perp=p-p\wedge q^\perp$ which, if \eqref{pq=qp3} holds, means $p=p\wedge q+p\wedge q^\perp=(p\wedge q)\vee(p\wedge q^\perp)$.
\end{itemize}
\end{proof}

So, by \autoref{pq=qp}, if $A$ is a commutative C*-algebra then every projection is central in $\mathcal{P}(A)$ (see \autoref{centredef}) and hence $\mathcal{P}(A)$ is a Boolean algebra (see the comments before \autoref{cenuniqcomp}).  Conversely, if $A$ is non-commutative C*-algebra that is generated by its projections then $pq\neq qp$ for some $p,q\in\mathcal{P}(A)$ and hence $\mathcal{P}(A)$ is not a Boolean algebra, again by \autoref{pq=qp}.  In particular, if $A$ is a von Neumann (or AW*-) algebra and $p\in\mathcal{P}(A)$ then $[p](=[p]_p$, as $\mathcal{P}(A)$ is orthomodular) is a Boolean algebra precisely when $pAp$ is commutative, i.e. precisely when $p$ is an abelian projection.  Thus if $p_\mathrm{I}$ is obtained as in \eqref{pI} (with $\mathbb{P}=\mathcal{P}(A)$) then $p_\mathrm{I}A$ is indeed the type I part of $A$, in the classical von Neumann algebra type decomposition.

Roughly speaking, the following result says that, for a projection $p$, being far from $q^\perp$ (as in (\ref{pnearq1})) is equivalent to being close to a subprojection of $q$ (as in (\ref{pnearq2}) and (\ref{pnearq3})), where $\lambda$ here quantifies this distance.

\begin{prp}\label{pnearq}
For $p,q\in\mathcal{P}(A)$ and $\lambda\in[0,1]$, the following are equivalent.
\begin{enumerate}
\item\label{pnearq1} $||pq^\perp||^2\leq\lambda$.
\item\label{pnearq2} $pqp\geq(1-\lambda)p$.
\item\label{pnearq3} There exists $r\in\mathcal{P}(A)$ with $r\leq q$ and $||r-p||^2\leq\lambda$.
\end{enumerate}
\end{prp}

\begin{proof}\
\begin{itemize}
\item[(\ref{pnearq1})$\Leftrightarrow$(\ref{pnearq2})] Just note $||pq^\perp||^2\leq\lambda\quad\Leftrightarrow\quad pq^\perp p\leq\lambda p\quad\Leftrightarrow\quad(1-\lambda)p\leq pqp$.
\item[(\ref{pnearq2})$\Rightarrow$(\ref{pnearq3})] If $\lambda=1$ let $r=q$.  Otherwise $\inf(\sigma(qpq)\backslash\{0\})\geq1-\lambda>0$ so $r=[qp]\in A$ and $||r-p||^2\leq\lambda$.
\item[(\ref{pnearq2})$\Leftarrow$(\ref{pnearq3})] Given such an $r$, simply note that $||pq^\perp||^2\leq||pr^\perp||^2\leq||p-r||^2\leq\lambda$.
\end{itemize}
\end{proof}

Similarly, the following says that a non-zero projection $p$ can not be simultaneously far from both another projection and its complement.  In fact, \eqref{Pythag} is still valid for $p\in A_+$ with $||p||=1$ (although $q$ still has to be a projection).

\begin{prp}
For all $p,q\in\mathcal{P}(A)$ with $p\neq0$,
\begin{equation}\label{Pythag}
||pq||^2+||pq^\perp||^2\geq1.
\end{equation}
Moreover, $||pq||^2+||pq^\perp||^2=1$ if and only if $\sigma_{pAp}(pqp)$ is a singleton.
\end{prp}

\begin{proof}
For \eqref{Pythag}, simply take $v\in\mathcal{R}(p)\backslash\{0\}$ and note that \[||v||^2=||qv||^2+||q^\perp v||^2\leq(||qp||^2+||q^\perp p||^2)||v||^2.\]  If $||pq||^2+||pq^\perp||^2=1$ then, setting $\lambda=1-||pq^\perp||^2$, we have $||pq||^2=\lambda$ and hence $pqp\leq\lambda p\leq pqp$ (using \autoref{pnearq} (\ref{pnearq1})$\Rightarrow$(\ref{pnearq2}) for the last inequality), i.e. $pqp=\lambda p$ and hence $\sigma_{pAp}(pqp)=\{\lambda\}$.  Conversely, if $\sigma_{pAp}(pqp)=\{\lambda\}$ for some $\lambda\in[0,1]$ then $pqp=\lambda p$ and $pq^\perp p=(1-\lambda)p$ so $||pq||^2+||pq^\perp||^2=\lambda+(1-\lambda)=1$.
\end{proof}

\begin{prp}\label{sigmapq}
For $p,q\in\mathcal{P}(A)$, \[\sigma(pq^\perp)\setminus\{0,1\}=1-\sigma(pq)\setminus\{0,1\}=\sigma(p^\perp q)\setminus\{0,1\}.\]
\end{prp}

\begin{proof}
From elementary spectral theory, we have $\sigma(pq)=\sigma(ppq)=\sigma(pqp)$ and $1-\sigma(pqp)\setminus\{0,1\}=\sigma(p-pqp)\setminus\{0,1\}=\sigma(pq^\perp p)\setminus\{0,1\}$.
\end{proof}

Note that $||pq^\perp||$ satisfies the triangle inequality, i.e. for $p,q,r\in\mathcal{P}(A)$, \[||pr^\perp||=||p(q+q^\perp)r^\perp||\leq||pq^\perp||+||qr^\perp||.\]  Also $||pq^\perp||=0=||qp^\perp||\Leftrightarrow p=q$ so \[\max(||pq^\perp||,||qp^\perp||)\] defines a metric on $\mathcal{P}(A)$.  the next result shows that it in fact coincides with the usual metric on $\mathcal{P}(A)$, i.e. $||p-q||$.  In particular, $\mathcal{P}(A)$ is complete in this metric.

\begin{prp}\label{||p-q||}
For $p,q\in\mathcal{P}(A)$, \[||p-q||=\max(||pq^\perp||,||p^\perp q||).\]  In fact, $||p-q||<1\Rightarrow||pq^\perp||=||p-q||=||qp^\perp||$.
\end{prp}

\begin{proof}
As $p-q=pq^\perp+pq-pq-p^\perp q=pq^\perp-p^\perp q$, and $pq^\perp(p^\perp q)^*=0=(pq^\perp)^*p^\perp q$, we have $||p-q||=\max(||pq^\perp||,||p^\perp q||)$.  So if $||pq^\perp||,||p^\perp q||<1$ then, by \autoref{sigmapq}, $||pq^\perp||^2=\max(\sigma(pq^\perp))=\max(\sigma(p^\perp q))=||p^\perp q||^2$.
\end{proof}

The following result shows that the Sasaki projection (see \cite{Kalmbach1983} \S7 Lemma 13) defined by $q$ takes any $p$ to the range projection of $qp$.

\begin{prp}\label{[qp]}
For $p,q\in\mathcal{P}(\mathcal{B}(H))$, $[qp]=(p\vee q^\perp)\wedge q$
\end{prp}

\begin{proof}
As $\mathcal{P}(\mathcal{B}(H))$ is orthomodular, this is equivalent to $p^\perp\wedge q=[qp]^\perp\wedge q$. To see this note that, for any $v\in\mathcal{R}(p)$ and $w\in\mathcal{R}(q)$,
\begin{equation}\label{vqw}
\langle v,w\rangle=\langle v,qw\rangle=\langle qv,w\rangle.
\end{equation}
As $w\in\mathcal{R}(p^\perp\wedge q)$ if and only if the left expression is $0$, for all $v\in\mathcal{R}(p)$, and $w\in\mathcal{R}([qp]^\perp\wedge q)$ if and only if the right expression is $0$, for all $v\in\mathcal{R}(p)$, we are done.
\end{proof}

The previous results of this subsection, while rarely stated explicitly, are no doubt well known.  The next result, however, might not be.  It characterizes the spectrum of a product of projections purely in terms of the norm and the ortholattice structure of $\mathcal{P}(\mathcal{B}(H))$.

\begin{thm}\label{orthospecpq}
For $p,q\in\mathcal{P}(\mathcal{B}(H))$, 
\begin{equation}\label{specpqeq}
\sigma(pq)\cup\{0\}=\overline{\{||qr||^2:r\leq p\textrm{ and }((r\vee q^\perp)\wedge q)\perp((r^{\perp_p}\vee q^\perp)\wedge q)\}}.
\end{equation}
\end{thm}

\begin{proof}
If $\theta\in\sigma(pq)\setminus\{0\}$ then, for any $\lambda>\theta$, $||q(pqp)_{(0,\lambda]}||^2\in[\theta,\lambda]$, while \[((pqp)_{(0,\lambda]}\vee q^\perp)\wedge q=(qpq)_{(0,\lambda]}\qquad\textrm{and}\qquad ((pqp)_{(0,\lambda]}^{\perp_p}\vee q^\perp)\wedge q=(pqp)_{(\lambda,1]},\] by \autoref{[aa*a]} (with $a=qp$) and \autoref{[qp]} (noting that we have $(pqp)_{(0,\lambda]}^{\perp_p}=(pqp)_{(\lambda,1]}\vee(p\wedge q^\perp)$).

Conversely, say $r\leq p$ and $(r\vee q^\perp)\wedge q=[qr]$ is orthogonal to $(r^{\perp_p}\vee q^\perp)\wedge q=[qr^{\perp_p}]$.  Thus $[qr]\perp r^{\perp_p}$ (see \eqref{vqw}) so $[pqr]=([qr]^\perp\wedge p)^{\perp_p}\leq r$ and this means $(pqp)r=r(pqp)r=r(pqp)$.  Thus $r$ commutes with every spectral projection of $pqp$.  Thus $r\leq(pqp)_{(0,||qr||^2]}$, otherwise $s=r(pqp)_{(\lambda,1]}\neq0$, for some $\lambda>||qr||^2$, and $||q^\perp s||^2\leq||q^\perp(pqp)_{(\lambda,1]}||^2\leq1-\lambda$ so $sqs\geq\lambda s$, by \autoref{pnearq}, which gives $||qr||^2\geq||qs||^2=||sqs||\geq\lambda>||qr||^2$, a contradiction.  On the other hand, if $\theta<||qr||$ then $r(pqp)_{(\theta,1]}\neq0$, otherwise we would have $r\leq(pqp)_{(0,\theta]}$ and hence $||qr||^2\leq||q(pqp)_{(0,\theta]}||^2\leq\theta<||qr||^2$, a contradiction.  But then $r(pqp)_{(\theta,||qr||^2]}=r(pqp)_{(\theta,1]}\neq0$ so $(\theta,||qr||^2]\cap\sigma(pq)\neq\emptyset$, for all $\theta<||qr||^2$, i.e. $||qr||^2\in\sigma(pq)$.
\end{proof}

\subsection{Annihilator Separativity}\label{SvsO}

The most fundamental difference between annihilators and projections is that, as shown in \autoref{nonorthoxpl},
\begin{center}
\textbf{the annihilators in an arbitrary C*-algebra may not be orthomodular.}
\end{center}
Nonetheless, we can show that $[A]^\perp$ is always separative, for arbitary C*-algebra $A$, which will suffice to allow us to apply the theory in \S\ref{OrderTheory}.  In fact, we will prove a strong version of separativity that is quite close to orthomodularity.

To see what this strong version is, note that with annihilators we can naturally quantify the degree of separativity.  Specifically, for $\epsilon\in[0,1]$, call $[A]^\perp$ \emph{$\epsilon$-separative} if, for all $B,C\in[A]^\perp$, \[B\subsetneqq C\Rightarrow\exists D\in[A]^\perp(\{0\}<D\subseteq C\textrm{ and }||BD||\leq\epsilon).\]  One immediately sees that if $[A]^\perp$ is $\epsilon$-separative, for any $\epsilon<1$, then it is separative and, by \autoref{orthoequiv}, $0$-separativity is equivalent to orthomodularity.  So \autoref{epsep} really is as close as we can get to orthomodularity without actually verifying it.

We can also use annihilators to separate arbitrary C*-subalgebras $B,C\subseteq A$ with $B\subseteq C$.  Specifically, define the \emph{separation} of $B$ from $C$ by \[\mathrm{sep}_C(B)=\inf_{D\in[C]\setminus\{\{0\}\}}||BD||.\]  By the comments after \autoref{prp1}, we could replace $[C]$ with $[C]_C$ here without changing the value of $\mathrm{sep}_C(B)$, so we might as well get rid of $C$, just assume we have a C*-subalgebra $B$ of $A$ and write $\mathrm{sep}(B)$ for $\mathrm{sep}_A(B)$.  Let us also assume $B=\overline{bAb}$, for $b\in B^1_+$, and define \[||x||_b=\sup_{n\geq0}||xb^{1/n}x^*||^{1/2}\qquad\textrm{and}\qquad\gamma(b)=\inf_{x\neq0}||x||_b/||x||,\] as in \cite{AkemannEilers2002}.  As $||xb^{1/n}x^*||^{1/2}=||b^{1/(2n)}x^*xb^{1/(2n)}||^{1/2}=||\sqrt{x^*x}b^{1/n}\sqrt{x^*x}||^{1/2}$, we actually have $\gamma(b)=\inf_{a\in A^1_+}||a||_b$.  And for any $a\in A^1_+$, \[||a||_b=\sup_{n\geq0}||ab^{1/n}||\geq||\{f_{1-\delta,1}(a)\}^{\perp\perp}B||-\delta,\] and hence $\gamma(b)\geq\mathrm{sep}(B)$.  While for any $C\in[A]^\perp$ and $c\in C^1_+$, we certainly have $||c||_b\leq||CB||$ so $\gamma(b)\leq\mathrm{sep}(B)$, i.e. \[\gamma(b)=\mathrm{sep}(B).\]

It is immediate that $\mathrm{sep}(B)=0$ whenever $B$ is not essential in $A$ or, equivalently, when $b$ is a zero-divisor.  Whether it is still possible to have $\mathrm{sep}(B)<1$ when $B$ is essential in $A$ was an open problem (see \cite{PeligradZsido2000}), usually phrased as asking whether all open dense projections $p$ ($B$ is essential means $p_B$ is dense in that $\overline{p}_B=1$) are \emph{regular}\footnote{This concept was first introduced and investigated in \cite{Tomita1959} Chapter 2 \S2 under a somewhat different, but equivalent, definition.} in the sense that $||ap||=||a\overline{p}||(=||a||$ when $p$ is dense), for all $a\in A$.  This was answered in the negative in \cite{AkemannEilers2002} Proposition 3.4 where it was shown that $\gamma(b)\leq4/5$ for a particular non-zero divisor $b$.  We improve on this in \autoref{0sep}, showing that $\mathrm{sep}(B)$ is, in fact, always $1$ or $0$, i.e. open dense projections are always either regular or very non-regular.

In order to prove these results, we first need the spectral projection inequalities contained in the lemmas below.  Note that if $c\in\mathcal{B}(H)_+$ and $p\in\mathcal{P}(\mathcal{B}(H))$ then $c\leq p$ implies $p^\perp cp^\perp\leq0$ and hence $p^\perp c=0$, i.e. $c=pc$.

\begin{lem}\label{lem1}
For $\epsilon,\lambda>0$, there exists $\delta>0$ such that, whenever $b,c\in\mathcal{B}(H)^1_+$, $c\leq q\in\mathcal{P}(\mathcal{B}(H))$ and $||bq||^2\leq\lambda+\delta$, we have
\begin{equation}\label{lem1eq}
||c_{[0,1-\epsilon]}(cb^2c)_{[\lambda-\delta,1]}||\leq\epsilon.
\end{equation}
\end{lem}

\begin{proof}
If $\epsilon\geq1$ then \eqref{lem1eq} holds trivially, so assume $\epsilon<1$.  For any $v\in H$,
\begin{eqnarray}
||cv||^2 &=& ||cc_{[0,1-\epsilon]}v||^2+||cc_{(1-\epsilon,1]}v||^2\nonumber\\
&\leq& (1-\epsilon)^2||c_{[0,1-\epsilon]}v||^2+||c_{(1-\epsilon,1]}v||^2\nonumber\\
&\leq& ||v||^2-\epsilon(2-\epsilon)||c_{[0,1-\epsilon]}v||^2\nonumber\\
&\leq& ||v||^2-\epsilon||c_{[0,1-\epsilon]}v||^2\qquad\textrm{(as $\epsilon\leq1$)}.\nonumber
\end{eqnarray}
Now for $v\in\mathcal{R}((cb^2c)_{[\lambda-\delta,1]})$,
\begin{eqnarray*}
(\lambda-\delta)||v||^2 &\leq& \langle cb^2c v,v\rangle\\
&=& ||bcv||^2\\
&\leq& ||bq||^2||cv||^2\\
&\leq& (\lambda+\delta)(||v||^2-\epsilon||c_{[0,1-\epsilon]}v||^2),\textrm{ so}\\
(\lambda+\delta)\epsilon||c_{[0,1-\epsilon]}v||^2 &\leq& 2\delta||v||^2\textrm{ and}\\
||c_{[0,1-\epsilon]}v||^2 &\leq& 2\delta||v||^2/(\lambda\epsilon),
\end{eqnarray*}
which immediately yields \eqref{lem1eq}, for $\delta\leq\lambda\epsilon^3/2$.
\end{proof}

\begin{cor}\label{cor1}
For $\epsilon,\lambda>0$, there exists $\delta>0$ such that, whenever $b,c\in\mathcal{B}(H)^1_+$, $c\leq q\in\mathcal{P}(\mathcal{B}(H))$ and $||bq||^2\leq\lambda+\delta$, we have
\[||(1-c)(cb^2c)_{[\lambda-\delta,1]}||\leq\epsilon.\]
\end{cor}

\begin{proof}
Replacing $\epsilon$ with $\epsilon/\sqrt{2}$ in \autoref{lem1}, we obtain $\delta>0$ such that, for any $v\in\mathcal{R}((cb^2c)_{[\lambda-\delta,1]})$,
\begin{eqnarray*}
||(1-c)v||^2 &=& ||(1-c)c_{[0,1-\epsilon/\sqrt{2})}v||^2+||(1-c)c_{[1-\epsilon/\sqrt{2},1]}v||^2\\
&\leq& ||c_{[0,1-\epsilon/\sqrt{2})}v||^2+\epsilon^2||c_{[1-\epsilon/\sqrt{2},1]}v||^2/2\\
&\leq& \epsilon^2||v||^2/2+\epsilon^2||v||^2/2.
\end{eqnarray*}
\end{proof}

The following result generalizes \cite{Bice2009} Lemma 5.3.

\begin{lem}\label{lem2}
For $\epsilon,\lambda>0$, there exists $\delta>0$ such that, whenever $b,c\in\mathcal{B}(H)^1_+$, $c\leq q\in\mathcal{P}(\mathcal{B}(H))$ and $||bq||^2\leq\lambda+\delta$, we have
\begin{equation}\label{lem2eq}
||b_{[0,\sqrt{\delta}]}(cb^2c)_{[\lambda-\delta,1]}||^2\leq1-\lambda+\epsilon.
\end{equation}
\end{lem}

\begin{proof}  Let $\delta>0$ be that obtained in \autoref{cor1} from replacing $\epsilon$ with $\epsilon/4$.  If necessary, replace $\delta$ with a smaller non-zero number so that we also have \[(1-\lambda+\delta+\epsilon/2)/(1-\delta)\leq1-\lambda+\epsilon.\]  Then, for all $v\in\mathcal{R}((cb^2c)_{[\lambda-\delta,1]})$,
\begin{eqnarray*}
(\lambda-\delta)||v||^2 &\leq& \langle cb^2cv,v\rangle\\
&=& \langle b^2cv,cv\rangle\\
&\leq& \langle b^2v,v\rangle+\epsilon||v||^2/2,\textrm{ by \autoref{cor1}},\\
&=& \langle b^2b_{[0,\sqrt{\delta}]}v,v\rangle+\langle b^2b_{(\sqrt{\delta},1]}v,v\rangle+\epsilon||v||^2/2\\
&\leq& \delta\langle b_{[0,\sqrt{\delta}]}v,v\rangle+\langle b_{(\sqrt{\delta},1]}v,v\rangle+\epsilon||v||^2/2\\
&=& \delta||b_{[0,\sqrt{\delta}]}v||^2+(||v||^2-||b_{[0,\sqrt{\delta}]}v||^2)+\epsilon||v||^2/2,\textrm{ so}\\
\quad(1-\delta)||b_{[0,\sqrt{\delta}]}v||^2 &\leq& (1-\lambda+\delta+\epsilon/2)||v||^2,\textrm{ and hence}\\
||b_{[0,\sqrt{\delta}]}v||^2 &\leq& (1-\lambda+\epsilon)||v||^2.
\end{eqnarray*}
\end{proof}

With \autoref{lem2}, we can already prove that $[A]^\perp$ is separative (see the proof of \autoref{epsep}, ignoring the last line).  However, for $\epsilon$-separativity, for arbitrary $\epsilon>0$, we need a couple more results.

\begin{lem}\label{lem3}
For $\epsilon,\lambda>0$, there exists $\delta>0$ such that, whenever $b,c\in\mathcal{B}(H)^1_+$, $p,q\in\mathcal{P}(\mathcal{B}(H))$, $b\leq p$, $c\leq q$ and $||pq||^2\leq\lambda+\delta$, we have
\begin{equation}\label{lem3eq}
||p(cb^2c)_{[\lambda-\delta,1]}||^2\leq\lambda+\epsilon.
\end{equation}
\end{lem}

\begin{proof}
Let $\delta>0$ be that obtained in \autoref{cor1} with $\epsilon$ replaced with $\epsilon/4$.  If necessary, decrease $\delta$ so that $\delta\leq\epsilon/2$.  Then, for $v\in\mathcal{R}((cb^2c)_{[\lambda-\delta,1]})$,
\begin{eqnarray*}
||pv||^2 &=& \langle pv,pv\rangle\\
&\leq& \langle pcv,pcv\rangle+\epsilon||v||^2/2\\
&\leq& (||pc||^2+\epsilon/2)||v||^2\\
&\leq& (\lambda+\delta+\epsilon/2)||v||^2.\\
&\leq& (\lambda+\epsilon)||v||^2.
\end{eqnarray*}
\end{proof}

\begin{thm}\label{septhm}
If we have $\epsilon>0$ and C*-subalgebras $B$ and $C\neq\{0\}$ of $A$ satisfying $||BC||^2=\lambda<1$, then there exists $D\in[A]^\perp$ with \[||BD||\leq\epsilon\qquad\textrm{and}\qquad||CD||^2\geq1-\lambda-\epsilon.\]
\end{thm}

\begin{proof} Choose $\delta>0$ small enough that it satisfies \autoref{lem2} and \autoref{lem3} with $\epsilon$ replaced by some $\mu>0$, to be determined later.  Take $c\in C^1_+$ and $b\in B^1_+$ with $||bc||^2>\lambda-\delta/2$.  Let $c'=f_{\lambda-\delta,\lambda-\delta/2}(cb^2c)\in C$ and $a=(1-f_\delta(b))c'^2(1-f_\delta(b))$, so
\begin{eqnarray}
||a|| &=& ||(1-f_\delta(b))c'||^2\nonumber\\
&\geq& ||b_{\{0\}}(cb^2c)_{[\lambda-\delta/2,1]}||^2\nonumber\\
&\geq& 1-||[b](cb^2c)_{[\lambda-\delta,1]}||^2,\textrm{ by \eqref{Pythag}}\nonumber\\
&\geq& 1-\lambda-\mu,\textrm{ by \eqref{lem3eq}.}\label{||s||}
\end{eqnarray}
In particular, $||a||>0$ as long as $\mu<1-\lambda$, and we may define $a'=||a||^{-1}a$.

By \eqref{lem2eq}, we have
\begin{equation}\label{anothereq}
||b_{[0,\sqrt{\delta}]}[c']||^2\leq||b_{[0,\sqrt{\delta}]}(cb^2c)_{[\lambda-\delta,1]}||^2\leq1-\lambda+\mu
\end{equation}
and so, by \autoref{pnearq},
\begin{equation}\label{BDeq}
[c']b_{[\delta,1]}[c']\geq[c']b_{[\sqrt{\delta},1]}[c']\geq(\lambda-\mu)[c'].
\end{equation}
Now take $b'\in B^1_+$ with 
\begin{equation}\label{qdef}
||(1-b')f_{\delta}(b))||\leq\mu,
\end{equation}
and note that, as $||BC||^2=\lambda$,
\begin{equation}\label{r^2}
c'b'c'\leq c'[c'][b'][c']c'\leq\lambda c'^2.
\end{equation}

Putting all this together, we have
\begin{eqnarray*}
(1-\lambda-\mu)||b'a'b'|| &\leq& ||a||||b'a'b'||,\textrm{ by \eqref{||s||}}\\
&=& ||b'ab'||\\
&=& ||b'(1-f_\delta(b))c'^2(1-f_\delta(b))b'||\\
&\leq& ||(b'-b'f_\delta(b)b')c'||^2+2\mu,\textrm{ by \eqref{qdef}},\\
&=& ||c'(b'-b'f_\delta(b)b')^2c'||+2\mu\\
&\leq& ||c'(b'-b'f_\delta(b)b')c'||+2\mu\\
&\leq& ||c'(\lambda-b'b_{[\delta,1]}b')c'||+2\mu,\textrm{ by \eqref{r^2}},\\
&\leq& ||c'(\lambda-[c']b_{[\delta,1]}[c'])c'||+4\mu,\textrm{ again by \eqref{qdef}},\\
&\leq& ||c'(\lambda-\lambda+\mu)c'||+4\mu,\textrm{ by \eqref{BDeq} (and $||BC||^2=\lambda$)}\\
&\leq& 5\mu,\textrm{ and hence},\\
||b'a'||^2 &\leq& ||b'\sqrt{a'}||^2\\
&\leq& 5\mu/(1-\lambda-\mu).
\end{eqnarray*}
Thus this inequality holds for all $b'$ in an approximate unit for $B$, and hence for all $b'\in B^1_+$.  If we let $D=\{f_{1-2\mu,1}(a')\}^{\perp\perp}\in[A]^\perp$ then, for any $d\in D^1_+$, we have $||d-a'd||\leq2\mu$ and hence, for any $b'\in B^1_+$, \[||b'd||^2\leq||b'a'd||^2+4\mu\leq||b'a'||^2+4\mu\leq5\mu/(1-\lambda-\mu)+4\mu.\]  Thus, so long as we choose $\mu>0$ at the start sufficiently small, this immediately gives $||BD||\leq\epsilon$.

Also, $||[1-f_\delta(b)]c'||^2\leq||b_{[0,\sqrt{\delta}]}[c']||^2\leq1-\lambda+\mu$, by \eqref{anothereq}, so as long as we chose $\mu$ at least half as small as the $\delta$ obtained in \autoref{lem2} (from the given $\epsilon$), we can also apply \eqref{lem2eq} with $c$ and $b$ replaced by $1-f_\delta(b)$ and $c'$ respectively to get
\[||c'_{[0,\sqrt{\mu}]}a'_{[1-\mu,1]}||^2\leq||c'_{[0,\sqrt{2\mu}]}((1-f_\delta(b))c'^2(1-f_\delta(b)))_{1-\lambda-2\mu}||^2\leq\lambda+\epsilon,\]
and hence $||CD||^2\geq||f_{\sqrt{\mu}}(c')f_{1-2\mu,1-\mu}(a')||^2\geq||c'_{[\sqrt{\mu},1]}a'_{[1-\mu,1]}||^2\geq1-\lambda-\epsilon$.
\end{proof}

Note that \autoref{septhm} is a modulo-$\epsilon$ generalization of \eqref{Pythag} to annihilators.  Essentially just rephrasing the first part also immediately yields the following.

\begin{thm}\label{0sep}
If $B$ is a C*-subalgebra of $A$ with $\mathrm{sep}(B)<1$ then $\mathrm{sep}(B)=0$.
\end{thm}

With a little more work, \autoref{septhm} also gives us $\epsilon$-separativity.

\begin{thm}\label{epsep}
$[A]^\perp$ is $\epsilon$-separative, for all $\epsilon>0$.
\end{thm}

\begin{proof}
Take $B,C\in[A]^\perp$ with $B\subsetneqq C$, so we have $c\in C^1_+\setminus B$.  This means we have $b\in B^\perp_+$ with $||b||=1$ and $bc\neq0$, and hence $b[c]\neq0$.  Set $q=[c]$, $\lambda=||bq||^2$, take positive $\epsilon<\lambda$ and let $\delta>0$ be that obtained in \autoref{lem2}.  Note that we may now assume that $||bc||^2>\lambda-\delta$ by replacing $c$ with $f_\mu(c)$ for sufficiently small $\mu$.  Take $s,s'\in(\lambda-\delta,||bc||^2)$ with $s<s'$ and set $D=\{f_{s,s'}(cb^2c)\}^{\perp\perp}$.  Then we see that $p_D\leq(cb^2c)_{[s,1]}\leq (cb^2c)_{[\lambda-\delta,1]}$ and $p_B\leq b_{\{0\}}\leq b_{[0,\sqrt{\delta}]}$ so, by \autoref{lem2}, \[||BD||^2\leq||b_{[0,\sqrt{\delta}]}(cb^2c)_{[\lambda-\delta,1]}||^2\leq1-\lambda+\epsilon<1.\]  Now simply apply \autoref{septhm} to get another $D\in[A]^\perp$ with $||BD||\leq\epsilon$.
\end{proof}

\begin{cor}\label{SP}
If $A$ has property (SP) then $[A]^\perp\cong[\mathcal{P}(A)]^\perp$.
\end{cor}

\begin{proof}
Property (SP) means that every hereditary C*-subalgebra of $A$ contains a non-zero projection (see \cite{Blackadar1994} Definition 6.1.1) which, by the comments after \autoref{prp1}, is equivalent to saying every annihilator contains a non-zero projection.  This, in turn, is equivalent to saying $\mathcal{P}(A)$ is order-dense in $A$ in the induced preorder.  As $[A]^\perp$ and hence $A$ is separative, this is equivalent to saying $\mathcal{P}(A)$ is join-dense in $A$, by \autoref{jdod}.  Now $[A]^\perp\cong[\mathcal{P}(A)]^\perp$, by \eqref{jdorthoiso}.
\end{proof}

Note that $[\mathcal{P}(A)]^\perp$ is just the canonical completion by cuts of the orthomodular partial order $\mathcal{P}(A)$.  However, this does not necessarily mean that $[A]^\perp$ is orthomodular for $A$ with property (SP), as orthomodularity is not necessarily preserved by completions.  In fact, an example is even given in \cite{Adams1969} of a modular lattice such that its completion is not orthomodular, and \autoref{nonorthoxpl}, with $X$ replaced by the Cantor space, even yields a real rank zero $A$ such that $[A]^\perp$ is not orthomodular.

Replacing projections above with any essential ideal (for any C*-algebra now), we get the same result (and more can be said in this case \textendash\, see \autoref{simsub}).

\begin{prp}\label{ess}
If $B$ contains an essential ideal in $A$ then $[A]^\perp\cong[B]^\perp$.
\end{prp}

\begin{proof}
Take $a\in A_+$.  If $B$ is an essential ideal, there exists $b\in B_+$ with $0\neq ba\in B$ and hence $\{ba\}^{\perp\perp}\subseteq\{a\}^{\perp\perp}$, i.e. $ba$ is below $a$ in the induced preorder.  As $a$ was arbitrary, $B$ is order-dense in $A$ in the induced preorder.  The result now follows as in the proof of \autoref{SP}.
\end{proof}

So internally, essential ideals are the same as the C*-algebra itself, as far as the annihilators are concerned.  The following shows that the same is true externally.

\begin{prp}\label{essidann}
If $C$ is an essential ideal in a hereditary C*-subalgebra $B$ of $A$ then $B^\perp=C^\perp$.
\end{prp}

\begin{proof}
For $a\in C^\perp_+$, $b\in B_+$ and $c\in C_+$ we have $bc\in C$ and hence $babc=0$.  As $c\in C_+$ was arbitrary, $bab\in C^\perp\cap B=\{0\}$ and hence $ab=0$.  As $b\in B_+$ was arbitrary, $a\in B^\perp$, i.e. $C^\perp\subseteq B^\perp$, while the reverse inclusion is immediate.
\end{proof}

Note that if \autoref{essidann} were true for all hereditary C*-subalgebras of a particular C*-algebra $A$, not just ideals, then $[A]^\perp$ would have to be orthomodular.

\subsection{Annihilator Ideals}\label{annideals}

Now that we know the annihilators in an arbitrary C*-algebra are separative, we immediately have the type decompositions given in \S\ref{tdsec}.  The only question that remains is whether these types can also be characterized algebraically in the same way as the projections appearing in the classical type decomposition of a von Neumann algebra.

Our first task is to identify the central annihilators.

\begin{thm}\label{centralannihilators}
$B\in[A]^\perp$ is central if and only if $B$ is an ideal.
\end{thm}

\begin{proof}
We first show that, whenever $B,D\in[A]^\perp$ and $B$ is an ideal, \[(D\cap B)^\perp\cap(D\cap B^\perp)^\perp\subseteq D^\perp.\]  To see this, take any $e\notin D^\perp$, so there is $d\in D$ such that $ed\neq0$.  But then there must also be $b\in B$ or $b\in B^\perp$ such that $edb\neq0$ and hence $edbd\neq0$.  But then $dbd\in D\cap B$ and hence $e\notin (D\cap B)^\perp$, or $dbd\in D\cap B^\perp$ and hence $e\notin (D\cap B^\perp)^\perp$.  So $e\notin(D\cap B)^\perp\cap(D\cap B^\perp)^\perp$ and the inclusion is proved.  The reverse inclusion is immediate and thus $B\in\mathrm{c}[A]^\perp$, by \autoref{centralequiv}.

Conversely, if $B$ is not an ideal then we have $a\in A^1_+$ and $b\in B^1_+$ with $ab\notin B$.  As $ba^2b\in B$ we must have $ab^2a\notin B$ (because $B$ is hereditary so $x\in B\Leftrightarrow x^*x\in B$ and $xx^*\in B$ \textendash\, see \cite{Pedersen1979} Theorem 1.5.2).  If $cc^*\leq\lambda ab^2a$ for some $\lambda\geq0$ then $c=abd$, for some $d\in\mathcal{B}(H)$, by \cite{Douglas1966} (alternatively, one could obtain a similar factorization in $A$ with \cite{Pedersen1979} Proposition 1.4.5).  Then $bc=babd=0\Leftrightarrow c=abd=0$, by \autoref{xab}, and hence $c\notin B^\perp$ as long as $c\neq0$.  So, if we set $C=\{f_\delta(ab^2a)\}^{\perp\perp}$ for $\delta>0$ sufficiently small that $f_\delta(ab^2a)\notin B$, then $C$ is not contained in $B$ and $cc^*\leq2||c||^2\delta^{-1}ab^2a$, for all $c\in C$, and hence $C\cap B^\perp=\{0\}$.  So we have $(B\wedge C)\vee(B^\perp\wedge C)=B\cap C\subsetneqq C$ and hence $B$ is not central. 
\end{proof}

In particular, $B$ is central (in $[A]^\perp$) if and only if $p_B$ is (in $\mathcal{P}(A'')$), by \autoref{centralideals}. This is perhaps a little surprising, considering that commutativity itself is not the same in $[A]^\perp$ and $\mathcal{P}(A'')$, as shown by the examples in \S\ref{Examples}.

Another important thing to note is the difference between central covers in $[A]^\perp$ and the central covers defined in \cite{Pedersen1979} 2.6.2.  There, the central cover $\mathrm{c}(a)$ of an element $a\in A''_\mathrm{sa}$ is defined to be the smallest $b\in(A'\cap A'')_\mathrm{sa}$ with $a\leq b$.  If $p\in\mathcal{P}(A'')$, this means that $\mathrm{c}(p)=\mathrm{c}_{\mathcal{P}(A'')}(p)$, i.e. central covers in this algebraic sense are the same for projections as central covers in the order theoretic sense \emph{with respect to} $\mathcal{P}(A'')$ (not $\overline{\mathcal{P}(A'')}^\circ$).  In general we have $\mathrm{c}(p_B)\leq p_{\mathrm{c}(B)}$, for $B\in[A]^\perp$, but this inequality can be strict.  We will primarily be concerned with central covers in $[A]^\perp$ rather than $A''_\mathrm{sa}$, although we would be remiss not to point out the following connections.

\begin{prp}\label{centralclosures}
If $B$ is a C*-subalgebra of $A$ then $p_{\mathrm{c}(B)}=\overline{\mathrm{c}(p_B)}^\circ$.
\end{prp}

\begin{proof}
As $\mathrm{c}(B)$ is an ideal, $p_{\mathrm{c}(B)}$ is a central projection and $B\subseteq\mathrm{c}(B)$ so $p_B\leq p_{\mathrm{c}(B)}$ and hence $\mathrm{c}(p_B)\leq p_{\mathrm{c}(B)}$ which, in turn, gives $\overline{\mathrm{c}(p_B)}^\circ\leq\overline{p}_{\mathrm{c}(B)}^\circ=p_{\mathrm{c}(B)}$.

For the reverse inequality, first note that $\overline{\mathrm{c}(p_B)}^\circ$ is central, by \autoref{centralprojections}.  So $\overline{\mathrm{c}(p_B)}^\circ A\overline{\mathrm{c}(p_B)}^\circ\cap A$ is an annihilator ideal, by \autoref{centralideals}, which certainly contains $B$ and hence also $\mathrm{c}(B)$.  Thus $p_{\mathrm{c}(B)}\leq\overline{\mathrm{c}(p_B)}^\circ$.
\end{proof}

Even if $p_{\mathrm{c}(B)}\neq\mathrm{c}(p_B)$, the analogous notion of `very orthogonal' is equivalent.

\begin{cor}\label{vo}
For $B,C\in[A]^\perp$, the following are equivalent.
\begin{enumerate}
\item\label{vo1} $B$ and $C$ are very orthogonal.
\item\label{vo2} $\mathrm{c}(p_B)\mathrm{c}(p_C)=0$
\item\label{vo3} $BAC=\{0\}$.
\end{enumerate}
\end{cor}

\begin{proof}\
\begin{itemize}
\item[(\ref{vo1})$\Rightarrow$(\ref{vo2})] As $\mathrm{c}(p_B)\leq p_{\mathrm{c}(B)}$, $||\mathrm{c}(p_B)\mathrm{c}(p_C)||\leq||p_{\mathrm{c}(B)}p_{\mathrm{c}(C)}||\leq||\mathrm{c}(B)\mathrm{c}(C)||=0$.
\item[(\ref{vo2})$\Rightarrow$(\ref{vo1})] If $\mathrm{c}(p_B)\mathrm{c}(p_C)=0$ then $\mathrm{c}(p_B)\leq\mathrm{c}(p_C)^\perp$ so $p_B=p_B^\circ\leq\mathrm{c}(p_B)^\circ\leq \mathrm{c}(p_C)^{\perp\circ}$.  By \autoref{centralprojections}, $\mathrm{c}(p_C)^{\perp\circ}$ is central.  Also $p_C\leq\mathrm{c}(p_C)^\circ$, and this latter projection is central, again by \autoref{centralprojections}.  Thus we must in fact have $\mathrm{c}(p_C)^\circ=\mathrm{c}(p_C)$, i.e. $\mathrm{c}(p_C)$ is open so $\mathrm{c}(p_C)^\perp$ is closed and $\mathrm{c}(p_C)^{\perp\circ}$ is topologically regular, by \eqref{intclosedtopreg}.  So we have $p_{\mathrm{c}(B)}\leq\mathrm{c}(p_C)^{\perp\circ}$ and hence $p_{\mathrm{c}(B)}\mathrm{c}(p_C)=0$.  Now $p_{\mathrm{c}(B)}p_{\mathrm{c}(C)}=0$ follows, by the same argument applied to $p_C$, and thus $\mathrm{c}(B)\cap\mathrm{c}(C)=\{0\}$.
\item[(\ref{vo2})$\Rightarrow$(\ref{vo3})] For $a\in A$, $b\in B$ and $c\in C$, $bac=b\mathrm{c}(p_B)a\mathrm{c}(p_C)c=\mathrm{c}(p_B)\mathrm{c}(p_C)bac=0$.
\item[(\ref{vo3})$\Rightarrow$(\ref{vo2})] If $BAC=\{0\}$ then if $I=ABA$ is the ideal generated by $B$ we have $IC=\{0\}$ and hence $\overline{I}C=\{0\}$, so $||\mathrm{c}(p_B)p_C||\leq||p_{\overline{I}}p_C||=0$.  But $\mathrm{c}(p_B)p_C=0$ means $p_C\leq\mathrm{c}(p_B)^\perp$ and hence $\mathrm{c}(p_C)\leq\mathrm{c}(p_B)^\perp$ so $\mathrm{c}(p_B)\mathrm{c}(p_C)=0$.
\end{itemize}
\end{proof}

The next result, applied to $B\in[A]^\perp$, shows $[A]^\perp$ has the relative centre property.

\begin{cor}\label{c_B}
If $B$ is a hereditary C*-subalgebra of $A$ then
\[\mathrm{c}_B[B]^\perp=\{B\cap C:C\in\mathrm{c}[A]^\perp\}.\]
\end{cor}

\begin{proof}
For one inclusion, take $C\in\mathrm{c}[A]^\perp$ and $b\in(B_+\cap C^\perp)^{\perp_B}$.  Then, for any $c\in C^\perp_+$, we have $\sqrt{b}c\in B\cap C^\perp_+$ and hence $bc=0$, i.e. $b\in C^{\perp\perp}=C$.  As $C$ is an ideal in $A$, $B\cap C$ is an ideal in $B$, and thus $B\cap C=(B\cap C^\perp)^{\perp_B}\in\mathrm{c}[B]^\perp$.

For the reverse inclusion, say we have $C\in\mathrm{c}_B[B]^\perp$ which, by \autoref{vo} \eqref{vo3}, means that $CBC^{\perp_B}=\{0\}$.  But, for any $c\in C_+$, $d\in C^{\perp_B}_+$ and $a\in A_+$, we have \[cad=\sqrt{c}(\sqrt{c}a\sqrt{d})\sqrt{d}\in CBC^{\perp_B}=\{0\}.\]  Thus $C$ and $C^{\perp_B}$ are very orthogonal (in $[A]^\perp$) so $\mathrm{c}(C)\cap B=C$.
\end{proof}

We also have something of an algebraic substitute for \autoref{ordertd}.

\begin{prp}\label{algebratd}
Given hereditary C*-subalgebras $(B_\alpha)$ with $\mathrm{c}(p_{B_\alpha})\mathrm{c}(p_{B_\beta})=0$, whenever $\alpha\neq\beta$, the C*-subalagebra $\bigoplus B_\alpha$ they generate is also hereditary.
\end{prp}

\begin{proof}
A C*-subalgebra $B$ is hereditary if and only if $bab\in B$ for all $a\in A$ and $b\in B$ (see \cite{Blackadar2006} II.5.3.9).  But given $a\in A$ and $\sum b_\alpha\in\bigoplus B_\alpha$ we have 
\begin{eqnarray*}
(\sum b_\alpha)a(\sum b_\alpha) &=& (\sum b_\alpha\mathrm{c}(p_{B_\alpha}))a(\sum\mathrm{c}(p_{B_\alpha})b_\alpha)\\
&=& \sum b_\alpha\mathrm{c}(p_{B_\alpha})a\mathrm{c}(p_{B_\alpha})b_\alpha,\quad\textrm{by centrality and orthogonality,}\\
&=& \sum b_\alpha ab_\alpha\in\bigoplus B_\alpha,\quad\textrm{by hereditarity of the }(B_\alpha).
\end{eqnarray*}
\end{proof}

\subsection{Equivalence}\label{Equivalence}

Here we study the basic properties of the following relation, which we believe to the be the natural analog for annihilators in C*-algebras of Murray-von Neumann equivalence.
\begin{dfn}\label{equidef}
$B,C\in[A]^\perp$ are \emph{equivalent}, written $B\sim C$, if $\{a\}^{\perp\perp}=B$ and $\{a^*\}^{\perp\perp}=C$, for some $a\in A$.
We write $B\precsim C$ if $B\sim D\subseteq C$, for some $D\in[A]^\perp$.
\end{dfn}

For $p,q\in\mathcal{P}(A)$, let us write $p\sim_\mathrm{MvN}q$ for the usual Murray-von Neumann equivalence notion, i.e. if there is a (partial isometry) $u\in A$ with $u^*u=p$ and $uu^*=q$.  Likewise, $p\precsim_\mathrm{MvN}q$ means $p\sim_\mathrm{MvN}r\leq q$ for some $r\in\mathcal{P}(A)$.  Also, we shall say a $A$ has \emph{polar decomposition} if, for all $a\in A$, there is a partial isometry $u\in A$ with $a=u|a|$, where $|a|=\sqrt{a^*a}$.  Von Neumann algebras certainly have polar decomposition, but so too do AW*-algebras, Rickart C*-algebras and all C*-algebra quotients of these (e.g. the Calkin algebra).  The following result shows that, for any such C*-algebra, the relation $\sim$ defined above really is a completely consistent extension of Murray-von Neumann equivalence.

\begin{prp}
If $A$ has polar decomposition then, for $p,q\in\mathcal{P}(A)$, \[p\sim_\mathrm{MvN}q\quad\Leftrightarrow\quad pAp\sim qAq\qquad\qquad\textrm{and}\qquad\qquad p\precsim_\mathrm{MvN}q\quad\Leftrightarrow\quad pAp\precsim qAq.\]
\end{prp}

\begin{proof}
It suffices to prove the first equivalence.  If $p\sim_\mathrm{MvN}q$ then the partial isometry witnessing this will also witness $pAp\sim qAq$.  Conversely, say we have $a\in A$ with $\{a\}^{\perp\perp}=pAp$ and $\{a^*\}^{\perp\perp}=qAq$ and take a partial isometry $u\in A$ with $a=u|a|$.  Then we immediately see that the left annihilator of $u$ is contained in the left annihilator of $a$, i.e. $\{u^*\}^\perp\subseteq\{a^*\}^\perp$ and hence $qAq=\{a^*\}^{\perp\perp}\subseteq\{u^*\}^{\perp\perp}=uAu^*$, which gives $uu^*\geq q$.  By replacing $u$ with $qu$ if necessary we may assume that $uu^*=q$ and hence $\{u^*\}^{\perp\perp}=\{a^*\}^{\perp\perp}$.  We also have $|a|^2=a^*a=|a|u^*u|a|$ and hence $|a|(1-u^*u)|a|=0$, which means $(1-u^*u)|a|=0$, i.e. $|a|=u^*u|a|=u^*a=a^*u$.  This immediately gives $\{a\}^{\perp\perp}\subseteq\{u\}^{\perp\perp}$, while if $0=ab$ then $0=\sqrt{aa^*}ab=a\sqrt{a^*a}b=a|a|b=aa^*ub$ which, as $\{a^*\}^\perp=\{u^*\}^\perp$, gives $0=u^*ub$, i.e. $b\in\{u\}^\perp$.  As $b\in\{a\}^\perp$ was arbitrary, $\{a\}^\perp\subseteq\{u\}^\perp$ and hence $\{u\}^{\perp\perp}\subseteq\{a\}^{\perp\perp}$, i.e. $u^*Au=\{u\}^{\perp\perp}=\{a\}^{\perp\perp}=pAp$ so $u^*u=p$ too.
\end{proof}

\begin{lem}\label{simlem}
If $a,b\in A$ and $\{b^*\}^{\perp\perp}\subseteq\{a\}^{\perp\perp}$ then $\{ab\}^{\perp\perp}=\{b\}^{\perp\perp}$.
\end{lem}

\begin{proof}
If $c\in A_+$ and $abc=0$ then $abcb^*=0$ so \[bcb^*\in\{a\}^\perp\cap\{b^*\}^{\perp\perp}\subseteq\{a\}^\perp\cap\{a\}^{\perp\perp}=\{0\}.\]  Thus $\{ab\}^\perp\subseteq\{b\}^\perp$, while $\{b\}^\perp\subseteq\{ab\}^\perp$ is immediate, so $\{ab\}^\perp=\{b\}^\perp$.
\end{proof}

\begin{cor}\label{simtran}
$\sim$ and $\precsim$ are transitive relations.
\end{cor}

\begin{proof}
If $a$ witnesses $B\precsim C$ and $b$ witnesses $C\precsim D$, then $ab$ witnesses $B\precsim D$, by \autoref{simlem}.  If $\precsim$ was actually $\sim$ here then one more application of \autoref{simlem} shows that $ab$ witnesses $B\succsim D$ and hence, as $\precsim$ and $\succsim$ are witnessed by the same element $ab$ of $A$, $B\sim D$.
\end{proof}

\begin{dfn}
$A$ is \emph{anniseparable} if $B\in[A]^\perp\Rightarrow B=S^\perp$ for countable $S\subseteq A$.  We call $B\in[A]^\perp$ \emph{principal} if there exists $b\in B$ with $\{b\}^{\perp\perp}=B$.  We say $A$ is \emph{orthoseparable} if every pairwise orthogonal subset of $A_+$ is countable.
\end{dfn}

We can see immediately that \[\textrm{anniseparability}\quad\Leftrightarrow\quad\textrm{every annihilator is principal}\quad\Leftrightarrow\quad\sim\textrm{ is reflexive.}\]  For, given $(s_n)\subseteq A^1_+$ such that $(s_n)^{\perp\perp}=B$, we can simply let $b=\sum2^{-n}s_n$ to get $\{b\}^{\perp\perp}$.  It then immediately follows from \autoref{simtran} that $\sim$ is an equivalence relation and $\precsim$ is a preorder.

Likewise, if we have orthogonal $(B_n)\subseteq[A]^\perp$, orthogonal $(C_n)\subseteq[A]^\perp$ and $(a_n)\subseteq A$ witnesses their equivalence, i.e. $\{a_n\}^{\perp\perp}=B_n$ and $\{a^*_n\}^{\perp\perp}=C_n$, for all $n\in\mathbb{N}$, then $a=\sum2^{-n}a_n$ witnesses the equivalence of $B=\bigvee B_n$ and $C=\bigvee C_n$, i.e. $\sim$ is countably additive.  This yields the second implicaiton in \[\textrm{separability}\quad\Rightarrow\quad\textrm{orthoseparability}\quad\Rightarrow\quad\textrm{additivity}.\]

Most C*-algebras one encounters are anniseparable.

\begin{prp}\label{annisep}
Each of the following conditions implies $A$ is anniseparable.
\begin{enumerate}
\item\label{annisep1} $A$ is an AW*-algebra (e.g. a von Neumann algebra).
\item $A$ is (norm) separable.
\item $A$ is orthoseparable and $[A]^\perp$ is orthomodular.
\end{enumerate}
\end{prp}

\begin{proof}\
\begin{enumerate}
\item An AW*-algebra is, by definition, a Baer *-ring (see \cite{Berberian1972} \S4 Definition 1 and 2, and also Proposition 1 to see how this is equivalent to other definitions sometimes given for AW*-algebras, like the one in \cite{Pedersen1979} 3.9.2), meaning that, for every $S\subseteq A$, $R(S)=pA$ for some $p\in\mathcal{P}(A)$, and hence $S^\perp=pAp=\{p\}^{\perp\perp}$.  To see that von Neumann algebras are AW*-algebras, simply note that, for any $S\subseteq A$, $R(S)$ is a weakly closed right ideal and hence of the form $pA$, for some $p\in\mathcal{P}(A)$, by \cite{Pedersen1979} Proposition 2.5.4.
\item For any $B\in[A]^\perp$, $B=S^{\perp\perp}$, where $S$ is any countable dense subset of $B$.
\item For any $B\in[A]^\perp$, let $S$ be a maximal pairwise orthogonal subset of $B_+$.  If we had $S^{\perp\perp}<B$ then, by orthomodularity, we would have $b\in B_+\cap S^\perp$, contradicting the maximality of $S$.
\end{enumerate}
\end{proof}

\begin{prp}\label{posp'}
If $B,C\in[A]^\perp$ are principal and semiorthoperspective, $B\sim C$.
\end{prp}

\begin{proof}
Take $b,c\in A_+$ with $\{b\}^{\perp\perp}=B$ and $\{c\}^{\perp\perp}=C$.  Given $d\in A_+$ with $bcd=0$ we have $cdc\in C\cap B^\perp=\{0\}$ and hence $cd=0$, i.e. $\{bc\}^{\perp}=\{c\}^\perp$ and hence $\{bc\}^{\perp\perp}=\{c\}^{\perp\perp}=C$.  A symmetric argument gives $\{cb\}^{\perp\perp}=\{b\}^{\perp\perp}=B$ and hence $bc$ witnesses $B\sim C$.
\end{proof}

\begin{cor}\label{posp'cor}
If $A$ is anniseparable, $\sim$ is weaker than perspectivity.
\end{cor}

\begin{cor}\label{simBsim}
If $A$ is anniseparable and $B\in[A]^\perp$ then $\sim_B$ is $\sim$ on $[B]_B$.
\end{cor}

\begin{proof}
If $\{b\}^{\perp\perp},\{b^*\}^{\perp\perp}\in[B]_B$, for some $b\in A$ (necessarily with $b\in B$), then $\{b\}^{\perp_B\perp_B}=\{b\}^{\perp\perp\perp_B\perp_B}=\{b\}^{\perp\perp}$ and, likewise, $\{b^*\}^{\perp_B\perp_B}=\{b^*\}^{\perp\perp}$.  Thus $\sim$ restricted to $[B]_B$ is stronger than $\sim_B$.  Conversely, if $[A]^\perp$ is anniseparable then, as $\sim$ is weaker than perspectivity, which is itself weaker than orthoperspectivity, we have $\{b\}^{\perp_B\perp_B}\sim\{b\}^{\perp\perp}\sim\{b^*\}^{\perp\perp}\sim\{b^*\}^{\perp_B\perp_B}$, for any $b\in B$.
\end{proof}

From \autoref{posp'cor}, we see that the only thing stopping $\sim$ from being a dimensional equivalence relation (for anniseparable $A$) according to \cite{Kalmbach1983} \S11 Definition 1, is the potential lack of finite (orthogonal) divisibility.  Although we can prove that $\sim$ satisfies a non-orthogonal version of divisibility.

\begin{prp}\label{nonorthodiv}
For any $a\in A$, the map $B\mapsto(Ba)^{\perp\perp}$ is order and supremum preserving on $[A]^\perp$.  Also $(Ba)^{\perp\perp}\sim B$, for all princpal $B\in[\{a^*\}^{\perp\perp}]$.
\end{prp}

\begin{proof}
The given map is certainly order preserving, even for arbitrary subsets $B$ of $A$.  If we have $\mathcal{B}\subseteq[A]^\perp$ then \[\bigvee_{B\in\mathcal{B}}(Ba)^{\perp\perp}=(\bigcap_{B\in\mathcal{B}}(Ba)^\perp)^\perp=((\bigcup\mathcal{B})a)^{\perp\perp}\subseteq((\bigvee\mathcal{B})a)^{\perp\perp}.\] To see that this last inclusion can be reversed, take $c\in((\bigcup\mathcal{B})a)^\perp_+$.  This means $aca^*\in(\bigcup\mathcal{B})^\perp=(\bigvee\mathcal{B})^\perp$ and hence $c\in((\bigvee\mathcal{B})a)^\perp$, i.e. $((\bigcup\mathcal{B})a)^\perp\subseteq((\bigvee\mathcal{B})a)^\perp$ and hence $((\bigvee\mathcal{B})a)^{\perp\perp}\subseteq((\bigcup\mathcal{B})a)^{\perp\perp}$.  Lastly, note that if $B=\{b\}^{\perp\perp}\subseteq\{a^*\}^{\perp\perp}$, for some $b\in B_+$, then $\{a^*b\}^{\perp\perp}=B$, by \autoref{simlem}, and \[c\in\{ba\}^\perp\quad\Leftrightarrow\quad aca^*\in\{b\}^\perp=B^\perp\quad\Leftrightarrow\quad c\in(Ba)^\perp,\] i.e. $\{ba\}^{\perp\perp}=(Ba)^{\perp\perp}$ and hence $ba$ witnesses $B\sim(Ba)^{\perp\perp}$.
\end{proof}

Note, however, that the map $B\mapsto(Ba)^{\perp\perp}$ may not preserve infimums, as the following example shows.  Specifically, consider $A=\mathcal{B}(H)$, where $H$ is a separable infinite dimensional Hilbert space with basis $(e_n)$.  Define $a\in\mathcal{B}(H)_+$ by $ae_n=\frac{1}{n^2}e_n$, so $\overline{\mathcal{R}(a)}=H$ and hence $\{a\}^{\perp\perp}=A$, and let $c$ be the projection onto $\mathbb{C}v$, where $v=\sum\frac{1}{n}e_n$.  Then $b=c^\perp\perp c$, and hence $B=bAb=\{b\}^{\perp\perp}\perp\{c\}^{\perp\perp}=cAc=C$ even though we still have $\overline{\mathcal{R}(ab)}=H$ and hence $(Ba)^{\perp\perp}=\{ba\}^{\perp\perp}=\{a\}^{\perp\perp}=A$, while $(Ca)^{\perp\perp}\neq\{0\}$.

\begin{prp}\label{simtr}
If $A$ is anniseparable and orthoseparable, $\sim$ is a type relation.
\end{prp}

\begin{proof}
Orthoseparability yields the $\Rightarrow$ part of \eqref{treq}.  For the $\Leftarrow$ part, say we have $B,C\in[A]^\perp$, $B\sim C$ and $D\in c[A]^\perp$.  By \autoref{nonorthodiv}, we have $E,F\in[C]$ with $B\cap D\sim E$, $B^\perp\cap D\sim F$ and $E\vee F=C$.  As the closed ideal generated by any $a\in A$ is the same as that generated by $a^*$, we have $c(E)=c(B\cap D)\subseteq B$ and $c(F)=c(B^\perp\cap D)\subseteq B^\perp$.  This means $E=B\cap C$ and $F=B^\perp\cap C$.
\end{proof}

In particular, we get the type-decompositions in \autoref{tdcor} coming from the type ideals of $\sim$-finite and $\sim$-orthofinite annihilators, by \autoref{trti}.  Whether these agree with the type decompositions you get from the type ideals of modular or relatively modular annihilators (see \autoref{tctd} and \eqref{Mdef}), we do not know.  However, it does follow from \autoref{permod} that if $\sim$ is finite on $[A]^\perp$ then $[A]^\perp$ is modular.  So if $[A]^\perp_{\sim\perp\mathrm{Fin}}$ is the type ideal of $\sim$-orthofinite annihilators in $[A]^\perp$ then $[A]^\perp_{\sim\perp\mathrm{Fin}}\cap[A]^\perp_{\mathbf{O}\mathrm{rel}}$ (see \eqref{Odef}) is the type ideal consisting of those $B\in[A]^\perp$ for which $\sim$ is finite on $[B]_B$ and \[[A]^\perp_{\sim\perp\mathrm{Fin}}\cap[A]^\perp_{\mathbf{O}\mathrm{rel}}\quad\subseteq\quad[A]^\perp_{\mathbf{M}\mathrm{rel}}.\]

\autoref{nonorthodiv} also yields a Cantor-Schroeder-Bernstein theorem.

\begin{thm}
If $A$ is anniseparable and $B,C\in[A]^\perp$ then \[B\precsim C\precsim B\quad\Rightarrow\quad B\sim C.\]
\end{thm}

\begin{proof}
Take $b,c\in A$ with $\{b\}^{\perp\perp}=B$, $\{b^*\}^{\perp\perp}\subseteq C$, $\{c\}^{\perp\perp}=C$ and $\{c^*\}^{\perp\perp}\subseteq B$.  We will apply Tarski's fixed point theorem, as in \cite{Berberian1972} \S1 Theorem 1.  Specifically, note that the map on $[B]$ defined by \[D\mapsto((Db^*)^{\perp_C}c^*)^{\perp_B}\] is order-preserving so, as $[B]$ is a complete lattice, it has a fixed point $F$.  By \autoref{nonorthodiv},
\begin{eqnarray*}
F^{\perp_B}=((Fb^*)^{\perp_C}c^*)^{\perp_B\perp_B} &\sim& (Fb^*)^{\perp_C},\textrm{ and}\\
F &\sim& (Fb^*)^{\perp_C\perp_C},\textrm{ so}\\
B\sim F\vee F^{\perp_B} &\sim& (Fb^*)^{\perp_C\perp_C}\vee(Fb^*)^{\perp_C}\sim C.
\end{eqnarray*}
\end{proof}

The next result shows $\precsim$ satisfies generalized comparison (see \autoref{gcdef}).

\begin{thm}\label{BDCDthm}
If $A$ is anniseparable and orthoseparable then, for all $B,C\in[A]^\perp$, there exists $D\in\mathrm{c}[A]^\perp$ with
\begin{equation}\label{BDCD}
B\cap D\precsim C\cap D\qquad\textrm{and}\qquad C\cap D^\perp\precsim B\cap D^\perp.
\end{equation}
\end{thm}

\begin{proof}
Let $(a_n)$ be a maximal subset of $A^1$ such that, for all distinct $m,n\in\mathbb{N}$, \[a_n^*a_n\in B,\quad a_na_n^*\in C,\quad\textrm{and}\quad a_ma_n^*=0=a_m^*a_n.\]  Let $a=\sum_n2^{-n}a_n$, $E=\{a\}^{\perp_B\perp_B}\subseteq B$ and $F=\{a^*\}^{\perp_C\perp_C}\subseteq C$.  By \autoref{orthoperpequiv} and \autoref{posp'cor}, $E\sim\{a\}^{\perp\perp}\sim\{a^*\}^{\perp\perp}\sim F$.  By maximality, $(B\cap E^\perp)A(C\cap F^\perp)=\{0\}$ and \eqref{BDCD} now follows from \autoref{simgcprp} and \autoref{vo}.
\end{proof}

Next we show that $\sim$ is the same in any sufficienlty large hereditary C*-subalgebra $B$ of $A$.  By \autoref{ess}, it applies to any $B$ containing an essential ideal of $A$.

\begin{prp}\label{simsub}
If $B$ is an anniseparable hereditary C*-subalgebra of $A$ and $C\mapsto B\cap C$ is an injective map from $[A]^\perp$ to $[B]^\perp$ then, for $D,E\in[A]^\perp$, \[D\sim E\quad\Leftrightarrow\quad B\cap D\sim_B B\cap E\]
\end{prp}

\begin{proof}
For any $b\in B$, $B\cap\{b\}^{\perp\perp}$ is the smallest element of $\{B\cap C:C\in[A]^\perp\}$ containing $b^*b$.  As $C\mapsto B\cap C$ maps $[A]^\perp$ to $[B]^\perp$, $[B]^\perp=\{B\cap C:C\in[A]^\perp\}$ so
\begin{equation}\label{perpeq}
\{b\}^{\perp_B\perp_B}=\{b\}^{\perp\perp}\cap B.
\end{equation}
So if $\{b\}^{\perp_B\perp_B}=B\cap D$ and $\{b^*\}^{\perp_B\perp_B}=B\cap E$ then $\{b\}^{\perp\perp}=D$ and $\{b^*\}^{\perp\perp}=E$, by the injectivity of $C\mapsto B\cap C$, i.e. $B\cap D\sim_B B\cap E$ implies $D\sim E$.

On the other hand, as $B$ is anniseparable, we have $d\in B\cap D_+$ such that $B\cap D=\{d\}^{\perp_B\perp_B}=B\cap\{d\}^{\perp\perp}$ and hence $D=\{d\}^{\perp\perp}$, as $C\mapsto B\cap C$ is injective.  Likewise, we have $e\in B\cap E_+$ such that $\{e\}^{\perp\perp}=E$.  If $a\in A$ witnesses $D\sim E$ then so does $dae$, by \autoref{simlem}.  But $B$ is a hereditary C*-subalgebra so $dae\in B$ and hence witnesses $B\cap D\sim_BB\cap E$, again by \eqref{perpeq}.
\end{proof}

In particular, under the hypothesis of \autoref{simsub}, $A$ is also anniseparable.  However, the converse fails in general.  For example, if $H$ is a non-separable Hilbert space then $\mathcal{K}(H)$ is not anniseparable, as $\mathcal{K}(H)$ is a non-principal annihilator in itself, even though $\mathcal{K}(H)$ is an essential ideal in the (necessarily anniseparable) von Neumann algebra $\mathcal{B}(H)$.  In fact, if $B$ is an essential ideal in anniseparable $A$ then \[B\textrm{ is principal (in itself) }\quad\Leftrightarrow\quad B\textrm{ is anniseparable}.\]  The $\Leftarrow$ part is immediate, and for the $\Rightarrow$ part simply note that if $\{b\}^{\perp_B\perp_B}=B$, for some $b\in B$, then $\{b\}^{\perp\perp}=A$ and, for any $c\in A$, we have $bc\in B$ and $\{bc\}^{\perp_B\perp_B}=\{c\}^{\perp\perp}\cap B$, by \eqref{perpeq}.

\subsection{Abelian Annihilators}\label{AA}

\begin{thm}\label{commuteBoolean}
$B\in[A]^\perp$ is commutative if and only if $[B]$ is a Boolean algebra.
\end{thm}

\begin{proof}
First note that $A$ is commutative if and only if every hereditary C*-subalgebra of $A$ is an ideal.  For if $A$ is commutative then so is $A''$ and, in particular, $\mathcal{P}(A'')^\circ$, so all hereditary C*-subalgebras of $A$ are ideals.  While if every hereditary C*-subalgebra of $A$ is an ideal then every open projection in $A''$ is central.  As $\mathcal{P}(A'')^\circ$ contains all spectral projections of elements of $A_+$ corresponding to open subsets of $\mathbb{R}$, it follows that $\mathrm{span}(\mathcal{P}(A'')^\circ)$ is dense in $A$ and hence $A$ is commutative.

So if $B$ is commutative then, by the previous paragraph and \autoref{centralannihilators}, $[B]_B=\mathrm{c}_B[B]_B$ is a Boolean algebra, so we only have to show that
\begin{equation}\label{BooleanAnn}
[B]_B=[B].
\end{equation}
Given $C\in[B]$, take $b\in C^{\perp_B\perp_B}_+$.  Then $babc=bacb=0$, for all $a\in C^\perp_+$ and $c\in C$, so $bab\in B\cap C^\perp=C^{\perp_B}$, but $bab\in C^{\perp_B\perp_B}$ too so $ba=0$ and hence $b\in C^{\perp\perp}=C$, i.e. $C^{\perp_B\perp_B}=C$ and hence $C\in[B]_B$.

Conversely, if $B$ is not commutative then, by the first paragraph and the other direction of \autoref{centralannihilators}, $[B]_B$ is not a Boolean algebra.  It may be that $[B]$ strictly contains $[B]_B$, but infimums still agree in both structures, as do supremums when one of the annihilators is $\{0\}$, so the counterexample to distributivity in the proof of \autoref{centralannihilators} still works in the possibly larger structure $[B]$.
\end{proof}

\begin{dfn}\label{abeliandef}
If $B\in[A]^\perp$ is commutative it will be called \emph{abelian}.  If $[A]^\perp$ contains no non-zero abelian annihilator ideals, $A$ will be called \emph{properly non-abelian}.  If $\mathrm{c}(B)=A$ for some abelian $B\in[A]^\perp$, we call $A$ \emph{discrete}.  If $A$ contains no non-zero abelian annihilators it will be called \emph{continuous}.
\end{dfn}

These definitions are consistent with classical von Neumann (or AW*-)algebra terminology (see \cite{Berberian1972} \S15 Definition 3).  It follows from \autoref{commuteBoolean} (and \eqref{BooleanAnn}) that the collection of abelian annihilators coincides with the type ideal \[[A]^\perp_{\mathbf{D}}=[A]^\perp_{\mathbf{D}\mathrm{rel}}\] (see \eqref{tctd1}, \eqref{tctd2} and \eqref{Ddef} for this notation).  It then follows from \autoref{BDCDthm} \autoref{Boolthm} that when $B$ and $C$ are abelian annihilators in anniseparable orthoseparable $A$, \[B\sim C\quad\Leftrightarrow\quad\mathrm{c}(B)=\mathrm{c}(C).\]  Also \eqref{tdcor1} gives us a decomposition of $A$ into central abelian and properly non-abelian parts, while \eqref{tdcor2} gives us a decomposition into central discrete and countinuous parts.  We could actually obtain these decompositions in a more algebraic, rather than order theoretic way, using \autoref{algebratd} and the following result.

\begin{thm}\label{commutativesubalgebra}
If $B$ is a commutative hereditary C*-subalgebra then so is $B^{\perp\perp}$.
\end{thm}

\begin{proof}
We first claim that every element of $B$ commutes with every element of $B^{\perp\perp}$.  If not, we would have $b\in B^1_+$ and $a\in B^{\perp\perp}_+$ such that $ab\neq ba$.  Then, for some $\epsilon>0$, we must have $ab_{[\epsilon,1]}\neq b_{[\epsilon,1]}a$ and hence $b_{[\epsilon,1]}a(1-b_{[\epsilon,1]})=b_{[\epsilon,1]}ab_{[0,\epsilon)}\neq0$.  Thus, for some $\delta<\epsilon$ sufficiently close to $\epsilon$, we must have $b_{[\epsilon,1]}ab_{[0,\delta]}\neq0$ and hence $f(b)ag(b)\neq0$ where $f=f_{(\epsilon+\delta)/2,\epsilon}$ and $g=1-f_{\delta,(\epsilon+\delta)/2}$.  If we had $g(b)af(b)^2ag(b)\in B^\perp$ then, as $a\in B^{\perp\perp}$, $f(b)ag(b)a=0$ and hence $f(b)ag(b)=0$, a contradiction.  Thus $f(b)ag(b)c\neq0$ for some $c\in B$.  As $B$ is hereditary and both $f(b)$ and $c$ are in $B$, this means that $d=f(b)ag(b)c\in B$ and, likewise $d^*\in B$.  However, $dd^*\leq\lambda f(b)^2$ for some $\lambda>0$ while $d^*d\leq\lambda'g(b)^2$ (note that $b,c\in B$ so $c$ commutes with $b$ and hence with $g(b)\in B+\mathbb{C}1$) for some $\lambda'>0$.  As $f(b)g(b)=0$, this means that $d$ and $d^*$ do not commute, contradicting the fact $B$ is commutative.

Now the claim is proved, take any $e,f\in B^{\perp\perp}_+$.  Given any $b\in B_+$, note that $b(ef-fe)=b^{1/4}eb^{1/2}fb^{1/4}-b^{1/4}fb^{1/2}eb^{1/4}=0$, as $b^{1/4}eb^{1/4},b^{1/4}fb^{1/4}\in B$.  Thus $ef-fe\in B^\perp\cap B^{\perp\perp}=\{0\}$ and hence, as $e$ and $f$ were arbitrary, $B^{\perp\perp}$ is commutative.
\end{proof}

It now follows that the continuous C*-algebras are, in fact, already well-known.

\begin{cor}\label{tII/III}
$A$ is continuous if and only if it is antiliminal.
\end{cor}

\begin{proof}
A C*-algebra is antiliminal if and only if it contains no commutative hereditary C*-subalgebras (either by definition, as in \cite{Pedersen1979} 6.1.1, or by a theorem, as in \cite{Li1992} Propositions 13.3.11 and 13.3.12).  So if $A$ is antiliminal then it certainly will not contain any abelian annihilators, while conversely if $A$ contains a commutative hereditary C*-algebra $B$ then $B^{\perp\perp}$ is an abelian annihilator, by \autoref{commutativesubalgebra}, and so $A$ can not be continuous.
\end{proof}

And we now have our first simple application of decomposition.

\begin{cor}\label{GCRtypeI}
If $A$ is postliminal then it is discrete.
\end{cor}

\begin{proof}
Every C*-subalgebra of a postliminal algebra is postliminal (see \cite{Pedersen1979} Proposition 6.2.9 or \cite{Li1992} Proposition 13.3.5) and, in particular, not antiliminal.  Thus the continuous part of $A$ must be $\{0\}$ and $A$ must be discrete.
\end{proof}

It should be mentioned that there is already a well-known postliminal-antiliminal decomposition theorem, which says that every C*-algebra has a unique postliminal ideal $I$ such that $A/I$ is antiliminal (see \cite{Pedersen1979} Proposition 6.2.7 or \cite{Li1992} Proposition 13.3.6).  The continous/antiliminal part $A_\mathrm{NI}$ obtained from \eqref{tdcor2} (using the type ideal of abelian annihilators) is quite different, being a subalgebra rather than a quotient, and not just any subalgebra either, an annihilator ideal.  Although it does naturally give rise to an antiliminal quotient.  To see how, note that, as the relation $ab=0$, for $a,b\in A^1_+$, is liftable (see \cite{Loring1997} Proposition 10.1.10), whenever $B\in\mathrm{c}[A]^\perp$ and $\pi:A\rightarrow A/B^\perp$ is the canonical homomorphism, $\pi\upharpoonright B$ is injective and $\pi(B)$ is an essential ideal in $A/B^\perp$.  Thus $A/A_\mathrm{I}$ (where $A_\mathrm{I}$ denotes the discrete part of $A$) contains $\pi(A_\mathrm{NI})$ as an essential antiliminal ideal.  If $A/A_\mathrm{I}$ contained a non-zero commutative hereditary C*-subalgebra $B$ then $B\cap\pi(A_\mathrm{NI})$ would be a non-zero commutative hereditary C*-subalgebra of $\pi(A_\mathrm{NI})$, a contradiction, so $A/A_\mathrm{I}$ must also be antiliminal.

As noted in \S\ref{Motivation}, any infinite dimensional type I von Neumann factor will not be postliminal, so the converse of \autoref{GCRtypeI} fails in general.  However, there is the possibility that they could be equivalent in the separable case.  If this were true then it would show that the discrete-continuous annihilator decomposition really is stronger than the classical postliminal-antiliminal decomposition in the separable case, because you could obtain the latter as a corollary of the former by the comments in the previous paragraph.

Here are some more consequences of \autoref{commutativesubalgebra} that will be important in the next subsection.

\begin{cor}\label{comsim}
If $B\in[A]^\perp_\mathbf{D}$, $C\in[A]^\perp$ and $B\sim C$ then $C\in[A]^\perp_\mathbf{D}$.
\end{cor}

\begin{proof}
Assume $\{a\}^{\perp\perp}=B$ and $\{a^*\}^{\perp\perp}=C$.  As $B$ is commutative, we have $ba^*ad=da^*ab$ and hence $aba^*ada^*=ada^*aba^*$, for all $b,d\in B$.  As $B$ is hereditary, so is $\overline{aBa^*}$ (because $aBa^*AaBa^*\subseteq aBa^*$) and \[C=\{a^*\}^{\perp\perp}=\{aa^*aa^*\}^{\perp\perp}\subseteq\overline{aBa^*}^{\perp\perp}\subseteq\{a^*\}^{\perp\perp}=C\] so, by \autoref{commutativesubalgebra}, $C$ is also commutative.
\end{proof}

\begin{cor}\label{discomden}
If $A$ is discrete, $[A]^\perp_\mathbf{D}$ is order-dense in $[A]^\perp$.
\end{cor}

\begin{proof}
Take $B\in[A]^\perp_\mathbf{D}$ with $\mathrm{c}(B)=A$.  For any other non zero $C\in[A]^\perp$, we have $\mathrm{c}(C)\wedge\mathrm{c}(B)=\mathrm{c}(C)\neq0$ and hence $bac\neq0$, for some $b\in B$ $a\in A$ and $c\in C$, by \autoref{vo}.  So $C'\sim B'\subseteq B$, where $B'=\{cab\}^{\perp\perp}$ and $C'=\{bac\}^{\perp\perp}$, which, as $B$ and hence $B'$ is commutative, means $C'$ is also commutative, by \autoref{comsim}.
\end{proof}

Before leaving the topic of abelian annihilators, let us point out that, while projection lattices are a complete isomorphism invariant for commutative AW*-algebras (by Stone's representation theorem), the same is not true for annihilator lattices of C*-algebras.  Indeed, if $X$ is any topological space with a countable basis of regular open subsets then \[[C^b(X)]^\perp\cong[\mathbb{B}_\infty]^\perp\times\{0,1\}^n,\] where $\mathbb{B}_\infty$ denotes the free Boolean algebra with infinitely many generators and $n$ is the number of isolated points of $X$ (see \cite{Birkhoff1967} Ch IX \S4 Theorem 3).  So $[C([0,1])]^\perp$ is orthoisomorphic to $[C([0,1]\times[0,1])]^\perp$, for example, even though $C([0,1])$ and $C([0,1]\times[0,1])$ are not isomorphic C*-algebras (as their spectrums $[0,1]$ and $[0,1]\times[0,1]$ are not homeomorphic, having dimension $1$ and $2$ respectively).

\subsection{Homogeneous Annihilators}\label{HomogeneousAnnihilators}

We have two competing notions of homogeneity.  One is classical, where $A$ is said to be \emph{$n$-homogeneous ($n$-subhomogeneous)}, for some $n\in\mathbb{N}$, if $\dim(H_\pi)=n$ ($\dim(H_\pi)\leq n$), for all $\pi\in\widehat{A}(=$ the set of all irreducible representations of $A)$.  The other is the natural one coming from the abelian annihilators, that is the notion of $A$ being $n$-$[A]^\perp_\mathbf{D}$-(sub)homogeneous, according to \autoref{homdef}.  We show in this subsection that these concepts are closely related.

\begin{prp}\label{[A]hom}
If $A$ is $n$-$[A]^\perp$-homogeneous then $A$ is not $(n-1)$-subhomogeneous.
\end{prp}

\begin{proof}
Take orthogonal $B_1,\ldots,B_n\in[A]^\perp$ with $\mathrm{c}(B_k)=A$, for all $k$.  This means that the ideal generated by each $B_k$ is essential in $A$.  So, taking any $b_1\in B_1\setminus\{0\}$, we can therefore find $a_1\in A$ and $b_2\in B_2$ with $b_1a_1b_2\neq0$.  Continuing in this manner, we obtain $a_1,\ldots,a_{n-1}\in A$ and $b_1,\ldots,b_n$ with $b_k\in B_k$, for all $k$, and $c=b_1a_1b_2a_2\ldots..a_{n-1}b_n\neq0$.  Let $\pi$ be an irreducible representation of $A$ on a Hilbert space $H$ with $\pi(c)\neq0$.  Then $\pi(b_k)\neq0$, for all $k$, which, as the $B_1,\ldots,B_n$ are orthogonal, means that $\dim(H)\geq n$.
\end{proof}

\begin{dfn}
A map $F$ on $A_\mathrm{sa}$, for each C*-algebra $A$, is \emph{functorial} if \[\pi\circ F=F\circ\pi\] whenever $\pi:A\rightarrow B$ is a (possibly non-unital) C*-algebra homomorphism.
\end{dfn}

\begin{lem}
If we have rank one $p_1,\ldots,p_n\in\mathcal{P}(M_n)$ with $p_1\vee\ldots\vee p_n=1$ then there are functorial maps $F=F_{p_1,\ldots,p_n}$ and $G=G_{p_1,\ldots,p_n}$ with $G(p_1,\ldots,p_n)=p_1$ such that, whenever we have rank one $q_1,\ldots,q_n\in\mathcal{P}(M_n)$ and $G(q_1,\ldots,q_n)\neq0$, \[q_1\vee\ldots\vee q_n=1=F(q_1,\ldots,q_n).\]
\end{lem}

\begin{proof}
For projections $p$ and $q$ on a Hilbert space, $\sigma(pq)=\sigma(p^\perp q^\perp)=\sigma(p^\perp q^\perp p^\perp)$.  So, if $0<\delta<1-\sup(\sigma(pq)\setminus\{1\})$, \[f_\delta(1-p^\perp q^\perp p^\perp)=1-(p^\perp q^\perp p^\perp)_{\{1\}}=(p^\perp\wedge q^\perp)^\perp=p\vee q.\]
So we get the functorial maps we want be letting $F_{p_1}$ and $G_{p_1}$ be the identity and recursively defining
\begin{eqnarray*}
F_{p_1,\ldots,p_n}(q_1,\ldots,q_n) &=& f_{\delta/2}(q)\quad\textrm{and}\\
G_{p_1,\ldots,p_n}(q_1,\ldots,q_n) &=& f_\delta(q)G_{p_1,\ldots,p_{n-1}}(q_1,\ldots,q_{n-1}),\quad\textrm{where}\\
q &=& 1-(1-q_n)(1-F_{p_1,\ldots,p_{n-1}}(q_1,\ldots,q_{n-1}))(1-q_n)\quad\textrm{and}\\
\delta &=& 1-||(p_1\vee\ldots\vee p_{n-1})p_n||^2.
\end{eqnarray*}
\end{proof}

\begin{thm}\label{[A]homthm}
If $A=B_1\vee\ldots\vee B_n$, for $B_1,\ldots,B_n\in[A]^\perp_{\mathbf{D}}$, $A$ is $n$-subhomogeneous.
\end{thm}

\begin{proof}
Assume that we have $\pi\in\widehat{A}$ with $\dim(H_\pi)>n$.  For all $k\leq n$, $B_k$ is a commutative hereditary C*-subalgebra of $A$ and hence, for all $\pi\in\widehat{A}$, there is a projection $r_{k\pi}$ of rank at most one with $\pi[B_k]=\mathbb{C}r_{k\pi}$.  Let \[m=\max\{\mathrm{rank}(\bigvee_{k\leq n}r_{k\pi}):\pi\in\widehat{A}\textrm{ and }\dim(H_\pi)>n\}\] and pick $\pi\in\widehat{A}$ such that $\dim(H_\pi)>n$ and $m=\mathrm{rank}(\bigvee_{k\leq m}r_{k\pi})$.  We can then also pick $b_1,\ldots,b_m\in\bigcup_{k\leq n}B_{k+}$ such that $\pi(b_k)$ is a (necessarily rank one) projection, for each $k\leq m$, and $\mathrm{rank}(\bigvee_{k\leq m}\pi(b_k))=m$.  Take any $\delta\in(0,1)$ and set
\begin{eqnarray*}
b_F &=& F_{\pi(b_1),\ldots,\pi(b_m)}(f_{\delta/2}(b_1),\ldots,f_{\delta/2}(b_m))\quad \textrm{and}\\
b_G &=& G_{\pi(b_1),\ldots,\pi(b_m)}(f_{\delta/2}(b_1),\ldots,f_{\delta/2}(b_m)).
\end{eqnarray*}

By irreducibility, we have $a\in A_+$ with at least $n+1$ points in $\sigma(\pi(a))$.  Take continuous functions $g_1,\ldots,g_{n+1}$ on $\mathbb{R}$ with disjoint supports that each have non-empty intersection with $\sigma(\pi(a))$.  Now define \[c=(1-b_F)a_0b_Ga_1g_1(a)\ldots a_{n+1}g_{n+1}(a)a_{n+2}f_\delta(b_1)\ldots a_{n+m+1}f_\delta(b_m),\] where $a_0,\ldots,a_{n+m+1}\in A$ are chosen so that $\pi(c)\neq0$ (which is possible by the irreducibility of $\pi$ and the fact $\mathrm{rank}(\pi(b_F))<\dim(H_\pi)$).

Now say we have another $\pi'\in\widehat{A}$ such that $\pi'(c)\neq0$.  Then $g_k(\pi'(a))\neq0$, for all $k\leq n+1$, and hence \[\dim(H_{\pi'})\geq|\sigma(\pi'(a))|>n.\] Likewise, for all $k\leq m$, $f_\delta(\pi'(b_k))\neq0$ and hence $f_{\delta/2}(\pi'(b_k))$ is a rank one projection.  As we also have $\pi'(b_G)\neq0$, this means that $\pi'(b_F)=\bigvee_{k\leq m}f_{\delta/2}(\pi'(b_k))$ and $\mathrm{rank}(\pi'(b_F))=m$ which, by our choice of $m$, means $\pi'(b_F)=\bigvee_{k\leq n}r_{k\pi'}$.  Thus $0=\pi'(b(1-b_F))=\pi'(bc)$, for all $b\in\bigcup_{k\leq n}B_k$, and this also holds of course if $\pi'(c)=0$.  As $\pi'$ was arbitrary (and the atomic representation is faithful), it follows that $bc=0$, for all $b\in\bigcup_{k\leq n}B_k$, i.e. $cc^*\in(B_1\vee\ldots\vee B_n)^\perp=A^\perp=\{0\}$, contradicting $\pi(c)\neq0$.
\end{proof}

\begin{thm}\label{homchar}
$A$ is $n$-$[A]^\perp_{\mathbf{D}}$-homogeneous if and only if $A$ contains an $n$-homogeneous essential ideal.
\end{thm}

\begin{proof}
Say $A$ contains an $n$-homogeneous essential ideal $B$.  Then $B$ is liminal and hence discrete so $B=\bigvee\mathrm{c}_B[B]^\perp_\mathbf{D}$.  As $[B]^\perp_\mathbf{D}$ is order-dense in $[B]^\perp$, by \autoref{discomden}, we have $B=\bigvee C_\lambda$ for $(C_\lambda)\subseteq\mathrm{c}_B[B]^\perp$ with each $C_\lambda$ being $\lambda$-$[B]^\perp_\mathbf{D}$-homogeneous.  If we had $C_\lambda\neq\{0\}$, for some $\lambda\neq n$, then we would have $\pi\in\widehat{C}_\lambda$ with $\dim(H_\pi)\neq n$, by either \autoref{[A]hom} or \autoref{[A]homthm}.  As $C_\lambda$ is an ideal, this $\pi$ extends to an irreducible representation of $B$, contradicting $n$-homogeneity.  Thus $B$ is $n$-$[B]^\perp_\mathbf{D}$-homogeneous and hence $A$ is $n$-$[A]^\perp_\mathbf{D}$-homogeneous, by \autoref{ess}.

Conversely, assume $A=B_1\vee\ldots\vee B_n$, for orthogonal $B_1,\ldots,B_n\in[A]^\perp_\mathbf{D}$ with $\mathrm{c}(B_k)=A$, for all $k$.  This means that the closed ideal $C_k$ generated by $B_k$ is essential in $A$, for all $k$.  As the intersection of a pair of essential ideals $D$ and $E$ is again essential (because then, for any $a\in A\setminus\{0\}$, we have $e\in E$ with $ae\neq0$, and then $f\in F$ with $aef\neq0$, where $ef\in E\cap F$), we see that $C=\bigcap C_k$ is an essential ideal in $A$.  If $\pi\in\widehat{C}$ then $\ker(\pi)\cap B_k\neq0$, for each $k$, otherwise we would have $C\subseteq\ker(\pi)$ and $\pi$ would be the zero representation.  Thus $\dim(H_\pi)\geq n$.  But also $C=B_1\vee_C\ldots\vee_CB_k$ and hence $\dim(H_\pi)\leq n$, by \autoref{[A]homthm}.
\end{proof}

\begin{cor}\label{subequiv}
The following are equivalent.
\begin{enumerate}
\item\label{subequiv1} $A\in[A]^{\perp n}_{\mathbf{D}}$.
\item\label{subequiv2} $A\in[A]^\perp_{\mathbf{D}<n+1}$.
\item\label{subequiv3} $A$ is $n$-subhomogeneous.
\item\label{subequiv4} There are orthogonal $k$-homogeneous ideals $B_k$ with $B_1\oplus\ldots\oplus B_n$ essential.
\end{enumerate}
\end{cor}

\begin{proof}\
\begin{itemize}
\item[\eqref{subequiv1}$\Rightarrow$\eqref{subequiv3}] See \autoref{[A]homthm}.
\item[\eqref{subequiv3}$\Rightarrow$\eqref{subequiv2}] If $A$ is discrete but $A\notin[A]^\perp_{\mathbf{D}<n+1}$, then there exists non-zero $B\in[A]^\perp_{\mathbf{D}(n+1)}$.  By \autoref{[A]homthm}, $B$ has representations $\pi$ with $\dim(H_\pi)>n$.  By \cite{Pedersen1979} Proposition 4.1.8, the same is true of $A$ and so $A$ is not $n$-subhomogeneous.
\item[\eqref{subequiv2}$\Rightarrow$\eqref{subequiv1}] Immediate.
\item[\eqref{subequiv2}$\Rightarrow$\eqref{subequiv4}] Immediate from the `only if' part of \autoref{homchar}.
\item[\eqref{subequiv4}$\Rightarrow$\eqref{subequiv2}] If \eqref{subequiv4} holds then, from the `if' part of \autoref{homchar}, we see that each $B_k^{\perp\perp}$ is $k$-$[B_k^{\perp\perp}]^\perp_{\mathbf{D}}$-homogeneous and $A=B_1^{\perp\perp}\vee\ldots\vee B_n^{\perp\perp}$.
\end{itemize}
\end{proof}

We should mention that a more classical proof of \eqref{subequiv3}$\Leftrightarrow$\eqref{subequiv4} could be obtained by utilizing the Jacobson topology on $\widecheck{A}(=\{\ker(\pi):\pi\in\widehat{A}\}=$ the primitive ideal space \textendash\, see \cite{Pedersen1979} \S 4.1), specifically by using \cite{Pedersen1979} Theorem 4.4.6 and Proposition 4.4.10.  Another thing $\widecheck{A}$ can be used for is establishing a similar result about $<\!\!\aleph_0$-subhomogeneity.  Specifically, we call $A$ \emph{$<\!\!\aleph_0$-subhomogeneous} if $\dim(H_\pi)<\infty$ for all $\pi\in\widehat{A}$.  This notion turns out to be different from $<\!\!\aleph_0$-$[A]^\perp_\mathbf{D}$-subhomogeneity (see \autoref{aleph0sub}), but we can at least show it is stronger.

\begin{prp}
If $A$ is $<\!\!\aleph_0$-subhomogeneous then $A\in[A]^\perp_{\mathbf{D}<\aleph_0}$.
\end{prp}

\begin{proof}
As $A$ is $<\!\!\aleph_0$-subhomogeneous, it is liminal and hence discrete.  Thus if $A\notin[A]^\perp_{\mathbf{D}<\aleph_0}$ we would have non-zero $B\in[A]^\perp_{\mathbf{D}\aleph_0}$, meaning there are orthogonal $(B_n)\subseteq[A]^\perp_\mathbf{D}$ with $\mathrm{c}(B_n)=\mathrm{c}(B)$, for all $n$.  This means that the subsets of $\widecheck{B}$ defined by $O_n=\{I\in\widecheck{B}:B_n\setminus I\neq\{0\}\}$ are open dense in the Jacobson topology.  Thus $O=\bigcap O_n$ is non-empty, as $\widecheck{B}$ is a Baire space (see \cite{Pedersen1979} Theorem 4.3.5), so $B$, and hence $A$, has an irreducible representation $\pi$ which is not zero on all of $B_n$, for any $n$.  Thus $\dim(H_\pi)=\infty$, contradicting the $<\!\!\aleph_0$-subhomogeneity of $A$.
\end{proof}

\begin{thm}\label{nsubess}
If $A$ is $<\!\!\aleph_0$-$[A]^\perp_\mathbf{D}$-subhomogeneous then every essential hereditary C*-subalgebra $B$ of $A$ contains an essential ideal.
\end{thm}

\begin{proof}
First assume that $A$ is $n$-($[A]^\perp_\mathbf{D}$-)subhomogeneous, for some $n\in\mathbb{N}$.  Let $\widehat{A}_B=\{\pi\in\widehat{A}:\dim(\pi[B]H_\pi)<\dim(H_\pi)\}$ and \[C=\bigcap_{\pi\in\widehat{A}_B}\ker(\pi).\]  This is immediately seen to be an ideal, and $C\subseteq B$ because the atomic representation is faithful on open projections, by \cite{Pedersen1979} Proposition 4.3.13 and Theorem 4.3.15.  All we need to verify is that $C^\perp=\{0\}$.  Assume to the contrary that we have non-zero $c\in C^\perp_+$.  Then $c\notin C$ so we have $\pi\in\widehat{A}_B$ with $\pi(c)\neq0$ and \[\dim(\pi[B]H_\pi)=m=\max\{\dim(\pi'[B]H_{\pi'}):\pi'\in\widehat{A}_B\textrm{ and }\pi'(c)\neq0\}.\]  As $B$ is hereditary, so is $\pi[B]$ and we can take $b\in B_+$ with $|\sigma(\pi(b)\setminus\{0\})|=m$.  Let $g_1,\ldots,g_m$ be continuous functions on $\mathbb{R}$ with disjoint supports contained in $[\delta,\infty)$, for some $\delta>0$, that each have non-empty intersection with $\sigma(\pi(b))$.  Now define \[a=(1-f_\delta(b))a_0g_1(b)a_1\ldots g_m(b)a_mc\in C^\perp,\] where $a_0,\ldots,a_m\in A$ are chosen so that $\pi(a)\neq0$.

Take any $b'\in B$.  As $a\in C^\perp$ we have $b'a\in C^\perp$ so, if $b'a\neq0$, there exists $\pi'\in\widehat{A}_B$ with $\pi'(b'a)\neq0$.  But then $\pi'(a)\neq0$ so $\pi'(c)\neq0$ and $g_k(\pi'(b))\neq0$, for all $k\leq m$.  This means $\sigma(\pi'(b)\backslash\{0\})\subseteq[\delta,\infty)$ and $\pi'(b'a)=\pi'(b'(1-f_\delta(b)))=0$, by the definition of $m$, a contradiction.  Thus $b'a=0$ which, as $b'$ was arbitrary, means $aa^*\in B^\perp=\{0\}$, contradicting $\pi(a)\neq0$.

For the general case, take an increasing sequence of (annihilator) ideals $(I_n)$ in $A$ such that $\bigcup I_n$ is essential in $A$ and $I_n$ is $n$-subhomogeneous, for all $n$.  As $B$ is essential in $A$, and $I_n$ is an ideal in $A$, $B\cap I_n$ is essential in $I_n$ (if $a\in I_{n+}$ then we have $b\in B_+$ with $ab\neq0\neq abab$ and $bab\in B\cap I_{n+}$), and hence contains an essential ideal $C_n$ in $I_n$.  But then $\bigcup C_n$ is an essential ideal in $\bigcup I_n$ which, as $\bigcup I_n$ is itself an essential ideal in $A$, means $\bigcup C_n$ is an essential ideal in $A$.
\end{proof}

\begin{cor}\label{subhom=}
If $B\in[A]^\perp$ is $<\!\!\aleph_0$-$[B]^\perp_\mathbf{D}$-subhomogeneous then $[B]=[B]_B$ so \[[A]^\perp_{\mathbf{D}<\aleph_0}\subseteq[A]^\perp_=.\]
\end{cor}

\begin{proof}
Take $C\in[B]$.  As $C$ is essential in $C^{\perp_B\perp_B}$, we have an essential ideal $D$ in $C^{\perp_B\perp_B}$ with $D\subseteq C$, by \autoref{nsubess}.  Then $C^{\perp_B\perp_B}=C^{\perp_B\perp_B\perp\perp}=D^{\perp\perp}=C$, by \autoref{essidann}.
\end{proof}

\begin{cor}
If $B$ is a hereditary C*-subalgebra of $A$ and $B\in[B]^\perp_{\mathbf{D}n}$ then \[B^{\perp\perp}\in[A]^\perp_{\mathbf{D}n}.\]
\end{cor}

\begin{proof}
Say $B_1,\ldots,B_n\in[B]^\perp_\mathbf{D}$ witness $B\in[B]^\perp_{\mathbf{D}n}$.  Then $B_1^{\perp\perp},\ldots,B_n^{\perp\perp}\in[A]^\perp_\mathbf{D}$, by \autoref{commutativesubalgebra}, and $C=B_1^{\perp\perp}\vee\ldots\vee B_n^{\perp\perp}\in[A]^\perp_{\mathbf{D}n}$, by \autoref{c_B}.  We do not immediately know that $B\subseteq C$, as supremums in $B$ and $A$ can differ, but $C$ will contain the hereditary C*-subalgebra generated by $B_1,\ldots,B_n$.  As this hereditary C*-subalgebra is essential in $B$, it contains an essential ideal $D$ in $B$, by \autoref{nsubess}.  By \autoref{essidann}, we have $B^{\perp\perp}=D^{\perp\perp}\subseteq C\subseteq B^{\perp\perp}$.
\end{proof}

\begin{thm}\label{Dnsim}
If $B,C\in[A]^\perp$ are perspective and $B\in[A]^\perp_{\mathbf{D}n}$ then $C\in[A]^\perp_{\mathbf{D}n}$.
\end{thm}

\begin{proof}
It suffices to prove the result for semiorthoperspective $B$ and $C$ (see \S\ref{persec}).  By \autoref{discomden}, $[A]^\perp_\mathbf{D}$ is order-dense in $[\bigvee \mathrm{c}[A]^\perp_\mathbf{D}]\supseteq[\mathrm{c}(B)]=[\mathrm{c}(C)]$.  By \autoref{commutativesubalgebra}, $[C]^\perp_\mathbf{D}$ is order-dense in $C$ so by \autoref{homthm}, $C$ is $[C]^\perp_\mathbf{D}$-subhomogeneous.  If $C$ is not $n$-$[C]^\perp_\mathbf{D}$-homogeneous then one of the non-zero homogeneous parts $D$ has order $<n$ or $>n$.  In the first case, $D=C\cap\mathrm{c}(D)$ and $B\cap\mathrm{c}(D)$ are semiorthoperspective.  In the second, we have $E\in[D]$ that is $(n+1)$-$[C]^\perp_\mathbf{D}$-homogeneous.  Then $[E]_E=[E]$ and $[B]=[B]_B$, by \autoref{subhom=}, and hence $E$ is semiorthoperspective to some $F\in[B]$, by \autoref{sopcor}.  By cutting down further by a central element, we may assume that $F$ is $m$-homogeneous for some $m\leq n$.  So in either case, after some renaming, we have $m$-$[A]^\perp_\mathbf{D}$-homogeneous $B$ semiorthoperspective to $n$-$[A]^\perp_\mathbf{D}$-homogeneous $C$ with $m<n$.  We show that this leads to a contradiction.

Take orthogonal $B_1,\ldots,B_m\in[A]^\perp_\mathbf{D}$ with $\mathrm{c}(B_k)=\mathrm{c}(B)$, for all $k\leq m$.  By \autoref{vo}, we have $B_1'\subseteq B_1$ and $B_2'\subseteq B_2$ such that $B_1'\sim B_2'$.  We also have non-zero $B_2''\subseteq B_2'$ and $B_3''\subseteq B_3$ with $B_2''\sim B_3''$.  We then also have $B_1''\subseteq B_1'$ with $B_1''\sim B_2''$, by \autoref{simtran}.  Continuing in this way, and again renaming, we get non-zero orthogonal $B_1\sim\ldots\sim B_m\sim C_1\sim\ldots\sim C_n$ in $[A]^\perp_\mathbf{D}$ such that $B\cap D=B_1\vee\ldots\vee B_m$ and $C\cap D=C_1\vee\ldots\vee C_n$, where $D=\mathrm{c}(B_1)=\ldots=\mathrm{c}(C_n)$.  As $D$ is central, replacing $B$ and $C$ with $B\cap D$ and $C\cap D$ we still see that $B$ and $D$ are semiorthoperspective.  As they are also principal, we have $B\sim C$, by \autoref{posp'}.  This means that $C=D_1\vee\ldots\vee D_m$, with $B_k\sim D_k$, for all $k\leq m$.  By \autoref{comsim}, each $D_k$ will be abelian and hence $C$ will be $m$-subhomogeneous, by \autoref{[A]homthm}, contradicting \autoref{[A]hom}.
\end{proof}

\begin{cor}\label{nsubcor}
Assume $A$ is $<\!\!\aleph_0$-$[A]^\perp_\mathbf{D}$-subhomogeneous.
\begin{enumerate}
\item\label{nsubcor1} Every open dense projection in $A''$ is regular.
\item\label{nsubcor2} The function $||BC^\perp||$, for $B,C\in[A]^\perp$, satisfies the triangle inequality.
\item\label{nsubcor3} $[A]^\perp$ is modular, and hence a continuous geometry.
\end{enumerate}
\end{cor}

\begin{proof}\
\begin{enumerate}
\item Every dense $p\in\mathcal{P}^\circ(A'')$ dominates some dense $q\in\mathcal{P}^\circ(A''\cap A')$, by \autoref{nsubess}.  As central projections are necessarily regular (see \cite{Effros1963} \S6), $p$ must also be regular.
\item Take $B,C,D\in[A]^\perp$ and $q\in\mathcal{P}^\circ(A''\cap A')$ with $q\leq p_C+p_{C^\perp}$.  For every $b\in B^1_+$ and $d\in D^{\perp1}_+$, we have \[||bd||=||bdq||=||bqd||\leq||b(p_C+p_{C^\perp})d||\leq||bp_{C^\perp}||+||p_Cd||,\] and hence $||BD^\perp||\leq||BC^\perp||+||CD^\perp||$.
\item We prove that perspectivity is finite (see \eqref{modperfin}).  If not, we would have perspective $B,C\in[A]^\perp$ with $C<B$.  By cutting down by a central element, we may assume that $C$ is $n$-$[A]^\perp_\mathbf{D}$-homogeneous, for some $n\in\mathbb{N}$.  By \autoref{subhom=}, $C\in[B]=[B]_B$ so $C<B$ means we have non-zero $b\in B_+\cap C^\perp$.  As $\mathrm{c}(B)=\mathrm{c}(C)$, by perspectivity, the proof of \autoref{[A]hom} (starting with $b$ instead of $b_1$) shows that $B$ is not $n$-subhomogeneous.  However, $B$ is $n$-$[A]^\perp_\mathbf{D}$-homogeneous, by \autoref{Dnsim}, and hence $n$-subhomogeneous, by \autoref{[A]homthm}, a contradiction.
\end{enumerate}
\end{proof}

So, by \autoref{nsubcor}\eqref{nsubcor3} and \cite{Kaplansky1955}, the annihilators $[A]^\perp$ in a $<\!\!\aleph_0$-$[A]^\perp_\mathbf{D}$-subhomogeneous C*-algebra $A$ form a continuous geometry.  As far as we know, these kinds of continuous geometries have not been investigated before.

\subsection{Matrix Valued Functions}\label{MVF}

\begin{dfn}\label{lscts}
Assume $X$ and $Y$ are topological spaces.  We denote the set of continuous functions from $X$ to $Y$ by $C(X,Y)$.  If $A$ is a C*-algebra, we call $f:X\rightarrow A$ \emph{bounded} if $\sup_{x\in X}||f(x)||<\infty$.  The C*-algebras (with pointwise operations) of bounded and bounded continuous functions from $X$ to $A$ will be denoted by $B(X,A)$ and $C^b(X,A)$ respectively.  We call $f:X\rightarrow\mathcal{P}(A)$ \emph{lower semicontinuous} at $x\in X$ if \[\forall\epsilon>0\exists\textrm{ a neighbourhood $Y$ of $x$ s.t. }\forall y\in Y(||p(y)^\perp p(x)||<\epsilon).\]  We say that $p$ is lower semicontinuous if it is lower semicontinuous at every $x\in X$ and denote the set of all lower semicontinuous $p$ by $C^\circ(X,\mathcal{P}(A))$.  We call $p$ \emph{upper semicontinuous} if $p^\perp$ (i.e. the function defined by $p^\perp(x)=p(x)^\perp$, for all $x\in X$) is lower semicontinuous and denote the set of upper semicontinuous $p$ by $\overline{C}(X,\mathcal{P}(A))$.
\end{dfn}

By \autoref{||p-q||}, $p:X\rightarrow\mathcal{P}(A)$ will be continuous at $x$ if and only if it is simultaneously upper and lower semicontinuous at $x$ and hence \[C(X,\mathcal{P}(A))=\overline{C}(X,\mathcal{P}(A))\cap C^\circ(X,\mathcal{P}(A)).\]  Also note that if $d:\mathcal{P}(A)\rightarrow\mathbb{R}$ is continuous and monotone in the sense that $p\leq q\Rightarrow d(p)\leq d(q)$ then, by \autoref{pnearq} \eqref{pnearq1}$\Rightarrow$\eqref{pnearq3}, $d\circ p$ will be lower semicontinuous in the usual sense (i.e. $f^{-1}(r,\infty)$ is open, for all $r\in\mathbb{R}$) if $p$ is lower semicontinous according to \autoref{lscts}.  In particular, $d$ could be the restriction to $\mathcal{P}(A)$ of a dimension function on $A$.

For $f:X\rightarrow A$, $[f]$ denotes the map $x\mapsto[f(x)]=(f(x)f(x)^*)_{(0,\infty)}$ to $\mathcal{P}(A'')$.

\begin{prp}\label{finspec}
If $f\in C(X,A_+)$, $x\in X$ and $f(x)$ has finite spectrum then $[f]$ is lower semicontinuous at $x$.
\end{prp}

\begin{proof}
Say $\delta>0$, $a,b\in A$, $||a-b||\leq\delta$ and $\lambda\in\sigma(b)$.  For unit $v\in\mathcal{R}(b_{\{\lambda\}})$ we have \[\lambda||a_{\{0\}}v||=||a_{\{0\}}(a-\lambda)v||\leq||(a-\lambda)v||\leq||(b-\lambda)v||+\delta=\delta,\] and hence $||a_{\{0\}}b_{\{\lambda\}}||\leq\delta/\lambda$.  As $f(x)_{(0,\infty)}$ is a finite sum of spectral projections, this calculation shows that, for any $\epsilon>0$, by choosing a sufficientely small neighbourhood $Y$ of $x$, we can ensure that $||f(y)_{\{0\}}f(x)_{(0,\infty)}||\leq\epsilon$, for all $y\in Y$.
\end{proof}

If $f(x)$ has infinite spectrum then $f$ can indeed fail to be lower semicontinuous at $x$.  For example, let $X=\mathbb{N}_\infty$ (the one point $\{\infty\}$ compactification of $\mathbb{N}$) and $A=\mathcal{K}(H)$, where $H$ is infinite dimensional separable with basis $(e_n)$.  Let $p_n$ denote the projection onto $\mathbb{C}e_n$, for each $n\in\mathbb{N}$.  Now consider $f\in C(X,A_+)$ defined by $f(n)=\sum_{k=1}^n2^{-k}p_{2k}+\sum_{k=n+1}^\infty2^{-k}p_{2k+1}$ and $f(\infty)=\sum_{k=1}^\infty2^{-k}p_{2k}$.  Note that $[f](n)=\sum_{k=1}^np_{2k}+\sum_{k=n+1}^\infty p_{2k+1}$ and $[f](\infty)=\sum_{k=1}^\infty p_{2k}$ so $||[f](n)^\perp[f](\infty)||=1$, for all $n\in\mathbb{N}$, and hence $[f]$ is not lower semicontinuous at $\infty$.

\begin{lem}\label{[pq]}
The map $(p,q)\mapsto[pq]$ is lower semicontinuous on \[\{(p,q)\in\mathcal{P}(A)\times\mathcal{P}(A):\inf(\sigma(pq)\setminus\{0\})>0\}.\]
\end{lem}

\begin{proof}
If $\inf(\sigma(pq)\setminus\{0\})>0$ then $r=(qpq)_{(0,1]}\in A$, $||p^\perp r||<1$ and $[pq]=[pr]\in A$.  By \cite{Bice2012} Lemma 2.3, for any $\epsilon>0$ we can choose $\delta>0$ so that $||p-p'||,||r-r'||\leq\delta$ implies $||[p'r']-[pr]||\leq\epsilon$.  But then $||p-p'||,||q-q'||\leq\delta$ implies $r'=(q'rq')_{(0,1]}\in A$ and $||r-r'||<\delta$ (see \cite{Bice2012} (2.3)) and hence \[||[p'q']^\perp[pq]||\leq||[p'r']^\perp[pr]||\leq||[p'r']-[pr]||\leq\epsilon.\]
\end{proof}

\begin{lem}\label{pveeq}
The map $(p,q)\mapsto p\vee q$ is lower semicontinuous on \[\{(p,q)\in\mathcal{P}(A)\times\mathcal{P}(A):\sup(\sigma(pq)\setminus\{1\})<1\}.\]
\end{lem}

\begin{proof}
If $\sup(\sigma(pq)\setminus\{1\})<1$ then $r=(qpq)_{[0,1)}\in A$, $||pr||<1$ and $p\vee q=p\vee r\in A$.  By \cite{Bice2012} Lemma 2.4, for any $\epsilon>0$ we can choose $\delta>0$ so that $||p-p'||,||r-r'||\leq\delta$ implies $||p'\vee r'-p\vee r||\leq\epsilon$.  But then $||p-p'||,||q-q'||\leq\delta$ implies $r'=(q'rq')_{(0,1]}\in A$ and $||r-r'||<\delta$ (see \cite{Bice2012} (2.3)) and hence \[||(p'\vee q')^\perp(p\vee q)||\leq||(p'\vee r')^\perp(p\vee r)||\leq||p'\vee r'-p\vee r||\leq\epsilon.\]
\end{proof}

If $p,q\in C^\circ(X,\mathcal{P}(M_n))$ then $\sup(\sigma(p(x)q(x))\setminus\{1\})<1$, for all $x\in X$, so it follows from \autoref{pveeq} and \autoref{pnearq} \eqref{pnearq1}$\Rightarrow$\eqref{pnearq3} that $p\vee q\in C^\circ(X,\mathcal{P}(M_n))$ (where $(p\vee q)(x)=p(x)\vee q(x)$, for all $x\in X$).  Also, if $(p_n)\subseteq C^\circ(X,\mathcal{P}(M_n))$ is increasing (i.e. for any $x\in X$, $(p_n(x))$ is (non-strictly) increasing in $\mathcal{P}(M_n)$) then, for any $x\in X$, it follows that $\bigvee_n p_n(x)=p_m(x)$, for some $m\in\mathbb{N}$, which, as $p_m\leq\bigvee_np_n$ and $p_m$ is lower semicontinuous, means that $\bigvee_np_n$ is lower semicontinuous at $x$.  As $x\in X$ was arbitrary, $\bigvee p_n\in\mathcal{P}(M_n)$.  This applies equally well to transfinite sequences and thus, combining thes facts, we see that $C^\circ(X,\mathcal{P}(M_n))$ is closed under taking arbitrary suprema, i.e.
\begin{equation}\label{bigveeP}
\mathcal{P}\subseteq C^\circ(X,\mathcal{P}(M_n))\quad\Rightarrow\quad\bigvee\mathcal{P}\in C^\circ(X,\mathcal{P}(M_n)).
\end{equation}

Given a Hilbert space $H$, we represent the C*-algebra of bounded functions $B(X,\mathcal{B}(H))$ from $X$ to $\mathcal{B}(H)$ on $\bigoplus_{x\in X}H$ by defining $f(v_x)_{x\in X}=(f(x)v_x)_{x\in X}$, for any $f\in B(X,\mathcal{B}(H))$ and $(v_x)_{x\in X}\in\bigoplus_{x\in X}H$.  Thus, our assumption that $A$ is represented concretely and faithfully on some Hilbert space $H$ means we can identify $C^b(X,A)''$ with a subset of $B(X,A'')$ and then use non-commutative topological concepts on $B(X,\mathcal{P}(\mathcal{B}(H)))$ with reference to $C^b(X,A)$. In fact, if $X$ is a completely regular Hausdorff topological space then it is not difficult to see that $C^b(X,A)''$ will be identified in this way with the entirety of $B(X,A'')$, i.e. \[C^b(X,A)''=B(X,A'').\]

\begin{dfn}
For any function $f$ between topological spaces we denote the set of points on which $f$ is continuous by $C_f$.
\end{dfn}

\begin{prp}\label{pcircp}
If $X$ is a completely regular topological space and $p:X\rightarrow\mathcal{P}(A)$ then $p^\circ(x)=p(x)=\overline{p}(x)$, for all $x\in C_p^\circ$.
\end{prp}

\begin{proof}
For any $x\in C_p^\circ$ we have a function $g:X\rightarrow[0,1]$ with $g(x)=1$ and $g[X\setminus C_p^\circ]=\{0\}$.  Then $gp$ (the map $x\mapsto g(x)p(x)$) is continuous on $X$, $gp\leq p$ and $g(x)p(x)=p(x)$.  As $p^\circ=\sup(pC^b(X,A)p\cap C^b(X,A))^1_+$, the first statement is proved.  The second follows from this, $C_p=C_{p^\perp}$ and $\overline{p}=p^{\perp\circ\perp}$.
\end{proof}

Representing $M_n$ canonically on $\mathbb{C}^n$, we have $M_n''=M_n$, and thus we can identify $C^b(X,M_n)''$ with a subset of $B(X,M_n)$.  The following result says that, under this identification, the open and closed projections in $C^b(X,M_n)''$ are precisely the lower and upper semicontinous projection functions respectively.

\begin{prp}\label{normregopenclosed}
If $X$ is a normal regular topological space then
\[C^\circ(X,\mathcal{P}(M_n))=\mathcal{P}(C^b(X,M_n)'')^\circ\qquad\textrm{and}\qquad\overline{C}(X,\mathcal{P}(M_n))=\overline{\mathcal{P}(C^b(X,M_n)'')}\]
\end{prp}

\begin{proof}
First note that $[f]\in C^\circ(X,\mathcal{P}(M_n))$ whenever $f\in C(X,M_n)$, by \autoref{finspec}.  Thus, for any hereditary C*-subalgebra $B$ of $C^b(X,M_n)$, we have $\sup B^1_+=\bigvee\{[f]:f\in B^1_+\}\in C^\circ(X,\mathcal{P}(M_n))$, by \eqref{bigveeP}, which means that $\mathcal{P}^\circ(C^b(X,M_n)'')\subseteq C^\circ(X,\mathcal{P}(M_n))$.

Given $p\in C^\circ(X,\mathcal{P}(A))$ and $x\in X$, let $Y$ be a closed neighbourhood of $x$ such that $||p(y)^\perp p(x)||<1$, for all $y\in Y$.  For each $m\in\mathbb{N}$, the set \[Y_m=\{y\in Y:\dim(p(y))\leq\dim(p(x))+m\}\] is closed.  The function defined by $f_0(y)=[p(y)p(x)]$ is continuous on $Y_0$ and thus has a continuous extension $g_0:X\rightarrow M_n^1$, by the Tietze extension theorem.  The function defined by $f_1(y)=p(y)g_0(y)^*g_0(y)p(y)$ is continuous on $Y_1$ and, continuing to apply the Tietze extension theorem in this way we eventually obtain $f=f_{n-\dim(p(x))}\in C(Y,M_{2+})$ with $f(y)\leq p(y)$, for all $y\in Y$.  As $X$ is completely regular, we have a function $h:X\rightarrow[0,1]$ with $h(x)=1$ and $h[X\setminus Y]=\{0\}$.  The function $f_x$ defined by $f_x(y)=h(y)f(y)$, for all $y\in Y$, and $f_x(y)=0$, for $y\in X\setminus Y$, is continuous on all of $X$.  We thus have $p=\sup_{x\in X}f_x\in\mathcal{P}^\circ(C^b(X,M_n)'')$
\end{proof}

We suspect that normality here could not be replaced with complete regularity.  To see this, all you would need is a continuous function $f:Y\rightarrow[0,1]$, where $Y$ is a closed subset of a completely regular space $X$ such that $Y$ is not locally compact and $f$ has no extension to any $Z\subseteq\overline{Y}^{\beta X}$ properly containing $Y$ ($\beta X$ here denotes the Stone-\v{C}ech compactification of $X$ \textendash\, see \cite{GillmanJerison1960}).  Identifying $[0,1]$ with a subspace of $\mathcal{P}(M_2)$ and extending $f$ to $X$ by defining $f(x)$ to be the identity of $M_2$, for $x\notin Y$, we see that $f\in C^\circ(X,\mathcal{P}(M_2))$.  If $f\in\mathcal{P}^\circ(C^b(X,M_n)'')$ then we would have a net $(f_n)\subseteq C^b(X,M^1_{n+})$ with $f=\sup f_n$ and hence $f^\beta=\sup f^\beta_n$ (where $f^\beta_n$ denotes the unique continuous extension of $f_n$ to $\beta X$) would be a lower semicontinuous extension of $f$ to $\beta X$.  Then $f$ would have to be $0$ on $\overline{Y}^{\beta X}\setminus Y$.  But $Y$ is not locally compact so $\overline{Y}^{\beta X}\setminus Y$ is not closed and hence we would have $f(y)=0$ for some $y\in Y$, a contradiction.

\begin{thm}\label{Cp}
For every $p\in C^\circ(X,\mathcal{P}(M_n))$, $C_p$ is open dense.
\end{thm}

\begin{proof}
For finite rank projections $q$ and $r$, $\dim(q)=\dim(r)$ implies $||q-r||=||qr^\perp||=||rq^\perp||$ and hence $C_p=C_{\dim\circ p}$.  As noted after \autoref{lscts}, $\dim\circ p$ is lower semicontinuous and, as $\{0,\ldots,n\}$ is discrete, $\dim\circ p$ will be constant, and hence continuous, on some open neighbourhood about any $x\in C_{\dim\circ p}$.  Now note that $\dim(p(x))<n$, for all $x\in X\setminus C_p$, by discontinuity.  If $X\setminus C_p$ contained an open set $O$, then we would have $\dim(p(x))<n-1$, for all $x\in O$, again by discontinuity.  But then $\dim(p(x))<n-2$, for all $x\in O$, etc. and hence we eventually get $\dim(p(x))=0$, for all $x\in O$, which shows that $\dim\circ p$ is continuous on $O$, a contradiction, i.e. $C_p$ is open dense.
\end{proof}

\begin{dfn}
Given functions $f,g$ on a topological space $X$, we define
\begin{eqnarray*}
f=_\mathrm{od}g &\Leftrightarrow& \{x\in X:f(x)=g(x)\}^\circ\textrm{ is dense, and}\\
f=_\mathrm{d}g &\Leftrightarrow& \{x\in X:f(x)=g(x)\}\textrm{ is dense}.
\end{eqnarray*}
\end{dfn}

Note that $=_\mathrm{od}$ is a transitive relation, while $=_\mathrm{d}$ may not be.  However, they coincide for the functions we are considering.

\begin{prp}\label{DpDq}
If $X$ is a completely regular topological space, the following are equivalent, for $p,q\in C^\circ(X,\mathcal{P}(M_n))$.
\begin{enumerate}
\item\label{DpDq1} $\overline{p}=\overline{q}$.
\item\label{DpDq2} $C_p\cap C_q\subseteq\{x\in X:p(x)=q(x)\}$.
\item\label{DpDq3} $p=_\mathrm{od}q$.
\item\label{DpDq4} $p=_\mathrm{d}q$.
\end{enumerate}
If these conditions hold and $p\leq q$ then $C_p\subseteq C_q$, so $C_p\subseteq\{x\in X:p(x)=q(x)\}$.
\end{prp}

\begin{proof}\
\begin{itemize}
\item[\eqref{DpDq1}$\Rightarrow$\eqref{DpDq2}]  Assume $p(x)\nleq q(x)$ for some $x\in C_p\cap C_q$.  As $X$ is completely regular, we have a continuous function $f:X\rightarrow[0,1]$ such that $f(x)=1$ and $f[X\setminus C_q]=\{0\}$.  Letting $g(y)=f(y)q(x)^\perp$, for all $y\in X$, we see that $g$ is continuous and $p(x)g(x)\neq0$, so $g$ witnesses the fact $q^{\perp\circ}\nleq p^{\perp\circ}$ and hence $\overline{p}=p^{\perp\circ\perp}\neq q^{\perp\circ\perp}=\overline{q}$.
\item[\eqref{DpDq2}$\Rightarrow$\eqref{DpDq3}]  Immediate from \autoref{Cp}.
\item[\eqref{DpDq3}$\Rightarrow$\eqref{DpDq4}]  Immediate from the definitions.
\item[\eqref{DpDq4}$\Rightarrow$\eqref{DpDq1}]  If $f\in p^\perp C(X,M_n)p^\perp\cap C^b(X,M_n)^1_+$ then $pf$ is $0$ everywhere on $X$.  If $p=_\mathrm{d}q$ then $qf$ is $0$ on a dense subset of $X$ which, as $||qf||$ is immediately seen to be lower semicontinuous, means $qf$ is $0$ on the entirety of $X$, i.e. $f\in q^\perp C(X,M_n)q^\perp\cap C^b(X,M_n)^1_+$.  Thus $p^{\perp\circ}=q^{\perp\circ}$ and hence we have $\overline{p}=p^{\perp\circ\perp}=q^{\perp\circ\perp}=\overline{q}$.
\end{itemize}
For the last statement, assume that $x\in C_p\setminus C_q$, so there is an open neighbourhood $Y$ of $x$ such that $\dim(p(y))=\dim(p(x))$, for all $y\in Y$.  As $x\notin C_q$, $x$ must be a limit point of the open set $Y\cap\{y\in X:\dim(q(y))>\dim(q(x))\}$ so, in particular, this set is not empty.  As $p(x)\leq q(x)$, it is contained in $\{y\in X:p(y)\neq q(y)\}$, showing that this set has non-empty interior and $p\neq_\mathrm{d}q$.
\end{proof}

Say $X$ is a completely regular topological space and we have a $=_\mathrm{od}$ equivalence class $E\subseteq C^\circ(X,\mathcal{P}(M_n))$.  For any $p\in E$, we have $\overline{p}^\circ\in E$, by \autoref{pcircp} and \autoref{Cp}.  And, for any other $q\in E$, we have $\overline{p}=\overline{q}$ and hence $\overline{p}^\circ=\overline{q}^\circ$, by \autoref{DpDq}, i.e. $E$ contains precisely one topologically open projection.  This will in fact be the maximum of $E\cap\mathcal{P}(C^b(X,M_n)'')^\circ$ and hence, if $X$ is normal, the maximum of $E$, by \autoref{normregopenclosed}.

\begin{prp}\label{commute=od}
If $X$ is completely regular and $p,q\in\overline{\mathcal{P}(C^b(X,M_n)'')}^\circ$, \[p\textrm{ and }q\textrm{ commute in }\overline{\mathcal{P}(C^b(X,M_n)'')}^\circ\quad\Leftrightarrow\quad pq=_\mathrm{od}qp.\]
\end{prp}

\begin{proof}
For any $p,q\in\overline{\mathcal{P}(C^b(X,M_n)'')}^\circ$, we have $p,q\in C^\circ(X,\mathcal{P}(M_n))$ by the first part of the proof of \autoref{normregopenclosed}.  We then have $p\vee q\in C^\circ(X,\mathcal{P}(M_n))$, by \eqref{bigveeP}, and hence $\overline{p\vee q}^\circ=_\mathrm{od}p\vee q$, by \autoref{pcircp} and \autoref{Cp}.  Likewise, we have $p^{\perp\circ}=_\mathrm{od}p$ and hence $p\wedge_{\overline{\mathcal{P}(C^b(X,M_n)'')}^\circ}q=(p^{\perp\circ}\vee q^{\perp\circ})^{\perp\circ}=_\mathrm{od}p\wedge q$.  Thus
\begin{equation}\label{commute=odeq}
\overline{(p^{\perp\circ}\vee q^{\perp\circ})^{\perp\circ}\vee(p^{\perp\circ}\vee q)^{\perp\circ}}^\circ=_\mathrm{od}(p\wedge q)\vee(p\wedge q^\perp).
\end{equation}
If $p$ and $q$ commute in $\overline{\mathcal{P}(C^b(X,M_n)'')}^\circ$ then \eqref{commute=odeq} becomes $p=_\mathrm{od}(p\wedge q)\vee(p\wedge q^\perp)$ which is equivalent to $pq=_\mathrm{od}qp$.  If $pq=_\mathrm{od}qp$ then $p=_\mathrm{od}\overline{(p^{\perp\circ}\vee q^{\perp\circ})^{\perp\circ}\vee(p^{\perp\circ}\vee q)^{\perp\circ}}^\circ$ which, as both sides here are topologically regular, means $=_\mathrm{od}$ is actually $=$ and hence $p$ and $q$ commute in $\overline{\mathcal{P}(C^b(X,M_n)'')}^\circ$.
\end{proof}

As $[A]^\perp$ and $\overline{\mathcal{P}(A'')}^\circ$ are isomorphic (see \autoref{annproj}), the relation $\sim$ we have on $[A]^\perp$ can be seen as a relation on $\overline{\mathcal{P}(A'')}^\circ$, which we will also denote by $\sim$, i.e. \[p\sim q\quad\Leftrightarrow\quad\exists a\in A(p=\overline{[a]}^\circ\textrm{ and }q=\overline{[a^*]}^\circ).\]

\begin{thm}\label{BCpara}
If $X$ is a hereditarily paracompact Hausdorff topological space then, for $p,q\in\overline{\mathcal{P}(C^b(X,M_n)'')}^\circ$,
\[p\sim q\qquad\Leftrightarrow\qquad\dim\circ p=_\mathrm{od}\dim\circ q\]
\end{thm}

\begin{proof}
If $a\in A$ witnesses $p\sim q$ then (even if $X$ is just completely regular), by \autoref{finspec}, \autoref{pcircp} and \autoref{Cp}, \[\dim\circ p=\dim\circ\overline{[a]}^\circ=_\mathrm{od}\dim\circ[a]=\dim\circ[a^*]=_\mathrm{od}\dim\circ\overline{[a]}^\circ=\dim\circ q.\]

On the other hand, if $\dim\circ p=_\mathrm{od}\dim\circ q$ then, for all $x\in C_p\cap C_q$, $\dim(p(x))=\dim(q(x))$.  Let $u$ be a partial isometry with $u^*u=p(x)$ and $uu^*=q(x)$.  Choose an open neighbourhood $Y_x$ of $x$ with $Y_x\subseteq C_p\cap C_q$ and $||p(x)-p(y)||,||q(x)-q(y)||<1$.  This means $p(x)p(y)$ has a polar decomposition for all $y\in Y_x$ and, moreover, the partial isometry appearing in this decomposition can be chosen continuously on $Y_x$.  Specifically, let $v(y)=p(x)p(y)\sqrt{p(y)p(x)p(y)}^{-1}$, where $^{-1}$ here denotes the quasi-inverse, so $v(y)$ is a partial isometry with $v(y)^*v(y)=p(y)$ and $v(y)v(y)^*=p(x)$, and $v$ is continuous on $Y_x$ (alternatively, this can be derived from \cite{Blackadar2006} II.3.3.4).  Likewise, define a continuous function $w$ on $Y_x$ so that $w(y)^*w(y)=q(x)$ and $w(y)w(y)^*=q(y)$, for all $y\in Y_x$.  Take a locally finite refinement $(Z_\alpha)$ of $(Y_x)_{x\in C_p\cap C_q}$.  By replacing each $Z_\alpha$ with $Z_\alpha\setminus\overline{\bigcup_{\beta<\alpha}Z_\beta}$ (and then throwing out the resulting empty sets), we obtain a locally finite collection of disjoint open sets with $\overline{\bigcup Z_\alpha}=X$.  For each $\alpha$ pick $x$ with $Z_\alpha\subseteq Y_x$, choose a function $g_\alpha:Z_\alpha\rightarrow[0,1]$ with $g_\alpha(x)=1$ and $g[X\setminus Z_\alpha]=\{0\}$, and set $f(z)=g(z)w(y)uv(y)$, for all $z\in Z_\alpha$.  Defining $f(z)=0$ for all $z\in X\setminus\bigcup Z_\alpha$ we see that the local finiteness of $(Z_\alpha)$ implies $f$ is continuous on all of $X$, i.e. $f\in C^b(X,M_n)$.  Furthermore, $[f]=_\mathrm{od}q$ and $[f^*]=_\mathrm{od}p$ and so we must have $\overline{[f]}^\circ=q$ and $\overline{[f^*]}^\circ=p$.
\end{proof}

We now show that \autoref{BCpara} can be used to prove a number of important facts about C*-algebras of the form $C^b(X,M_n)$.

\begin{cor}\label{matfunfin}
If $A$ is isomorphic to a C*-algebra of the form $C^b(X,M_n)$, for any hereditarily paracompact Hausdorff topological space $X$, then
\begin{enumerate}
\item\label{matfunfin1} $A$ is anniseparable,
\item\label{matfunfin2} $\sim$ is finite on $[A]^\perp$,
\item\label{matfunfin3} $[A]^\perp$ is modular, and
\item\label{matfunfin4} $\sim$ coincides with perspectivity.
\end{enumerate}
\end{cor}

\begin{proof}
\eqref{matfunfin1} follows from \autoref{BCpara} and the fact $=_\mathrm{od}$ is reflexive.  For \eqref{matfunfin2}, simply note that, for $p,q\in\overline{\mathcal{P}(C^b(X,M_n)'')}^\circ$ with $p\leq q$, $\dim\circ p=_\mathrm{od}\dim\circ q$ implies $p=_\mathrm{od}q$ and hence $p=q$, by \autoref{DpDq}.  Now \eqref{matfunfin3} and \eqref{matfunfin4} follow from \autoref{permod}.
\end{proof}

Lastly, let us point out that any C*-algebra of the form $C^b(X,M_n)$ will actually be isomorphic to the C*-algebra $C(\beta X,M_n)$, where $\beta X$ is the Stone-\v{C}ech compactification of $X$.  So we could have restricted ourselves to compact $X$ without restricting the class of C*-algebras under consideration.  However, as mentioned in \S\ref{NCT}, the Stone-\v{C}ech compactification can often be significantly harder to work with than $X$ itself, which is why we chose not to do this (e.g. it is not clear to us that \autoref{BCpara} would apply to $\beta X$ if it applied to $X$, i.e. we do not know if the Stone-\v{C}ech compactification of a hereditarily paracompact space is again hereditarily paracompact).

In fact, if $X$ is compact Hausdorff then the representation of $C(X,M_n)$ coming from the canonical representation of $M_n$ on $\mathbb{C}^n$ is just the atomic representation of $C(X,M_n)$.  As every such space is normal and regular, \autoref{normregopenclosed} shows that lower semicontinous projection functions correspond precisely to the open projections which, in turn, correspond precisely to the hereditary C*-subalgebras of $C(X,M_n)$, by \cite{Pedersen1979} Proposition 3.11.9, Proposition 4.3.13 and Theorem 4.3.15.  Also, analogous theorems for $C_0(X,M_n)$, where $X$ is locally compact, can also be proved in much the same way, or can be derived from the corresponding theorems for $C(X_\infty,M_n)$ (where $X_\infty$ is the one point compactification of $X$), using \autoref{ess} and the fact that $C_0(X,M_n)$ is an essential ideal in $C(X_\infty,M_n)$.

Also, there is presumably some room for the results of this subsection to be generalized to non-trivial Hilbert and C*-bundles (see \cite{RaeburnWilliams1998}).  Indeed, we could have derived \autoref{matfunfin}\eqref{matfunfin3} (which, combined with \eqref{matfunfin4}, also gives \eqref{matfunfin2}) without reference to the topological space $X$, simply by using \autoref{nsubcor}\eqref{nsubcor3} and the fact $C^b(X,M_n)\cong C(\beta X,M_n)$ is $n$-homogeneous.  At any rate, C*-algebras of the form $C^b(X,M_n)$ already allow us to create a number of instructive elementary examples, as we now show.


\subsection{Examples}\label{Examples}

For use in the following examples, define $P_\theta\in\mathcal{P}(M_2)$ by
\[P_\theta=\begin{bmatrix}\sin\theta \\ \cos\theta\end{bmatrix}\begin{bmatrix}\sin\theta & \cos\theta\end{bmatrix}=\begin{bmatrix} \sin^2\theta & \sin\theta\cos\theta \\ \sin\theta\cos\theta & \cos^2\theta \end{bmatrix}.\]

For $p,q\in\overline{\mathcal{P}(A'')}^\circ$, we have \[\overline{p}\leq q\quad\Rightarrow\quad p\leq q\quad\Leftrightarrow\quad p\leq\overline{q}\quad\Leftrightarrow\quad\overline{p}\leq\overline{q},\] and the first implication can not be reversed in general, even when $A$ is commutative.  Slightly more worthy of note is the fact that $p\leq q$ does not even imply that $\overline{p}$ and $q$ commute.

\begin{xpl}[$p,q\in\overline{\mathcal{P}(A'')}^\circ$ with $p<q$ but $\overline{p}q\neq q\overline{p}$]\label{commuteclosure1}\hfill\\
Let $A=C([0,1],M_2)$ and define
\[p(x) = \begin{cases} 0 & \text{for } x\in[0,\frac{1}{2}] \\ P_0 & \text{for } x\in(\frac{1}{2},1]\end{cases}\qquad\textrm{and}\qquad
q(x) = \begin{cases} P_{\pi/4} & \text{for } x\in[0,\frac{1}{2}]\\ 1 & \text{for } x\in(\frac{1}{2},1]\end{cases}.\]
We immediately have $p<q$ and \[\overline{p}(x) = \begin{cases} 0 & \text{for } x\in[0,\frac{1}{2})\\ P_0 & \text{for } x\in[\frac{1}{2},1]\end{cases},\]
so $\overline{p}q(\frac{1}{2})=P_0P_{\pi/4}\neq P_{\pi/4}P_0=q\overline{p}(\frac{1}{2})$.
\end{xpl}

\begin{xpl}[$p,q\in\overline{\mathcal{P}(A'')}^\circ$ with $p\overline{q}=\overline{q}p$, $\overline{p}q=q\overline{p}$ and $\overline{p}\,\overline{q}=\overline{q}\,\overline{p}$ but $pq\neq qp$]\label{commuteclosure2}\hfill\\
Let $A=C([0,1],M_2)$ and define
\[p(x) = \begin{cases} 1 & \text{for } x\in[0,\frac{1}{2}) \\ P_0 & \text{for } x\in[\frac{1}{2},1]\end{cases}\qquad\textrm{and}\qquad
q(x) = \begin{cases} P_{\pi/4} & \text{for } x\in[0,\frac{1}{2}]\\ 1 & \text{for } x\in(\frac{1}{2},1]\end{cases}.\]
We immediately have $pq(\frac{1}{2})=P_0P_{\pi/4}\neq P_{\pi/4}P_0=qp(\frac{1}{2})$ and \[\overline{q}(x) = \begin{cases} P_{\pi/4} & \text{for } x\in[0,\frac{1}{2})\\ 1 & \text{for } x\in[\frac{1}{2},1]\end{cases},\]
so $p\overline{q}=\overline{q}p$.
\end{xpl}

By concatenating the projection functions (and their orthocomplements) in \autoref{commuteclosure1} and \autoref{commuteclosure2} (with the functions taking the constant value $1$ on open intervals in between) it is easy to see that one can obtain $p,q\in\overline{\mathcal{P}(C([0,1],M_2)'')}^\circ$ having any combination of truth values for the statements
\begin{equation}\label{pq=qps}
pq=qp,\quad p\overline{q}=\overline{q}p,\quad\overline{p}q=q\overline{p}\quad\textrm{and}\quad\overline{p}\,\overline{q}=\overline{q}\,\overline{p},
\end{equation}
while still having $pq=_\mathrm{od}qp$, i.e. while still having $p$ and $q$ commute in $\overline{\mathcal{P}(C([0,1],M_2)'')}^\circ$ (although, if even one of the statements in \eqref{pq=qps} is true, then this is automatic from \autoref{commutativityimplication} and the fact $\overline{\mathcal{P}(C([0,1],M_2)'')}^\circ$ is orthomodular, by \autoref{matfunfin} \eqref{matfunfin3}).

While on the topic of commutativity, we mention the following natural question posed by Akemann \textendash\, if $p<q$ are topologically regular open projections in $A''$, for some separable C*-algebra $A$, can we always find commuting $a,b\in A$ that have range projections $p$ and $q$ respectively?  The answer is no in general, even when $p$ and $q$ satisfy all the various combinations of commutativity in \eqref{pq=qps}, as the following example shows.

\begin{xpl}[$p,q\in\overline{\mathcal{P}(A'')}^\circ$, $p<q$ but $ab\neq ba$ when $[a{]}=p$ and $[b{]}=q$]\hfill\\
Let $A=C([0,1],M_2)$ and define
\[p(x) = \begin{cases} P_0 & \text{for } x\in[0,\frac{1}{2})\\ 0 & \text{for }x=\frac{1}{2} \\ P_{\pi/4} & \text{for } x\in(\frac{1}{2},1]\end{cases}\qquad\textrm{and}\qquad
q(x) = \begin{cases} P_0 & \text{for } x\in[0,\frac{1}{2}]\\ 1 & \text{for } x\in(\frac{1}{2},1]\end{cases}.\]
We immediately see that $p,q\in\overline{\mathcal{P}(A'')}^\circ$, $p<q$, and all the equations in \eqref{pq=qps} hold.  If we have $a\in A$ with $[a]=p$ then, for $x\neq\frac{1}{2}$, $a(x)=\lambda(x)p(x)$, for some non-zero $\lambda(x)\in\mathbb{C}$.  Thus if we have $b\in A$ with $ab=ba$ then, for $x\neq\frac{1}{2}$, $b(x)\in\mathbb{C}p(x)+\mathbb{C}p(x)^\perp$.  But $(\mathbb{C}P_0+\mathbb{C}P_0^\perp)\cap(\mathbb{C}P_{\pi/4}+\mathbb{C}P_{\pi/4}^\perp)=\mathbb{C}1$ and hence $b(\frac{1}{2})=\lambda1$, for some $\lambda\in\mathbb{C}$.  In particular, $[b(\frac{1}{2})]\neq P_0$ and hence $[b]\neq q$.
\end{xpl}

A non-regular projection in a liminal C*-algebra is given in \cite{Akemann1970} Example I.2.  The following example shows that non-regular projections even exist in homogeneous C*-algebras and, moreover, that they can still be topologically regular and even equivalent to a regular projection.  In fact, the $p$ and $q$ given below are even Murray-von Neumann equivalent in $A^{**}$, and the partial isometry witnessing this will even yield an isomorphism between $pAp\cap A$ and $qAq\cap A$, i.e $p$ and $q$ will even be Peligrad-Zsid\'{o} equivalent (see \cite{PeligradZsido2000}).

\begin{xpl}[$p,q\in\overline{\mathcal{P}(A'')}^\circ$ with $\dim\circ p=\dim\circ q$, $p$ non-regular and $q$ regular]
Let $A=C([0,1],M_2)$ and let $p(x)=P_{(\pi/4)\sin(1/x)}$, for all $x\in(0,1]$, and set $p(0)=0$.  We immediately see that $p\in\overline{\mathcal{P}(A'')}^\circ$ and $\overline{p}$ is identical to $p$ except at $0$, where we have $\overline{p}(0)=1$.  If $r(x)=P_{\pi/2}$, for all $x\in[0,1]$, then $r\in\mathcal{P}(A)$ and $||\overline{p}r||=||r(0)||=1>1/\sqrt{2}=||pr||$, so $p$ is not regular.  Let $q$ be a continuous function from $(0,1]$ to rank $1$ projections in $M_2$ such that $\{q(1/n):n\in\mathbb{N}\setminus\{0\}\}$ is dense in the collection of rank $1$ projections in $M_2$.  Extending $q$ to a lower semicontinuous function on $[0,1]$ by defining $q(0)=0$, we immediately see that $q\in\overline{\mathcal{P}(A'')}^\circ$, while we also have $||a(0)||\leq\sup_{x\in(0,1]}||a(x)q(x)||\leq||aq||$ and hence $||a\overline{q}||=||aq||$, for any $a\in A$, i.e. $q$ is regular.
\end{xpl}

On the other hand, it is easy to find examples of open projections that are regular but not topologically regular.  Indeed, if $A$ is commutative then every projection in $A''$ is central and hence regular (as noted before \cite{Effros1963} Theorem 6.1) and so any non-regular open subset of the spectrum of $A$ will represent such a projection.  So, apart from the name, there does not appear to be any strong connection between regular and topologically regular projections.


\begin{xpl}[$\overline{\mathcal{P}(A'')}^\circ$ is not (operator norm) closed]\hfill\\
Let $A=C([0,1],M_2)$.  For each $n\in\mathbb{N}$, define $p_n\in\overline{\mathcal{P}(A'')}^\circ$ by $p_n(x)=P_{(1/n)\sin(1/x)}$, for all $x\in(0,1]$, and set $p_n(0)=0$.  Then $p_n$ converges to $p_\infty$, where $p_\infty(x)=P_0$, for all $x\in(0,1]$, and again $p_\infty(0)=0$.  This $p_\infty$ is open (as it should be, because the set of open projections is always norm closed in $A''$, by \cite{Pedersen1979} Proposition 3.11.9) but not topologically regular, as $\overline{p}_\infty(0)=P_0$ and $\overline{p}_\infty^\circ=\overline{p}_\infty$.
\end{xpl}

Note the above sequence also converges in the orthometric $\max(||pq^{\perp\circ}||,||qp^{\perp\circ}||)$ coming from the orthonorm (see \autoref{nsubcor}\eqref{nsubcor2}), but to a topologically regular open projection this time, namely $\overline{p}_\infty^\circ=\overline{p}_\infty$.  Thus the orthometric would appear to yield the more natural topology on $\overline{\mathcal{P}(A'')}^\circ$, at least in this case.  However, $\overline{\mathcal{P}(A'')}^\circ$ may not be complete even with respect to the orthometric.

\begin{xpl}[$\overline{\mathcal{P}(A'')}^\circ$ is not complete in the orthometric]\hfill\\
Let $A=C([0,1],M_2)$.  Let $g_1(x)=\sin(1/x)$, for all $x\in(0,1]$, and recursively define $g_{n+1}(x)=g_n(2x)$, for $x\in(0,\frac{1}{2}]$, and $g_{n+1}(x)=g_n(2x-1)$, for $x\in(\frac{1}{2},1]$.  For each $n$, let $p_n(x)=0$, when $x=m2^{-n}$ for some $m\in\mathbb{N}$ with $m<2^{-n}$, and $p_n(x)=P_{\sum_{k=1}^n4^{-k}g_k(x)}$, for all other $x\in[0,1]$.  Now $(p_n)\subseteq\overline{\mathcal{P}(A'')}^\circ$ and $(p_n)$ is Cauchy in the orthometric but, as dyadic rationals are dense in $[0,1]$, $(p_n)$ has no limit in $\overline{\mathcal{P}(A'')}^\circ$.
\end{xpl}

If we want to work with a complete metric space, we can of course just take the completion with respect to the orthometric (as long as $||\cdot\cdot||$ is an orthonorm on $[A]^\perp$).  The orthonorm then extends continuously to this completion and we can then naturally extend the ordering too by defining \[B\leq C\quad\Leftrightarrow\quad||BC^\perp||=0.\]  However, we do not know whether the order properties of the orthometric completion are generally better or worse than those of $[A]^\perp$ itself.

The next example is important because it shows that all the work we did in \S\ref{OrderTheory}, extending results about orthomodular lattices to separative ortholattices and carefully distinguishing $[p]$ and $[p]_p$, was in fact necessary for the theory to apply to arbitrary C*-algebras.

\begin{xpl}[$[A{]}^\perp$ is not orthomodular]\hfill\label{nonorthoxpl}\\
Let $A=C([0,1],\mathcal{K}(H))$, where $H$ is a separable infinite dimensional Hilbert space with basis $(e_n)$, and let $(s_n)$ enumerate a dense subset of $[0,1]$.  Define $U_n:H\rightarrow\mathbb{C}^2$ by $U_nv=(\langle v,e_{2n}\rangle,\langle v,e_{2n+1}\rangle)$ and set $p_n(s_n)=0$ and, for $x\neq s_n$, \[p_n(x)=U_n^*P_{1/(x-s_n)}U_n.\]  Let $p=\sum p_n\in\overline{\mathcal{P}(A)}^\circ$.  Also let $v=\sum(1/n)e_n\in H$, let $Q\in\mathcal{P}(\mathcal{K}(H))$ be the projection onto $\mathbb{C}v$ and define $q\in\mathcal{P}(A)$ by $q(x)=Q$, for all $x\in[0,1]$.  Note that $p\vee q>p$.

We first claim that $p\vee q\in\overline{\mathcal{P}(A)}^\circ$.  To see this, note that $\overline{p}_n(x)=p_n(x)$, for $x\neq s_n$, and $\overline{p}_n(s_n)=U_n^*U_n$.  Also $\overline{p}=\sum\overline{p}_n$ and $\overline{p\vee q}=\overline{p}\vee q$.  Now take $m\in\mathbb{N}$, let $(x_n)\subseteq[0,1]$ be such that $x_n\rightarrow s_m$ and $P_{1/(x_n-s_m)}=P_0$, for all $n$.  Then $\overline{p\vee q}(x_n)\rightarrow p(s_m)\vee Q\vee U_m^*P_0U_m$ in the weak (and strong) operator topology, while if $a\in A_+$ then $a(x_n)\rightarrow a(s_m)$ in norm.  Thus if $a\leq\overline{p\vee q}$ then $a(s_m)\leq p(s_m)\vee Q\vee U_m^*P_0U_m$.  But we could have also chosen $(x_n)$ such that $P_{1/(x_n-s_m)}=P_{\pi/2}$, for all $n$, and this would show that $a(s_m)\leq p(s_m)\vee Q\vee U_m^*P_{\pi/2}U_m$.  But \[(p(s_m)\vee Q\vee U_m^*P_0U_m)\wedge(p(s_m)\vee Q\vee U_m^*P_{\pi/2}U_m)=p(s_m)\vee Q\] which, as $m$ was arbitrary, shows that $\overline{p\vee q}^\circ=p\vee q$.

We next claim that $(p\vee q)\wedge\overline{p}^\perp=\{0\}$, which will verify the non-orthomodularity of $[A]^\perp$.  If not, we would have non-zero $r\in A_+$ with $r\leq p\vee q$ and $r\perp p$.  For some $n$, we would then have $r(s_n)\neq0$ and hence $r(s_n)=\lambda R$, where $R=p(s_n)\vee Q-p(s_n)$ and $\lambda\neq0$.  But $RU_n^*P_0U_n\neq0$ and hence $r(x)p_n(x)\neq0$ for some $x$ sufficiently close to $s_n$.  This contradicts the fact that $rp=0$.
\end{xpl}

\begin{xpl}[$A\in[A{]}^\perp_{\mathbf{D}<\aleph_0}$ that is not $<\!\!\aleph_0$-subhomogeneous]\label{aleph0sub}\hfill\\
Let $(e_n)$ be a basis for a separable Hilbert space $H$.  For each $n\in\mathbb{N}$, identify $M_n$ in the canonical way with $P_n\mathcal{B}(H)P_n$, where $P_n$ is the projection onto $\mathrm{span}\{e_1,\ldots,e_n\}$.  Let $A$ be the C*-subalgebra of $\prod_nM_n$ of sequences $(a_n)$ converging to some $a_\infty\in\mathcal{B}(H)$.  Then each canonical copy of $M_n$ in $A$ will be an $n$-homogeneous annihilator ideal (the annihilator of the kernel of the $n^\mathrm{th}$ coordinate representation in fact) and $(\bigvee_nM_n)^\perp=\{0\}$ so $A\in[A]^\perp_{\mathbf{D}<\aleph_0}$.  But the representation $\pi$ that takes each $(a_n)$ to its limit $a_\infty$ will map $A$ onto $\mathcal{K}(H)$.  Thus $\pi$ is irreducible which, as $H$ is infinite dimensional, means $A$ is not $<\!\!\aleph_0$-subhomogeneous.
\end{xpl}

We have focused on examples of C*-algebras of continuous functions, as these are the most tractable and already provide a good intuitive basis for working with annihilators.  But there are undoubtedly many secrets to be gleaned from looking at the annihilator structure of all the various C*-algebras under investigation in current research.  We make the first tentative step in this direction with the following simple example.

\begin{prp}\label{CAR}
The CAR algebra $M_{2^\infty}$ is (purely) infinite.
\end{prp}

\begin{proof}
For each $n\in\mathbb{N}$, let $(e_{s,t})_{s,t\in\{0,1\}^n}$ be the canonical matrix units of the canonical unital copy of $M_{2^n}$ in $M_{2^\infty}$.  So, for each $m,n\in\mathbb{N}$ and $s,t\in\{0,1\}^n$, we have $e_{s,t}=\sum_{r\in\{0,1\}^m}e_{sr,tr}$, where $sr$ is the $m+n$ length sequence obtained from simply concatenating $s$ and $r$ (and likewise for $tr$).  We also have $e_{q,r}e_{s,t}=0$ whenever $r$ and $s$ differ on their common domain.  For each $n\in\mathbb{N}$, let $\delta_n$ be the sequence of $n$ $0$s followed by a $1$, and recursively define $\sigma_n\in\{0,1\}^{n+1}$ so that $\sigma_j$ and $\sigma_k$ differ on their common domain, for distinct $j$ and $k$, and such that every finite sequence of $0$s and $1$s has some restriction or extension in $(\sigma_n)$.  Setting $a=\sum_{n=1}^\infty2^{-n}e_{\delta_n,\sigma_n}$, we have \[aa^*=\sum_{n=1}^\infty2^{-2n}e_{\delta_n,\delta_n}\qquad\textrm{while}\qquad a^*a=\sum_{n=1}^\infty2^{-2n}e_{\sigma_n,\sigma_n}.\]  Thus, for any $n\in\mathbb{N}$, we can find $m\in\mathbb{N}$ and $(t_r)_{r\in\{0,1\}^n}\subseteq\{0,1\}^m$ so that $e_n=\sum_{r\in\{0,1\}^n}e_{rt_r,rt_r}\leq 2^{2(m+n)}a^*a$ and hence $e_n\in\{a\}^{\perp\perp}$ with $||e_n||=1$.  Also, for any $n\in\mathbb{N}$ and $b=\sum_{r,s\in\{0,1\}^n}\lambda_{r,s}e_{r,s}\in M_{2^n}\subseteq M_{2^\infty}$, we have $||e_nbe_n||=||\sum_{r,s\in\{0,1\}^n}\lambda_{r,s}e_{rt_r,st_s}||=||b||$.  As $\bigcup_nM_{2^n}$ is dense in $M_{2^\infty}$ this means that $\sup_{e\in\{a\}^{\perp\perp},||e||=1}||be||=||b||$ for all $b\in M_{2^\infty}$ so $\{a\}^\perp=\{a\}^{\perp\perp\perp}=\{0\}$, i.e. $\{a\}^{\perp\perp}=M_{2^\infty}$.  On the other hand, we immediately see that $e_{1,1}\in\{a^*\}^\perp$ and so $\{a^*\}^{\perp\perp}\neq M_{2^\infty}$, so $a$ witnesses the equivalence of $M_{2^\infty}$ with a proper subannihilator.

To see that any annihilator $B$ is equivalent to a proper subannihilator just take any projection $p\in B$ and note that it must be unitarily equivalent to some projection in $M_{2^n}$, for some $n\in\mathbb{N}$, which itself will have a subprojection unitarily equivalent to a matrix unit $e\in M_{2^n}$.  But then $\{e\}^{\perp\perp}$ is isomorphic to $M_{2^\infty}$ and so the argument of the previous paragraph applies.
\end{proof}

This example might be a little surprising, given that the CAR algebra, and all other UHF algebras for that matter, are usually thought of as being very finite C*-algebras, primarily due to the fact they possess a (unique bounded) faithful trace $\tau$ which, furthermore, always gives rise to the unique hyperfinite $\mathrm{II}_1$ factor as $\pi_\tau[A]''$ (where $\pi_\tau$ is the representation coming from the GNS construction applied with $\tau$).  This might even lead some operator algebraists to dismiss the annihilators as clearly giving the `wrong' notion of finiteness.  Alternatively, you might blame the relation $\sim$ rather than the annihilators themselves, and if it turned out that the inclusion ordering on the annihilators in the CAR algebra is modular, despite being infinite w.r.t. $\sim$, then it might indeed be the better to focus on perspectivity rather than the $\sim$ relation, at least for some C*-algebras.  However, we would interpret \autoref{CAR} rather as merely showing that the close relationship between traces and projections in von Neumann algebras does not extend to annihilators in C*-algebras.  If the tracial structure is what you are interested in then you are better off looking at the positive elements of $A$ under the equivalence notion given in \cite{CuntzPedersen1979}, as mentioned in \S\ref{Motivation}.

In fact, traces give rise to dimension functions on $A$ but these are quite different from the natural dimension functions you would consider on $[A]^\perp$, even in the commutative case.  To see this, note that dimension functions on $C(X)$ correspond to measures $\mu$ on $X$, while annihilators correspond to regular open subsets $O$ of $X$.  Thus you would naturally think $D(O)=\mu(O)$ defines a dimension function on these regular open sets, but this is not the case.  Consider, for example, $X=[-1,1]$ and $\mu=\delta_0$, the point probability measure at $0$.  Then $\mu([-1,0))=0=\mu((0,1])$ even though $\mu([-1,0)\vee(0,1])=\mu([-1,1])=1$, so this function on regular open sets does not even satisfy the basic dimension function axiom
\begin{equation}\label{dimax}
D(N\vee O)+D(N\wedge O)=D(N)+D(O).
\end{equation}
Even using the Lebesgue measure for $\mu$ instead would not help, as there are Cantor-like nowhere dense subsets of $[-1,1]$ with non-zero Lebesgue measure, and their complements can be expressed as the union of two disjoint regular open sets.  Going up a dimension to $[-1,1]\times[-1,1]$, there even exists a pair of connected disjoint regular open sets whose complement is an Osgood curve (see \cite{Osgood1903}), i.e. a Jordan curve of non-zero Lebesgue measure.

\autoref{CAR} also illustrates the fact that taking direct limits, a common construction with C*-algebras, is likely to produce a C*-algebra that is infinite.  This begs the question of whether there might exist other general constructions that can still produce C*-algebras with the properties we want, but which remain finite.  Indeed, do there exist any separable finite type II C*-algebras at all?  Surely there must, and it is probably just a matter of examining enough examples through the lens of annihilators until one is found.  Although there is the possibility that some new technique might be needed to create one or that somehow separability precludes their existence (for non-separable examples, we can of course just take any type $\mathrm{II}_1$ von Neumann algebra), which would be quite intriguing in either case.

\bibliography{maths}{}
\bibliographystyle{alphaurl}

\end{document}